\documentclass[10pt]{article}
\RequirePackage[colorlinks,citecolor=blue,urlcolor=blue,linkcolor=blue]{hyperref}
\usepackage{url}
\usepackage{booktabs}
\usepackage{svg}
\usepackage{tcolorbox}
\usepackage{subcaption}
\usepackage{url,amscd,paralist,bbm,wrapfig}
\usepackage{bm}
\usepackage{algorithm}
\usepackage{algorithmic} 
\usepackage{todonotes}
\usepackage{amsmath}
\usepackage{amsthm}
\usepackage{amsopn}
\usepackage{amsfonts}
\usepackage{cleveref}
\usepackage[T1]{fontenc}

\usepackage{latexsym,amssymb,amsmath,amsthm,graphics,graphicx,float,psfrag, epsfig, color, authblk, enumerate}

\usepackage{url,amssymb,amsmath,amsthm,amscd,paralist,bbm,wrapfig}
\usepackage{bm}
\usepackage{algorithm}
\usepackage{algorithmic} 

\usepackage{mathtools}

\usepackage[numbers]{natbib}

\newtheorem{remark}{Remark}

\newcommand{\R}{\mathbb{R}}

\newcommand{\Tcal}{{\mathcal{T}}}

\DeclareMathOperator*{\argmin}{arg\,min}

\DeclarePairedDelimiterX{\infdivx}[2]{(}{)}{%
  #1\;\delimsize\|\;#2%
}

\newcommand*{\email}[1]{\texttt{#1}}

\newtheorem{theorem}{Theorem}
\newtheorem{proposition}[theorem]{Proposition}
\newtheorem{lemma}[theorem]{Lemma}
\newtheorem{corollary}[theorem]{Corollary}

\newtheorem{definition}[theorem]{Definition}

\title{The Geometry of Anisotropic Dilation for Optimal Regularization}

\author{
\begin{tabular}{cc}
Carson Newman\thanks{Department of Statistics and Data Science, University of California, Los Angeles (email: \email{carsonnewman@g.ucla.edu})}
\qquad
Oscar Leong\thanks{Institute of Natural Sciences and School of Mathematical Sciences, Shanghai Jiao Tong University (email: \email{oleong@sjtu.edu.cn})}
\end{tabular}
}

\date{\today}

\begin{document}

\maketitle

\begin{abstract}
    
    How can structured transformations of data improve regularizer design for inverse problems? We address this question through \textbf{anisotropic dilation}, a direction-preserving map that rescales each point along its Euclidean ray by a positive profile on the sphere. Our analysis builds on a characterization of optimal regularizers within a broad Gibbs class, where the regularizer best matched to a distribution $P$ is determined by a direction-dependent radial summary statistic $\rho_P$. We show that anisotropic dilation acts explicitly and transitively on the space of radial summary statistics: any prescribed optimal regularizer in this class can be obtained by transforming an admissible source distribution using a single explicit profile. We characterize the resulting distributional orbits and establish conditions under which this control of regularizer geometry yields a weakly convex denoising formulation even when the original objective is not weakly convex. We also derive in closed form the profile that optimally adapts a fixed base regularizer to data. Its improvement over isotropic rescaling with respect to KL divergence is governed by a Jensen gap, for which we establish strict positivity under explicit conditions for several natural families of distributions. For jointly learning the base regularizer and anisotropic profile from samples, we establish finite-sample generalization bounds, with explicit rates when the logarithm of the profile is parameterized by linear feature models or sparse ReLU networks. Experiments on a controlled two-dimensional family and on MNIST denoising demonstrate that learning the anisotropic profile can improve reconstruction performance for both handcrafted and learned base regularizers. 
\end{abstract}

\section{Introduction} \label{sec:intro}

Many applications in science and engineering concern solving inverse problems, which ask to reconstruct data from corrupted measurements. Such problems are often ill-posed in the sense of Hadamard, meaning that there are no solutions, infinitely many solutions, or the problem is
discontinuous in the data \cite{Arridgeetal19}. Variational regularization is a long-standing technique to solve such problems by exploiting a regularization functional to promote certain structures one believes the underlying data to obey. Popular forms of structure include smoothness \cite{knapik2011bayesian}, sparsity \cite{Tao2006, Donoho2006}, low-rankness \cite{FazelThesis, Rechtetal10}, compressibility in a suitable basis \cite{haupt2006signal}, and many others. Due to the ill-posedness of inverse problems, the selection of a suitable regularizer is crucial for obtaining high-quality and stable reconstructions.

A central difficulty is that the geometry of a good regularizer is usually specified indirectly. Classical regularizers, such as Tikhonov, total variation \cite{Rudinetal92}, sparsity-inducing norms \cite{StatLearnSparsityBook, Daubechiesetal04}, and atomic norms \cite{Chandrasekaranetal12a}, are often motivated by domain knowledge and by the desire to obtain computationally tractable optimization problems. While powerful, such hand-crafted regularizers are often limited in the types of structure they can promote. Hence, one may be interested in tailoring a given regularizer to a dataset of interest, whose geometry may not be easy to describe in a hand-crafted way, but may be better designed in a learned fashion. This has led to a suite of learnable, nonconvex regularizers that are substantially more flexible and have shown strong empirical performance in several contemporary inverse-problem pipelines \cite{Lunzetal18, Goujonetal24, venkatakrishnan2013pnp, romano2017RED, Albertietal21, alberti2026learning}. However, such regularizers generally come with fewer guarantees regarding global optimization, stability, and statistical behavior, and the exact geometry they exploit is less well-understood.

These considerations motivate a basic mathematical question: given a data source specified by a probability distribution $P$, what regularizer is best matched to the geometry of $P$? Recent work on optimal regularization \cite{leong2025optimal} takes a first step towards answering this question by considering the class of continuous and positively homogeneous functionals that are positive away from the origin on $\R^d$ -- equivalently, gauges of star bodies $\|\cdot\|_K$, a nonconvex generalization of convex bodies that contains norms, quasinorms, and many nonconvex sparsity-promoting penalties. The optimal regularizer in this class is the one whose induced Gibbs density $\propto \exp(-\|\cdot\|_K)$ is closest to $P$ in Kullback-Leibler divergence. The central finding of \cite{leong2025optimal} is that the solution is governed by a single function on the sphere: a \emph{radial summary statistic} $\rho_P : \mathbb{S}^{d-1} \rightarrow (0,\infty)$, given by a weighted radial moment of the density of $P$ along each direction: \begin{align}
    \rho_P(u) := \left(\int_0^{\infty} r^d p(ru)\mathrm{d}r\right)^{1/(d+1)},\ \quad u \in \mathbb{S}^{d-1}. \label{eq:radial-summary-statistic}
\end{align} The optimal regularizer is, up to normalization, the gauge of the star body with radial function $\rho_P$. Thus, for this family of regularizers, the relevant information in $P$ for optimal regularization can be compressed into a single direction-dependent quantity, which quantifies how much mass of the distribution lies in each direction and how far from the origin that mass typically sits. 

This characterization raises a new set of questions that motivate our work:
\begin{enumerate}
    \item If the optimal regularizer for $P$ is determined by $\rho_P$, how do transformations of the data act on $\rho_P$? Are there natural families of maps $\mathcal{T}$ for which $\rho_{\mathcal{T}_{\#}(P)}$ is explicit in terms of $\rho_P$?
    \item If so, can such a map be used to transform the data so that regularization becomes easier -- for instance, so that the induced variational problem is better behaved -- or, in the other direction, to adapt a given base regularizer to a dataset of interest?
    \item Can such a map increase the flexibility of the Gibbs family in a controlled way to learn better regularizers while also retaining closed-form characterizations?
\end{enumerate}

To address these questions, we introduce \emph{anisotropic dilation}: for a positive continuous function $g$ on the sphere, the map $\Tcal_g$ is defined as $$\mathcal{T}_g(x) := g\left( \frac{x}{\|x\|_2} \right) x,\ \quad x \neq 0.$$ This map rescales each point along its ray according to its direction. Despite its simplicity, we show that this family acts transitively on the space of radial summary statistics, and additionally illustrate its utility as a statistical model when applied to the Gibbs class discussed previously.

\subsection{Our Contributions}

\paragraph{Geometry and Orbit Structure of Anisotropic Dilation} In Section~\ref{sec:anisotropic-dilation}, we characterize the induced Abelian group action on probability distributions and show that this transformation preserves a distribution's directional marginal while rescaling its conditional mean radius. As a consequence, this transformation acts on the radial summary statistic by $$\rho_P \mapsto \rho_{(\Tcal_g)_{\#}P}=g^{1/(d+1)} \rho_P.$$ Consequently, any admissible target radial summary statistic can be reached from a given source by a unique profile $g$, even though the resulting distributional orbits retain scale-invariant radial information. We further show, in a denoising setting, that this flexibility can transform an objective that is not weakly convex near the origin into a weakly convex formulation for sufficiently strong regularization.

\paragraph{Population Estimation of the Optimal Scaling Profile} In Section~\ref{sec:learning-g},
    we demonstrate the utility of these orbits as statistical models for $P$ by deriving exact, closed-form population projections onto orbits of baseline Gibbs densities. Furthermore, we define the isotropy cost to quantify the advantage of anisotropic over isotropic approximations, the latter referring to when we do not learn an optimal anisotropic dilation map. This cost is governed by how much the following diagnostic ratio varies over the sphere: $$\alpha_P(u) = \frac{m_P(u)}{\rho_P(u)},\ \quad u \in \mathbb{S}^{d-1},$$ where $m_P(u) := \mathbb{E}_P[R|U=u]$ is the conditional mean radius of $P$ in a given direction $u$ (see Section \ref{sec:preliminaries} for more details). We evaluate this criterion for several distribution classes, showing, among other results, that centered Gaussian mixtures have zero cost if and only if all component covariance matrices are proportional. 
\paragraph{Extensions and a Transport Interpretation} Further in Section \ref{sec:learning-g}, we extend the previous projection theory to Gibbs energies of arbitrary positive homogeneity and to a higher-order conditional radial exponential family, obtaining explicit optimal scaling profiles and a moment-matching characterization of the residual radial shape. We also formulate anisotropic dilation as a raywise optimal transport problem: for distributions with common directional marginal, we identify $\Tcal_g$ as the optimal ray-preserving map when the conditional radial laws differ by scaling.

\paragraph{Learning, Computation, and Finite-Sample Guarantees} In Section \ref{sec:denoising-experiments}, we parameterize the base regularizer and anisotropic profile to directly learn them from samples from a data distribution. On a controlled two-dimensional problem and on MNIST denoising, learning $g$ improves both a learned base regularizer and a fixed hand-crafted one, namely total variation \cite{Rudinetal92}. We complement this computational framework in Section \ref{sec:statistical-learning} with sample complexity estimates and generalization bounds for learning such scaling profiles from data.

\subsection{Related Work}

\paragraph{Transforms of Data} Transformations of data have long played a central role in inverse problems and statistical estimation. Classical examples include Fourier, wavelet \cite{daubechies1992ten}, discrete cosine transforms \cite{strang1999discrete}, and scattering transforms \cite{mallat2012group}, where signals are assumed to be sparse or compressible after applying a fixed change of representation. This viewpoint underlies many standard regularizers, such as transformed $\ell_1$-penalties, total variation, and others \cite{TikReg, Rudinetal92, daubechies1992ten, mallat2012group}. A related data-adaptive approach is dictionary learning or sparse coding, where the representation itself is learned from examples and subsequently used as a regularizer by promoting sparse coefficients in the learned basis or frame \cite{OlshausenField96, Aharonetal2006, MairaletalSurvey14}. These methods can be interpreted as transforming data into a coordinate system in which a simple geometric prior becomes appropriate.

More recently, invertible transformations have become a central tool in generative modeling and inverse problems. Normalizing flows construct expressive probability distributions by pushing forward a simple base distribution through a sequence of bijective maps, while retaining tractable density evaluation through a change-of-variables \cite{papamakarios2021normalizing}. Such models have also been used as priors for inverse problems, where one either optimizes in a latent space or penalizes reconstructions according to the likelihood induced by the invertible model \cite{asim2020invertible, whang2021solving}. Related work in flow matching \cite{lipmanflow, liu2022flow, liu2022rectified} instead learns the velocity field mapping a simple base distribution to a more complicated data distribution, which induces an invertible flow map that can be used as a prior for inverse problems as well \cite{ben2024d, zhang2024flow, martin2025pnp}.

\paragraph{Transport Maps}

Transport maps provide a general framework for representing a complicated distribution as the pushforward of a simple reference distribution. In Bayesian computation, this idea has been used to construct maps from a prior or reference measure to a posterior measure, thereby replacing or accelerating sampling. For example, optimal-map approaches to Bayesian inference construct deterministic maps that push the prior to the posterior and then use these maps to compute posterior moments or generate independent posterior samples. Related transport-map MCMC methods learn triangular or Knothe-Rosenblatt-type maps to learn non-Gaussian proposals that better explore the target distribution \cite{el2012bayesian, parno2018transport, marzouk2017sampling}. 

A common theme in this literature is the tension between expressivity and tractability. Fully general transport maps are highly flexible, but difficult to estimate and interpret in high dimension. This has motivated structured families of maps, including triangular maps \cite{baptista2024representation}, low-dimensional or subspace-constrained transports \cite{muzellec2019subspace}, conditional transports \cite{chemseddine2025conditional}, and radial transports \cite{ghafourpour2026variational}. The aforementioned radial transport map is closest to our current work. In particular, \cite{ghafourpour2026variational} propose radVI, a variational inference method that optimizes over radial profiles in order to improve Gaussian, mean-field Gaussian, or Laplace approximations. Their motivation is that Gaussian variational families often fail to capture the radial profile of the target posterior, leading to poor coverage. They propose to optimize over the class of radially symmetric distributions and represent such distributions as pushforwards of a standard Gaussian by radial maps of the form $T_{\mathrm{rad}}(x) := \Psi(\|x\|_2) \cdot \frac{x}{\|x\|_2}.$ This map is complementary to our anisotropic dilation map. Note that $T_{\mathrm{rad}}$ is nonlinear in the radius but isotropic over directions as every direction is transformed by the same scalar function $\Psi(r)$. By contrast, our map is anisotropic across directions, while preserving the angular distribution. We finally note that the motivations for our maps are different, as we consider the question of mapping between optimal regularizers whereas radVI is focused on variational inference.

\paragraph{Learned Regularization}

A complementary line of work aims to learn regularizers directly from data. As mentioned previously, dictionary learning and sparse coding provide some of the earliest examples of this paradigm, where one learns a representation in which the data have low-complexity descriptions. With the advent of deep learning, there is also a growing body of work parameterizing the regularizer directly with a deep neural network. One prominent example is adversarial regularization \cite{Lunzetal18}, where a neural network is trained as a critic to distinguish between desirable reconstructions and undesirable ones, and the resulting critic is used as a penalty in a variational reconstruction problem. As such regularizers are generally highly nonconvex, related approaches have also imposed structural constraints such as convexity \cite{Mukherjeeetal2021, Goujonetal23}, weak convexity \cite{Goujonetal24, Shumaylovetal2024}, and difference-of-convex structure \cite{zhang2025learning}. See the survey article \cite{hertrich2025learning} for a deeper discussion of such methods along with empirical comparisons between them.

\paragraph{Optimal Regularization}

Our work is most directly motivated by recent aims to characterize which regularizer is optimal for a given signal model or data distribution. Some works have focused on learning optimal regularizers within a given restricted class of functionals, such as those given by Tikhonov regularizers \cite{Albertietal21} or nonsmooth transformed $\ell_1$-penalties \cite{alberti2026learning}, while others have focused on learning of the penalty parameter from data \cite{ chirinos2024learning, kunisch2013bilevel, calatroni2017bilevel, ehrhardt2023optimal}. Another line of work studies optimal convex regularization for low-dimensional recovery \cite{Traonmilinetal-21}. In this setting, one fixes a signal model and a class of linear inverse problems, and asks which convex regularizer promotes recovery according to a suitable compliance measure. Such results recover canonical examples, such as $\ell_1$-regularization for sparse recovery, and have been extended to understanding the optimal algorithms that promote such structure \cite{traonmilin2024towards}.

A closer line of work studies optimal regularization at the level of an entire data distribution \cite{leong2025optimal}, as discussed in Section \ref{sec:intro}. This theory characterizes for a given data distribution, what is the optimal positively homogeneous and continuous gauge associated to such a distribution. It has been extended to settings when the objective is given by adversarial regularization \cite{leong2024star} and also when one aims to incorporate distributional robustness and convexity constraints \cite{leong2025optimal2}. The present work builds on this viewpoint, but asks a different question. In particular, instead of asking what the optimal regularizer for a given distribution is, we ask how structured transformations of the data act on the statistic that determines the optimal regularizer. The anisotropic dilation map $\Tcal_g$ that we introduce is designed so that the action on this summary statistic is explicit, and we show that this map offers other benefits in fitting distributions while maintaining interpretability.

\section{Preliminaries}
\label{sec:preliminaries}
\paragraph{Star and Convex Geometry} Our work will use many concepts arising in star geometry and convex geometry. For a deeper overview of such topics, we refer the reader to \cite{hansen2020starshaped, gardner1995geometric} for star geometry and \cite{schneider2013convex} for convex geometry.

For $d \geq 1$, let $\mathbb{S}^{d-1}$ and $B^d$ denote the unit Euclidean sphere and closed ball in $\R^d$, respectively. For $x \in \R^d$ and $r > 0$, let $B(x,r) := \{ y \in\R^d : \|y-x\|_2 \leq r\}$ where $\|\cdot\|_2$ denotes the Euclidean norm. A set $K\subset \R^d$ is a convex body if it is compact, convex, with the origin in its interior. We say that $K \subset \R^d$ is a star body if it is compact, contains the origin in its interior, and for every $x \neq 0$, the ray $\{\lambda x : \lambda > 0\}$ intersects the boundary of $K$ exactly once. Equivalently, $K$ is a star body if and only if it is compact, star-shaped with respect to the origin, and its \emph{radial function},
$$
\rho_K(u) := \sup\{\lambda \geq 0 : \lambda u \in K\},
$$
is positive and continuous on $\mathbb{S}^{d-1}$. Here, star-shaped with respect to the origin means that $tx \in K$ for every $x \in K$ and $t \in [0,1]$. We represent any star body $K$ throughout this paper as
$$
K = \{ ru : u \in \mathbb S^{d-1}, \, 0 \le r \le \rho_K(u) \}.
$$
The set of all star bodies in $\R^d$ is denoted by $\mathcal{S}^d$. There is a one-to-one correspondence between star bodies in $\mathcal{S}^d$ and positive continuous functions on $\mathbb{S}^{d-1}$ via the radial function $\rho_K$.

 Another geometric function that is closely related to regularization is the gauge (or Minkowski functional) induced by a star body $K$, defined as $\|x\|_K := \inf\{ \lambda \geq 0 : x \in \lambda K\}.$ For $u \in \mathbb{S}^{d-1}$, the radial function and gauge satisfy the reciprocal relationship
 $\rho_K(u) = 1/\|u\|_K.$ Finally, let $\mathrm{vol}_d(\cdot)$ denote the usual $d$-dimensional volume functional.

\paragraph{Polar Decomposition of Distributions and Summary Statistics} Since it is natural to understand the anisotropic dilation map $\mathcal{T}_g$ through its action on the radial and directional components of a given input, we will use several notions concerning polar decompositions of random vectors. Let $P$ be a probability measure on $\mathbb{R}^d$ satisfying $P(\{0\})=0$. By the Disintegration Theorem \cite[Theorem 5.4]{kallenberg1997foundations}, the polar map $x \mapsto (\|x\|_2, x/\|x\|_2)$ induces a decomposition of $P$ into a directional marginal distribution $P_U$ on $\mathbb{S}^{d-1}$ and a family of conditional radial distributions $P_{R\mid U}(\cdot \mid u)$ on $(0,\infty)$, defined $P_U$-almost everywhere. Equivalently, any random vector $X\sim P$ admits a natural polar decomposition $X=R_PU_P$ almost surely, where $U_P=X/\|X\|_2\sim P_U$ and, conditionally on $U_P=u$, the radius $R_P=\|X\|_2$ has law $P_{R\mid U}(\cdot\mid u)$. When $P$ is clear from context, we drop the subscripts and write $X=RU$.

Assume now that $P$ admits a Lebesgue density $p$. Passing to polar coordinates via
 $\mathrm{d}x = r^{d-1}\mathrm{d}r\,\mathrm{d}\sigma(u)$ yields the joint density $p_{R,U}(r,u) = r^{d-1}p(ru)$ of $(R_P, U_P)$ with respect to $\mathrm{d}r\,\mathrm{d}\sigma(u)$, where $\sigma$ denotes the standard surface measure on $\mathbb{S}^{d-1}$. Integrating out the radial component defines the \emph{directional marginal density} of $U_P$: 
\begin{align}
    h_P(u) := \int_0^\infty r^{d-1} p(ru) \, \mathrm{d}r. \label{eq:directional-marginal}
\end{align}
In particular, $P_U(\mathrm du)=h_P(u)\,\mathrm d\sigma(u)$. The directional marginal $h_P$ and the radial summary statistic $\rho_P$ from \eqref{eq:radial-summary-statistic} are foundational to developing the theory in this paper. To ensure these quantities are pointwise well-defined and non-degenerate on the entire sphere, we impose the following standing assumption.

\begin{definition}[Polar Regular Distributions] \label{def:polar-regular} Let $G := C\bigl(\mathbb S^{d-1}, (0, \infty)\bigr)$
denote the space of continuous, strictly positive functions on the unit sphere, equipped with pointwise multiplication.
We define the class of polar regular distributions by 
$$
\mathcal{P} 
:=
\left\{
P \text{ a probability measure on } \R^d :
P \ll \mathcal{L}^d\ \text{with density}\ p\ \text{such that}\
\rho_P,h_P \in G
\right\},
$$
where $\mathcal{L}^d$ denotes Lebesgue measure on $\R^d$, $\rho_P$ is the radial summary statistic defined in \eqref{eq:radial-summary-statistic}, and $h_P$ is the directional marginal density defined in \eqref{eq:directional-marginal}.
\end{definition}
By the compactness of $\mathbb S^{d-1}$, every $P\in\mathcal P$
has $h_P$ and $\rho_P$ bounded above and bounded away from zero.
Furthermore, these assumptions are mild; $\mathcal P$ includes,
for example, continuous strictly positive densities with
sufficiently rapid uniform decay, as well as continuous,
compactly supported densities that are strictly positive on
a neighborhood of the origin.

For any $P \in \mathcal{P}$, the strict positivity of $h_P(u)$ guarantees the conditional radial density $p_{R\mid U}(r \mid u) = r^{d-1}p(ru)/h_P(u)$ is well-defined for every $u\in\mathbb{S}^{d-1}$. We define the \emph{conditional mean radius} in direction $u$ by
\begin{align}
    m_P(u) := \mathbb{E}[R_P \mid U_P = u]   \label{eq:conditional-radial}
\end{align}
This quantity represents the first moment of the conditional radial distribution $P_{R \mid U_P =u}.$ Expanding the definition of $m_P(u)$ and rearranging terms, we obtain the following fundamental decomposition of the radial summary statistic:
\begin{align}
    \rho_P(u)^{d+1} = h_P(u)m_P(u). \label{eq:rho-decomposition}
\end{align}
This identity demonstrates that the radial summary statistic governing the optimal Gibbs regularizer is determined by two interpretable quantities: the directional mass $h_P(u)$ and the conditional mean radius $m_P(u)$. This explicit decoupling plays a central role in characterizing the anisotropic dilation map in subsequent sections.

\section{Anisotropic Dilation}\label{sec:anisotropic-dilation}

\subsection{Definition and Fundamental Properties}

We first define the anisotropic dilation map and discuss several basic properties. 
\begin{definition}[Anisotropic Dilation] 
    For $g \in G$, the anisotropic dilation induced by $g$ is the map
    \begin{align*}
        \Tcal_g(x):= 
        \begin{cases}
            g \left( \dfrac{x}{\|x\|_2} \right)x, & x \neq 0, \\
            0, & x = 0.
        \end{cases}
    \end{align*}
    Throughout the text, we refer to $g$ as the positive and continuous scaling profile associated with $\mathcal T_g$.
\end{definition} 
\noindent

\paragraph{Basic Algebraic and Distributional Properties}
Letting $x = ru$, where $r = \|x\|_2 $ and $u = x/\|x\|_2$ for $x \neq 0$, the transformation takes the simple polar form
\begin{align*}
    \Tcal_g : (r,u) \mapsto ( g(u) r,u), \qquad r > 0, \quad u \in \mathbb S^{d-1}.
\end{align*} 
Thus, $\mathcal{T}_g$ preserves the direction of each point and rescales only its radial coordinate, with the scaling factor varying continuously across directions. The strict positivity of $g$ ensures that no ray is reversed or collapsed. In fact, the set $\{ \mathcal \mathcal{T}_g : g \in G\}$ forms an Abelian group of homeomorphisms of $\mathbb{R}^d$, $$\mathrm{Homeo}(\R^d) := \{f :\R^d\rightarrow \R^d : f\ \text{continuous bijection with}\ f^{-1}\ \text{continuous}\}.$$

\begin{proposition}[Abelian Group Structure of Anisotropic Dilations]
    The map $g \mapsto \mathcal{T}_g$ is a group isomorphism from $G$ onto an Abelian subgroup of $\operatorname{Homeo}(\mathbb{R}^d)$. Explicitly, for every $g, h \in G$, $\mathcal{T}_g \circ \mathcal{T}_h = \mathcal{T}_{gh}$, where $(gh)(u) := g(u) h(u)$. The identity element is $\mathcal{T}_e = \operatorname{id}_{\mathbb{R}^d}$ with $e(u) \equiv 1$, and the inverse is $(\mathcal{T}_g)^{-1} = \mathcal{T}_{1/g}$.
\end{proposition}
\begin{proof}
    The homomorphism property $\mathcal{T}_g \circ \mathcal{T}_h = \mathcal{T}_{gh}$, invertibility $\mathcal{T}_g \circ \mathcal{T}_{1/g} = \mathcal{T}_e$, and commutativity $\mathcal{T}_{gh} = \mathcal T_{hg}$ follow directly from the definition of $\mathcal{T}_g$ and pointwise multiplication. Injectivity is immediate: if $\mathcal{T}_g(x) = \mathcal{T}_h(x)$ for all $x$, evaluating at any $u \in \mathbb{S}^{d-1}$ yields $g(u)u = h(u)u$, forcing $g = h$. Thus, $g \mapsto \mathcal{T}_g$ is an isomorphism onto its image. 

    We verify that the image lies in $\operatorname{Homeo}(\mathbb R^d)$. Fix $g \in G$. The commutativity and invertibility imply that $\mathcal T_{1/g}$ is a two-sided inverse of $\mathcal \mathcal{T}_g$, making $\mathcal \mathcal{T}_g$ a bijection. Furthermore, $g \in G$ forces continuity of $\mathcal \mathcal{T}_g$ on $\mathbb R^d \setminus \{0\}$. It remains to check continuity at the origin. Because $\mathbb{S}^{d-1}$ is compact, $g$ attains its maximum $M_g := \max_{u\in\mathbb{S}^{d-1}}g(u) \in (0,\infty)$. For any $x \neq 0$, $\|\mathcal{T}_g(x)\|_2 = g(x/\|x\|_2)\|x\|_2 \leq M_g\|x\|_2$. As $x \to 0$, $\|\mathcal{T}_g(x)\|_2 \to 0 = \mathcal{T}_g(0)$, confirming continuity on all of $\mathbb{R}^d$. Since $g$ is strictly positive and continuous, $1/g \in G$, and thus an identical bounding argument with $\max_{u \in \mathbb{S}^{d-1}} 1/g(u)=1/\min_{u\in \mathbb{S}^{d-1}} g(u) \in (0,\infty)$ by compactness of $\mathbb{S}^{d-1}$ establishes the continuity of the inverse $\mathcal{T}_{1/g}$ at the origin, confirming that $\mathcal{T}_g \in \operatorname{Homeo}(\mathbb{R}^d)$.
\end{proof}
The pointwise action of $\mathcal{T}_g$ can be visualized through its effect on star bodies. In particular, suppose $\rho_K$ and $\rho_L$ are the radial functions of star bodies $K$ and $L$, respectively. Then, the scaling profile
\begin{align}
g_{K\to L}(u) = \frac{\rho_L(u)}{\rho_K(u)}, \qquad u\in\mathbb S^{d-1} \label{eq:starbodymatching}
\end{align}
belongs to $G$ and satisfies $\mathcal{T}_{g_{K \to L}}(K) = L$. To see this succinctly, note that the origin is fixed. For $x = ru \in K \setminus \{0\}$, one has $0 < r \leq \rho_K(u)$. Applying $\mathcal T_{g_{K \to L}}(x)$ yields $su \in \mathbb R^d \setminus \{0\}$ where $s := g_{K \to L}(u)r$ and $0 < s \leq \rho_L(u)$. Therefore, $\mathcal T_{g_{K \to L}}(K) \subseteq L$. Applying the same argument to the inverse profile $g_{L \to K} = 1/g_{K \to L}$ gives the reverse inclusion. Therefore, $\mathcal{T}_{g_{K\to L}}$ reshapes $K$ into $L$ via a continuous raywise deformation, as illustrated in Figure~\ref{fig:star-body-matching}.
\begin{figure}[t]
    \centering
    \includegraphics[width=0.8\linewidth]{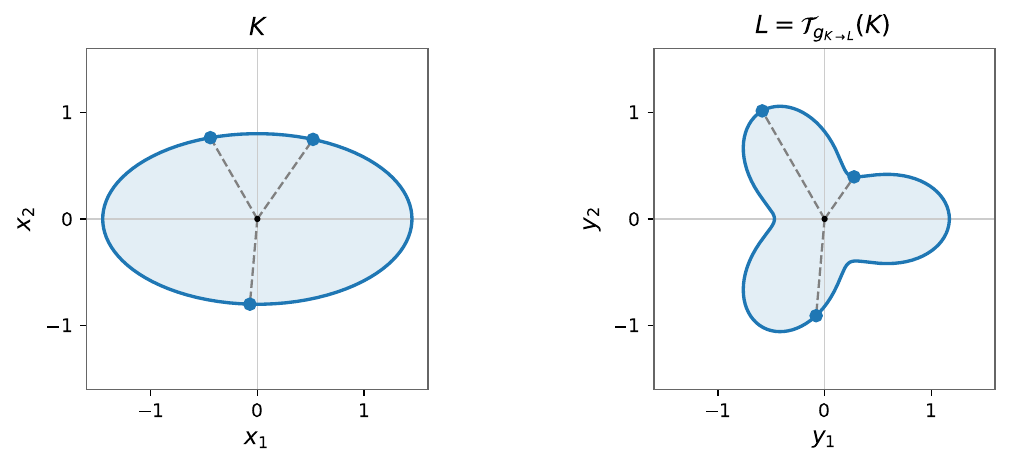}
    \caption{2D example of matching star bodies via \eqref{eq:starbodymatching}. Here $\rho_K(u) = \left( \cos^2(\theta)/1.45^2 + \sin^2(\theta)/0.8^2 \right)^{-1/2}$ and $\rho_L(u) = 0.82 + 0.35 \cos( 3 \theta)$ where $u = ( \cos \theta, \sin \theta)^\top$.}
    \label{fig:star-body-matching}
\end{figure}
 
Our primary analysis considers how the group of transformations $\{\mathcal{T}_g : g \in G\}$ acts on probability distributions. The pushforward operation defined on Borel sets as $((\mathcal{T}_g)_\# P)(A) = P(\mathcal{T}_g^{-1} (A))$ for $A \in \mathcal B(\mathbb R^d)$ defines a natural group action on the space of Borel probability measures $\mathcal{M}_1(\mathbb{R}^d)$. We first define the object containing all possible pushforwards of a base distribution $P$.
\begin{definition}[Anisotropic Radial Orbit]
Let $P \in \mathcal{P} \subset \mathcal M_1(\mathbb R^d)$ be a probability distribution. The anisotropic radial orbit of $P$ is the set of all pushforward distributions generated by the Abelian group of anisotropic dilations, denoted by
$$ 
[\,P\,]_{\mathcal{T}} := \{ (\mathcal{T}_g)_\# P \mid g \in G \}.  
$$
\end{definition}
Recalling the polar decomposition $X = R_P U_P \sim P$, we observe that the anisotropic dilation map acts on the distribution $P$ exactly as it does pointwise: it scales the radial distribution $P_{R \mid U}$ while leaving the directional distribution $P_U$ intact. Therefore, the anisotropic radial orbit of $P$ consists of distributions that share the directional structure of $P$, with radial information related via $g$. The following theorem formalizes this behavior.
\begin{theorem}[Structure of Anisotropic Radial Orbits] \label{thm:structure-of-orbit}
    Let $P \in \mathcal P$ have Lebesgue density $p$. Let $X \sim P$ be a random vector with polar decomposition $X = R_P U_P$. For some fixed $g \in G$, define $Y := \mathcal{T}_g(X) \sim Q := (\mathcal{T}_g)_\# P$. Then, the polar decomposition $Y = R_Q U_Q$ satisfies
    \begin{align*}
       U_Q = U_P \quad \text{and} \quad R_Q = g(U_P)R_P \quad \text{almost surely.} 
    \end{align*}
    Furthermore, $Q$ admits the Lebesgue density
    \begin{align}
    q(x) = p\left(\frac{x}{g(x /\|x\|_2)} \right) g(x /\|x\|_2)^{-d}, \qquad x \neq 0. \label{eq:pushforward-change-of-variables}
\end{align}
\end{theorem}
\begin{proof}
    Since $X \neq 0$ almost surely, the polar decomposition is well-defined. By definition of the pushforward, $Y = \mathcal{T}_g(X) = g(U_P) R_P U_P$ almost surely. Because $g$ is strictly positive and $\|U_P\|_2 = 1$, taking the Euclidean norm yields $R_Q = \|Y\|_2 = g(U_P) R_P$. Normalizing $Y$ then gives $U_Q = Y/\|Y\|_2 = U_P$.
    
    To establish the form of the density $q(x)$, fix an arbitrary profile $g \in G$ and consider any nonnegative Borel measurable function $\psi: \mathbb{R}^d \to [0, \infty]$. By the definition of the pushforward $Q = (\mathcal{T}_g)_{\#}(P)$, we have
    \begin{align}
        \int_{\R^d} \psi(x)\mathrm{d}Q(x) = \int_{\R^d}\psi(\Tcal_g(z))\mathrm{d}P(z) & = \int_{\R^d} \psi(\Tcal_g(z))p(z)\mathrm{d}z \nonumber \\
        & = \int_{\R^d}\psi(g(z/\|z\|_2)z)p(z)\mathrm{d}z  \label{eq:Tg-def-CoV}\\
        & = \int_{\mathbb{S}^{d-1}}\int_0^{\infty}\psi(g(u)ru)p(ru)r^{d-1}\mathrm{d}r\mathrm{d}\sigma(u) \label{eq:polar-coords}\\
        & = \int_{\mathbb{S}^{d-1}}\int_0^{\infty} \psi(su) p\left( s /g(u) \cdot u\right)g(u)^{-d}s^{d-1}\mathrm{d}s \mathrm{d}\sigma(u) \label{eq:change-of-vars}\\
        & = \int_{\R^d} \psi(x) q(x) \mathrm{d}x \nonumber
    \end{align} where we used the definition of $\mathcal{T}_g$ in \eqref{eq:Tg-def-CoV}, converted to polar coordinates in \eqref{eq:polar-coords}, applied the change-of-variables $s = g(u)r$ with $\mathrm{d}r= \mathrm{d}s/g(u)$ in \eqref{eq:change-of-vars}, and converted back to standard coordinates while recognizing the definition of $q(x)$ from \eqref{eq:pushforward-change-of-variables} in the final equality.
\end{proof}
\begin{figure}[b!]
    \centering
    \includegraphics[width=1\linewidth]{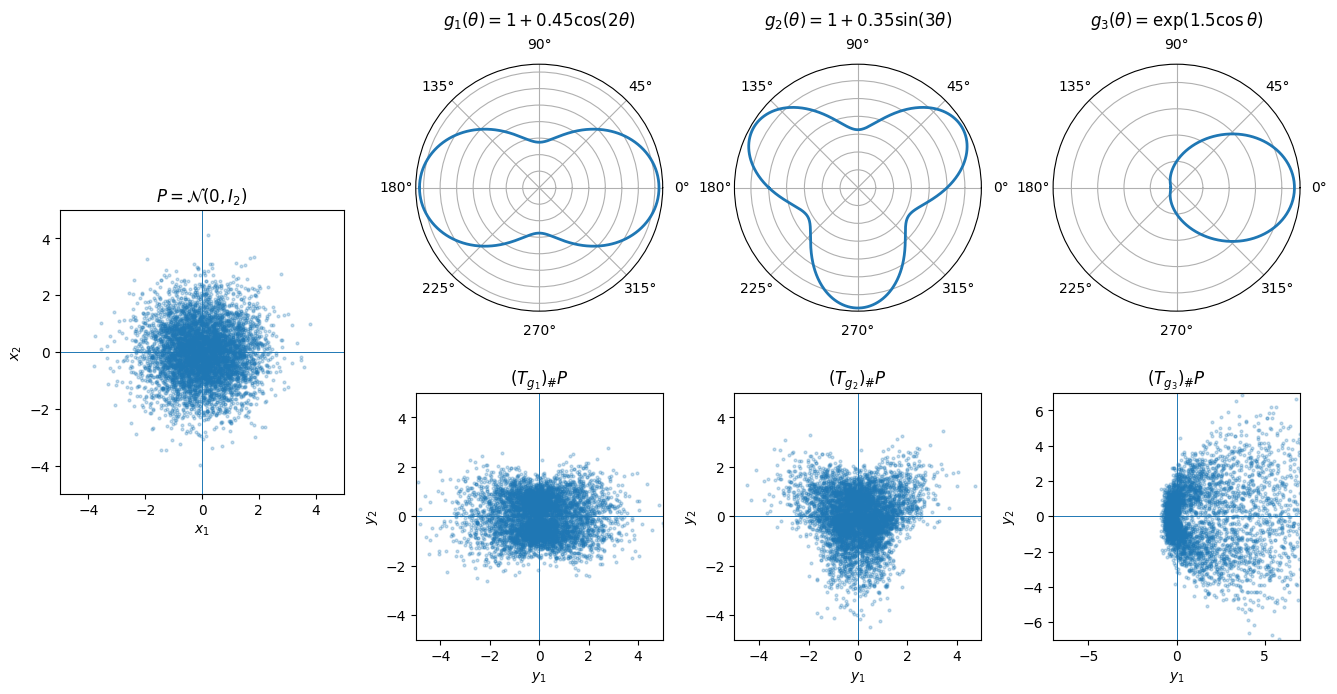}
    \caption{Given the base distribution $P = \mathcal{N}(0,I_2)$, we visualize the samples generated by $(\Tcal_{g_i})_{\#}P$ for different scaling profiles $g_1,g_2,g_3$ where $g_1(u):=1 + 0.45 \cos (2\theta)$, $g_2(u) := 1+0.35\sin(3\theta)$ and $g_3(u):= \exp(1.5\cos\theta)$ with $u = (\cos \theta, \sin \theta)^T$. In each column, we visualize how $g_i$ distorts the boundary of the unit sphere (top) and samples from the induced distribution $(\Tcal_{g_i})_{\#}P$ (bottom).
    }   
    \label{fig:examples-of-Tg-gaussian}
\end{figure}
 This theorem demonstrates that anisotropic dilations fix the directional distribution $P_U$ while flexibly rescaling the conditional radial distribution along each ray via $g$. As  illustrated in Figure~\ref{fig:examples-of-Tg-gaussian}, different scaling profiles $g$ can significantly reshape the radial geometry of an unstructured prior, such as an isotropic Gaussian.

\paragraph{Transformation of Polar Quantities} We next translate this raywise description into transformation laws for the polar quantities introduced in Section~\ref{sec:preliminaries}. We show that since the anisotropic dilation map $\mathcal{T}_g$ leaves the directional distribution $U\sim P_U$ unchanged, the directional marginal $h_P$ is invariant along the orbit. Moreover, the conditional mean radius scales in a direction-dependent fashion $m_{(\mathcal{T}_g)_{\#}P}(u)=g(u)m_P(u)$ since the conditional radius changes via $R \mapsto g(U) R$. One can combine these two identities to yield an explicit transformation law for the radial summary statistic $\rho_{(\mathcal{T}_g)_\# P}$. The following corollary records these relations and, in particular, shows that the class $\mathcal{P}$ is closed under anisotropic dilation.

\begin{corollary}[Transformation of Polar Quantities under Anisotropic Dilation] \label{cor:pushforward-rho}
    Let $P \in \mathcal{P}$ have directional marginal $h_P$, conditional mean radius $m_P$, and radial summary statistic $\rho_P$. Then for any $g \in G$, the pushforward $Q = (\mathcal{T}_g)_{\#}P \in [\, P\,]_{\mathcal T}$ satisfies
    \begin{align}
        h_Q(u) = h_P(u), \qquad m_Q(u) = g(u)m_P(u). \label{eq:pushforward-h-m}
    \end{align}
    Consequently, 
    the radial summary statistic of $Q$ satisfies 
    \begin{align}
         \rho_Q(u) = g(u)^{1/(d+1)}\rho_P(u). \label{eq:pushforward-rho} 
    \end{align}
    Thus, $Q \in \mathcal{P}$, meaning that pushforward by the Abelian group $\{ \mathcal{T}_g : g \in G \}$ defines a group action on $\mathcal{P}$.
\end{corollary}
\begin{proof}
    Let $x = ru$ with $r = \|x \|_2$ and $u = x/\|x\|_2$. Using \eqref{eq:pushforward-change-of-variables}, we compute the directional marginal $h_Q(u)$ for any $u \in \mathbb{S}^{d-1}$:
    \begin{align}
        h_Q(u) = \int_0^\infty q(ru) r^{d-1} \mathrm{d}r &= \int_0^\infty p\left(\frac{r}{g(u)}u\right) g(u)^{-d} r^{d-1} \mathrm{d}r \nonumber \\
        &= \int_0^\infty p(su) g(u)^{-d} (g(u)s)^{d-1} g(u) \mathrm{d}s \label{eq:change-of-vars-2}\\
        & = \int_0^\infty p(su) s^{d-1} \mathrm{d}s  = h_P(u), \nonumber
    \end{align}
    where in \eqref{eq:change-of-vars-2} we substituted $s = r/g(u)$, hence $\mathrm{d}r = g(u)\mathrm{d}s$. The conditional radial mean $m_Q(u)$ uses the exact same change of variables, giving
    \begin{align*}
        m_Q(u) &= \frac{1}{h_Q(u)} \int_0^\infty r q(ru) r^{d-1} \mathrm{d}r = \frac{1}{h_P(u)} \int_0^\infty g(u)s \, p(su) s^{d-1} \mathrm{d}s =   g(u)m_P(u).
    \end{align*}
    Finally, applying the polar decomposition of radial summary statistics \eqref{eq:rho-decomposition}, we have
    \begin{align*}
        \rho_Q(u) = \left( h_Q(u) m_Q(u) \right)^{1/(d+1)} = \left( h_P(u) g(u) m_P(u) \right)^{1/(d+1)} = g(u)^{1/(d+1)} \rho_P(u).
    \end{align*}
    Because $P \in \mathcal{P}$ and $g \in G$, $h_P$ and $\rho_P$ are strictly positive and continuous, ensuring $h_Q$ and $\rho_Q$ are as well. Thus, $Q \in \mathcal{P}$, confirming the group action. 
\end{proof}
Corollary~\ref{cor:pushforward-rho} reveals a direct algebraic relationship between a chosen scaling profile $g \in G$ and the resulting summary statistic $\rho_Q \in G$ via $\rho_{(\mathcal{T}_g)_\#P}=g^{1/(d+1)}\rho_P$. A simple, but significant consequence of this result is that one can always choose the profile $g_*$ so that the distribution $(T_{g_*})_{\#}P$ has a desired radial summary statistic $\rho_*$. In particular, equation \eqref{eq:pushforward-rho} immediately gives the following Corollary:
\begin{corollary}[Matching Radial Summary Statistics] \label{cor:matching-radial-stat}
    Fix $P \in \mathcal{P}$. For any target $\rho_* \in G$, there is a unique $g_*\in G$ that satisfies $\rho_{(\mathcal{T}_{g_*})_\# P} = \rho_*$, which is given by \begin{align}
        g_*(u) = \left(\frac{\rho_*(u)}{\rho_P(u)}\right)^{d+1}. \label{eq:opt-g-star}
    \end{align}
\end{corollary}
Corollary \ref{cor:matching-radial-stat} highlights a significant benefit of the anisotropic dilation map. In particular, we showed that the radial summary statistic identifies the optimal regularizer for a given data distribution. Through the action of the map $\mathcal{T}_g$ on $P$, one can design a scaling profile $g_*$ in order to map a given distribution $P$ to a new distribution $Q = (\mathcal{T}_{g_*})_\#P$ with a desired radial summary statistic and, hence, desired optimal regularizer.  

 \paragraph{Structural Rigidity of the Orbit} While the previous discussion highlighted how the anisotropic dilation map is flexible on the space of radial summary statistics, we show here that there is still rigidity of the map at the level of distributions. This can be seen by Theorem \ref{thm:structure-of-orbit}, which showed how certain properties (e.g., its directional distribution) remain fixed. A consequence of this is that a given orbit need not contain every distribution in $\mathcal{P}$. In particular, we demonstrate that the group $\{ \mathcal{T}_g : g \in G\}$ does not act transitively on $P$ by exhibiting two distributions that belong to distinct orbits. Indeed, conditional on each direction $u$, anisotropic dilation changes the radial distribution only through multiplication by the positive scalar $g(u)$. Consequently, scale-invariant features of these conditional distributions, including quantile ratios and standardized moments, remain fixed across the orbit \cite[Section 3.5]{casella2024statistical}. Tail exponents are likewise preserved under this scaling. The following example illustrates the latter point, which is that a distribution with a slower tail decay cannot be reached starting from a distribution with a faster tail decay.

\begin{proposition}[Non-transitivity on the Space of Distributions] \label{cor:normal-gibbs}
     Let $P = \mathcal{N}(0, I_d)$ be the standard multivariate normal distribution, and let $Q$ be a Gibbs distribution with density $q(x) \propto \exp(-\|x\|_K)$ characterized 
     by a star body $K$. Then $Q \notin [\,P\,]_{\mathcal{T}}$.
\end{proposition}

\begin{proof}
    Let $P = \mathcal{N}(0, I_d)$ with density $p(x) = (2\pi)^{-d/2} \exp(-\|x\|_2^2/2)$, and let $Q$ have the Gibbs density $q(x) = Z_K^{-1} \exp(-\|x\|_K)$. Suppose for contradiction that there exists a scaling profile $g \in G$ such that $Q = \widetilde{Q} := (\mathcal{T}_g)_\# P$. Let $\widetilde{q}$ denote the density of the pushforward $\widetilde{Q}$. Because $g \in G$ is strictly positive and continuous, \eqref{eq:pushforward-change-of-variables} implies that $\widetilde{q}$ is continuous on $\mathbb{R}^d \setminus \{0\}$. The target Gibbs density $q$ is also continuous on this domain. Since the induced probability measures are equal, their continuous densities must be identically equal everywhere on $\mathbb{R}^d \setminus \{0\}$. 
    
    Therefore, we can fix an arbitrary direction $u \in \mathbb{S}^{d-1}$ and evaluate the densities along the ray $x = ru$ for all $r > 0$. Applying \eqref{eq:pushforward-change-of-variables}, the log-density of the pushforward evaluated along this ray is 
    $$
        \log \widetilde{q}(ru) = \log p\left(\frac{ru}{g(u)}\right) - d \log g(u) = - \frac{r^2}{2g(u)^2} - d \log g(u) + C_1,
    $$
    where $C_1 = -\frac{d}{2}\log(2\pi)$. Conversely, the log-density of the target Gibbs distribution is
    $$
        \log q(ru) = -r\|u\|_K + C_2,
    $$
    where $C_2 = -\log Z_K$. For these distributions to coincide, their log-densities must be exactly equal for all $r > 0$. Equating the two expressions and differentiating twice with respect to $r$ yields $-1/(g(u)^2) = 0$. Since $g(u) > 0$ by definition, this equation has no solution, yielding a contradiction. Thus, no such profile $g$ exists, and $Q \notin [\,P\,]_{\mathcal{T}}$.
\end{proof}

\subsection{Optimization Benefits of Matching Radial Summary Statistics} \label{sec:canonicalization}

The previous section showed that one can exploit the change-of-variables formula to show that $\Tcal_g$ can map $P$ to a distribution $Q := (\Tcal_g)_{\#}(P)$ such that $Q$ has a desired radial summary statistic with a specific choice of $g_*$ given by \eqref{eq:opt-g-star} in Corollary \ref{cor:matching-radial-stat}. Here, we ask what the potential benefits of such a property may be. One potential benefit is that this gives a direct way to choose a profile $g$ to ensure $\rho_Q$ is much simpler than $\rho_P$. This may be desirable, for example, when the gauge or regularizer $\|\cdot\|_{K_P} \propto\rho_P(\cdot)^{-1}$ is highly nonconvex, but one would still like to use it as a regularizer in a variational formulation to solve inverse problems. A potential solution to this issue is to map $P \mapsto Q = (\Tcal_g)_{\#}(P)$ in such a way that the induced variational optimization problem with regularizer $\|\cdot\|_{K_Q}\propto \rho_Q^{-1}$ is better behaved than the original formulation. We show here that this can indeed be done in a simple setting.

Suppose we aim to solve a denoising problem $y = x + \eta$ where $x\sim P$ and $\eta \in \R^d$ denotes noise. A natural formulation to solve this using the regularizer $\|\cdot\|_{K_P}$ is to solve the following variational problem for a choice of $\lambda > 0$, \begin{align}
    \min_{x \in \R^d} F_{\mathrm{orig}}(x): = \frac{1}{2}\|y - x\|_2^2 + \lambda \|x \|_{K_P}. \label{eq:original-problem}
\end{align} An alternative formulation is to exploit the mapping $\Tcal_g$ and pose the problem as a regularized variational problem where we view $x \sim P$ as $x = (\Tcal_g)^{-1}(z)= \Tcal_{1/g}(z)$ with $z \sim Q = (\Tcal_g)_{\#}(P)$ and we regularize $z$ using $\|\cdot\|_{K_Q}$ instead: \begin{align}
    \min_{z \in \R^d} F_{\mathrm{inv}}(z) := \frac{1}{2}\|\Tcal_{1/g}(z) - y\|_2^2 + \lambda\|z\|_{K_Q}. \label{eq:transformed-problem}
\end{align} Formulations of this type arise in inverse problems with generative priors, where the unknown is parameterized by a generator and regularization is imposed on its latent variables \cite{Boraetal17, asim2020invertible, Duffetal2024regularising}. Related formulations have also been studied within nonlinear synthesis regularization \cite{Obmannetal2021synthesis}. Here, we exploit the explicit structure of anisotropic dilation to establish conditions under which formulation \eqref{eq:transformed-problem} is weakly convex, even though formulation \eqref{eq:original-problem} is not weakly convex on any neighborhood of the origin. Prior to the statement of the theorem, recall that a function $F$ is $\beta$-weakly convex if $F(\cdot) + \frac{\beta}{2}\|\cdot\|_2^2$ is convex for any $\beta \geq \beta_*$.
\begin{theorem} \label{thm:better-landscape}
Let $P$ be a distribution with density $p$ on $\R^d$ such that $\rho_P$ is the radial function of a nonconvex star body $K_P$. Then for any $\lambda > 0$, the objective \eqref{eq:original-problem} is never weakly convex on any neighborhood of the origin. In contrast, suppose we set $g(u):= \rho_P(u)^{-(d+1)}$ with $\rho_P$ $C^2$ on the sphere. Then $\|\cdot\|_{K_Q}=\|\cdot\|_2$ and there exists a constant $\beta_* \in (0,\infty)$ and a parameter $\lambda_* \geq 0$ such that whenever $\lambda \geq \lambda_*$, the objective \eqref{eq:transformed-problem} is $\beta_*$-weakly convex.  
\end{theorem}

\begin{proof}
    For the first part of the result, assume for the sake of contradiction that $F_{\mathrm{orig}}$ is weakly convex for some $\lambda > 0$ on a ball $B(0,\epsilon)$. That is, for some $\lambda > 0$,  there exists a $\beta \geq 0$ such that $$H_{\beta}(x):= F_{\mathrm{orig}}(x) + \frac{\beta}{2}\|x\|_2^2$$ is convex on $B(0,\epsilon)$. Consider two arbitrary points $x,q \in \R^d$ and their linear combination $w(\theta):= \theta x + (1-\theta)q$ for $\theta \in [0,1]$. For some sufficiently small $t > 0$, $tx, tq,tw \in B(0,\epsilon)$. Convexity of $H_{\beta}$ implies that \begin{align*}
        H_{\beta}(tw(\theta)) \leq \theta H_{\beta}(tx) + (1-\theta) H_{\beta}(tq).
    \end{align*} Expanding and using positive homogeneity of the gauge, the above inequality is equivalent to \begin{align*}
        t\lambda \|w(\theta)\|_{K_P} & - t \langle y,w(\theta)\rangle + \frac{1+\beta}{2}t^2 \|w(\theta)\|_2^2 \leq \theta\left(\lambda t \|x\|_{K_P} - t\langle y,x\rangle + \frac{1+\beta}{2}t^2\|x\|_2^2\right) \\
        & \quad + (1-\theta) \left(\lambda t \|q\|_{K_P} - t\langle y,q\rangle + \frac{1+\beta}{2}t^2\|q\|_2^2\right).
    \end{align*} Then dividing by $t > 0$ and noting that $\langle y,w(\theta)\rangle = \theta \langle y, x\rangle + (1-\theta)\langle y, q\rangle,$ the linear terms cancel. Letting $t \downarrow 0$, we obtain $$\lambda \|w(\theta)\|_{K_P} \leq \lambda \theta\|x\|_{K_P} + \lambda (1-\theta)\|q\|_{K_P}.$$ Since $\lambda > 0$, this is equivalent to convexity of the gauge $x \mapsto \|x\|_{K_P}$, which is a contradiction.

    For the second part of the result, note that when $g(u)=\rho_P(u)^{-(d+1)}$, $\rho_Q(u)^{d+1} =1$ so that $\|\cdot\|_{K_Q} = \|\cdot\|_2$. Consider the map $\psi(u):= \rho_P(u)^{d+1}$. Since $\rho_P$ is $C^2$ on the sphere, $\psi$ is also $C^2$ on the sphere. Further, note that for $z = ru \neq 0$ with $r > 0$ and $u=z/\|z\|_2$, we have \begin{align*}
        F_{\mathrm{inv}}(ru) & = \frac{1}{2}r^2\psi(u)^2 + r(\lambda - \psi(u)\langle y,u\rangle) + \frac{1}{2}\|y\|^2_2.
    \end{align*} Let $F_1(z):= \frac{1}{2}\|z\|^2_2 \psi(z/\|z\|_2)^2$ and $F_2(z):= \|z\|_2(\lambda - \psi(z/\|z\|_2)\langle y, z/\|z\|_2\rangle).$ Then $F_{\mathrm{inv}}(z) = F_1(z) + F_2(z) + \mathrm{const}$ so it suffices to show $F_1 + F_2$ is weakly convex. We will show that $F_1$ is weakly convex and $F_2$ is convex for $\lambda > 0$ sufficiently large. For $F_1$, note that $F_1$ is $2$-homogeneous and $C^2$ away from the origin. Since $\psi \in C^2(\mathbb{S}^{d-1})$, the Hessian $\nabla^2F_1(z)$ is continuous on $\R^d \setminus\{0\}$ and zero-homogeneous, i.e., $\nabla^2 F_1(tz) = \nabla^2 F_1(z)$ for all $t > 0$ \cite[Lemma D.1]{ji2020directional}. Hence $\sup_{\|z\|_2=1} \|\nabla^2 F_1(z)\|_{\mathrm{op}} < \infty$, which means that if we choose $\beta_* := 1 + \sup_{\|z\|_2=1} \|\nabla^2 F_1(z)\|_{\mathrm{op}} \in (0,\infty)$, we have that the Hessian of the map $z\mapsto F_{1,\beta_*}(z):= F_1(z) + \frac{\beta_*}{2}\|z\|_2^2$ is positive semi-definite for any $z \neq 0$, meaning that it is convex away from $0$. Now, if one considers a line $\ell(t):= t v$ for some $v\neq0$ passing through the origin, we want to show that $t\mapsto F_{1,\beta_*}(\ell(t))$ is convex. By positive two-homogeneity, $F_{1,\beta_*}(\ell(t)) = t^2 F_{1,\beta_*}(v)$ if $t \geq 0$ and $t^2F_{1,\beta_*}(-v)$ otherwise. Moreover, $F_{1,\beta_*}(\pm v) \geq 0$ so that the map $t \mapsto F_{1,\beta_*}(\ell(t))$ has second derivative $2F_{1,\beta_*}(\pm v) \geq 0$ with first derivative $2tF_{1,\beta_*}(-v) \leq 0$ for $t < 0$ and $2tF_{1,\beta_*}(v) \geq 0$ for $t > 0$. Hence $F_{1,\beta_*}(\ell(t))$ is convex on $\R$. Since $F_{1,\beta_*}$ has $\nabla^2 F_{1,\beta_*}(z) \succeq 0$ for $z \neq 0$ and the restriction of $F_{1,\beta_*}$ to any line passing through the origin is convex, $F_{1,\beta_*}$ is convex, establishing that $F_1$ is $\beta_*$-weakly convex.

    It remains to show that $$z\mapsto F_2(z)=\lambda \|z\|_2 - \langle y,z\rangle \psi(z/\|z\|_2)$$ is convex for sufficiently large $\lambda > 0$. Since $F_2$ is positively homogeneous, it suffices to show that $F_2$ is sub-additive. For arbitrary $z_1,z_2 \in \R^d$, let $\delta := \|z_1\|_2 + \|z_2\|_2 - \|z_1 + z_2\|_2 \geq 0.$ We first require a technical estimate, which is proven in Lemma \ref{lem:technical-bound-convexity} in the Appendix. In particular, Lemma \ref{lem:technical-bound-convexity} shows that when the map $\Phi(z):= z\psi(z/\|z\|_2)$ has $\psi \in C^2(\mathbb{S}^{d-1})$, there is a constant $C_{\psi} < \infty$ such that \begin{align}
        \|\Phi(z_1) + \Phi(z_2) - \Phi(z_1+z_2)\|_2 \leq C_{\psi}(\|z_1\|_2 + \|z_2\|_2 - \|z_1+z_2\|_2),\ \quad \forall z_1,z_2\in\R^d
        \label{eq:lemma-helper-bound-Phi}
    \end{align} 
    Now, using the definition of $F_2$, note that 
    \begin{align}
         F_2(z_1 + z_2) - F_2(z_1) - F_2(z_2) & = \lambda (\|z_1+z_2\|_2 - \|z_1\|_2 - \|z_2\|_2) - \langle y, \Phi(z_1+z_2) - \Phi(z_1) - \Phi(z_2)\rangle \nonumber \\
        & = - \lambda \delta + \langle y, \Phi(z_1) + \Phi(z_2) - \Phi(z_1+z_2)\rangle. \label{eq:F_2-representation}
    \end{align}
    By \eqref{eq:lemma-helper-bound-Phi}, we have that $\langle y, \Phi(z_1)+\Phi(z_2) - \Phi(z_1+z_2)\rangle \leq \|y\|_2 C_{\psi} \delta$ for some constant $C_{\psi} > 0$. Combining this bound with \eqref{eq:F_2-representation}, we get that \begin{align*}
        F_2(z_1+z_2) - F_2(z_1) - F_2(z_2) \leq -(\lambda - \|y\|_2 C_{\psi})\delta.
    \end{align*} Thus, whenever $\lambda \geq \lambda_* := \|y\|_2 C_{\psi}$, we have that $F_2$ is subadditive, and hence convex.
\end{proof}
While this result focused on the case when the induced regularizer is the simple $\ell_2$-norm, other generalizations of this result are possible. The key idea is that by exploiting this invertible map, we can move the nonconvexity from the nonsmooth regularization term to the quadratic data-fidelity term while enjoying a well-behaved, convex regularizer.

\section{Learning Optimal Anisotropic Dilations}\label{sec:learning-g}

The preceding results characterize anisotropic dilation as an action on probability distributions: for any $P\in\mathcal P$ and $g\in G$, the pushforward $(\Tcal_g)_{\#}P$ changes radial information in a controlled way while preserving directions. This transformation offers two complementary perspectives. The first, as explored in Section \ref{sec:canonicalization}, is to apply a specific $\Tcal_g$ to the data distribution $P$ to obtain a simpler induced regularizer. In this section, we adopt a generative viewpoint: starting from a tractable base family, we seek to learn an anisotropic scaling profile $g$ such that the pushforward of a suitable base distribution closely approximates a target data distribution $P$. Specifically, we consider Gibbs densities $p_K(x) \propto \exp(-\|x\|_K)$ and define the base class
 \begin{align}
\mathcal D = \left\{ p_K : \mathbb R^d \to \mathbb R  :p_K(x)=Z_K^{-1}e^{-\|x\|_K}, \; Z_K = \int_{\mathbb R^d} e^{- \| x\|_K} \mathrm{d}x, \; \mathrm{vol}_d(K)=1, \; K \in \mathcal S^d \right\}. \label{eq:D-definition}
\end{align} 
Throughout this section, we identify each density $p_K$ with the probability measure it induces when using pushforward notation.
The unit-volume constraint fixes the global scale of the base star body $K$, removing a scaling ambiguity between $K$ and $g$. Indeed, for any $c>0$, dilating $K$ to $cK$ yields the distribution $p_{cK}=(\Tcal_c)_{\#}p_K$, where $\Tcal_c$ denotes anisotropic dilation with constant profile $g(u)\equiv c$. Thus, the global scale could otherwise be represented through either $K$ or $g$, creating a redundant parameterization.

\paragraph{Why Can Learning $g$ Help?}
The Gibbs family $\mathcal D$ is structurally rigid, as a single star body $K$ must simultaneously determine both the directional marginal and the conditional mean radius in each direction. For general data distributions, these two quantities need not be coupled in this way. Anisotropic dilation introduces the direction-preserving flexibility needed to separate them.

To see the source of this rigidity, recall from Section~\ref{sec:preliminaries} the decomposition $\rho_P^{d+1}=h_Pm_P$ in \eqref{eq:rho-decomposition}. This identity shows that the radial summary statistic $\rho_P$ depends both on the density of directional mass, as quantified by $h_P(u)$, and on how far that mass lies from the origin on average, as quantified by $m_P(u)$. For a Gibbs density $p_K\in\mathcal D$, both $h_{p_K}$ and $m_{p_K}$ are determined by the same radial function $\rho_K$. This dependence is made explicit in the following lemma.

\begin{lemma}
\label{lem:gibbs-polar-densities}
Let $p_K(x)=Z_K^{-1}\exp(-\|x\|_K)$ be the Gibbs density associated with a star body $K\in\mathcal S^d$. Then, for every $u\in\mathbb S^{d-1}$, the directional marginal density and conditional mean radius satisfy
$$
    h_{p_K}(u)=C_K\rho_K(u)^d,
    \qquad
    m_{p_K}(u)=d\rho_K(u),
$$
where $C_K:=Z_K^{-1}\Gamma(d)=1/(d\,\mathrm{vol}_d(K))$. In particular, if $p_K\in\mathcal D$, then $C_K=1/d$.
\end{lemma}
\begin{proof}
    Using the definition of the directional marginal and the positive homogeneity $\|ru\|_K=r\|u\|_K$ for $r>0$, we obtain
    \begin{align*}
    h_{p_K}(u) = \int_0^\infty r^{d-1}p_K(ru)\,\mathrm{d}r = Z_K^{-1}\int_0^\infty r^{d-1}e^{-r\|u\|_K}\,\mathrm{d}r = Z_K^{-1}\Gamma(d)\|u\|_K^{-d} = C_K\rho_K(u)^d,
\end{align*}
where we used the identity $\rho_K(u)=\|u\|_K^{-1}$. Furthermore,
$Z_K=\Gamma(d+1)\mathrm{vol}_d(K)$
\cite[Proposition 1]{leong2025optimal}, and hence $C_K = \Gamma(d)/(\Gamma(d+1)\mathrm{vol}_d(K)) =
1/(d\,\mathrm{vol}_d(K)).$

For the conditional mean radius $m_{p_K}(u)=\mathbb{E}[R\mid U=u]$, using the conditional radial density induced by polar coordinates, we obtain
\begin{align*}
    m_{p_K}(u) = \frac{1}{h_{p_K}(u)} \int_0^\infty r^d p_K(ru)\,\mathrm{d}r &= \frac{Z_K^{-1}}{C_K\rho_K(u)^d} \int_0^\infty r^d e^{-r\|u\|_K}\,\mathrm{d}r \\
    &= \frac{Z_K^{-1}\Gamma(d+1)\|u\|_K^{-(d+1)}} {Z_K^{-1}\Gamma(d)\rho_K(u)^d} \\
    &= \frac{d\Gamma(d)\rho_K(u)^{d+1}}{\Gamma(d)\rho_K(u)^d} \\
    &= d\rho_K(u).
\end{align*}
\end{proof}
The lemma provides closed-form expressions for $h_{p_K}$ and $m_{p_K}$ that will be used throughout the subsequent analysis. It also shows that the directional marginal and conditional mean radius cannot vary independently within $\mathcal D$, since both are governed by the same radial function $\rho_K$.

Anisotropic dilation breaks this dependence by introducing the scaling profile $g$ as an additional continuous degree of freedom. To see this, for $p_K\in\mathcal D$ and $g\in G$, let $Q_{K,g}:=(\Tcal_g)_{\#}p_K$ denote a distribution in the anisotropic radial orbit $[\,p_K\,]_{\mathcal T}$. By Theorem~\ref{thm:structure-of-orbit}, $\Tcal_g$ preserves the directional marginal while rescaling the conditional radius along each ray. Therefore, for every $u\in\mathbb S^{d-1}$,
\begin{align*}
h_{Q_{K,g}}(u)
&= h_{p_K}(u) = \frac{1}{d}\rho_K(u)^d, \\
m_{Q_{K,g}}(u)
&= g(u)m_{p_K}(u) = d\,g(u)\rho_K(u).
\end{align*}
Hence, $K$ alone determines the directional marginal, after which $g$ rescales the conditional mean radius. This enables the model to adjust its mean radial geometry without altering its directional mass, thereby accommodating distributions in which these two structures need not be coupled.

In the following subsection, we formalize this observation as a population Kullback--Leibler (KL) projection problem and characterize the optimal pair $(K^*,g^*)$ for which $Q_{K^*,g^*}$ best approximates a given target distribution $P$. Then, in Sections~\ref{sec:provable-benefits} and \ref{sec:examples-isotropy-cost}, we derive a diagnostic quantity determined entirely by $P$ that characterizes when the additional flexibility provided by $g$ strictly improves the approximation, and evaluate this diagnostic across several common classes of distributions.

\subsection{Population Projection onto Gibbs Orbits}
\label{sec:Gibbs-orbit-model}

We now formalize the preceding discussion. For $K\in\mathcal S^d$ and $g\in G$, write $Q_{K,g}:=(\mathcal T_g)_\#p_K$. For each $p_K\in\mathcal D$, let $[\,p_K\,]_{\mathcal T}:=\{Q_{K,g}:g\in G\}$ denote its orbit under anisotropic dilation. We consider the family of distributions obtained by taking the union of these Gibbs orbits:
\begin{align*}
\mathcal D_{\mathrm{orb}} :=
\bigcup_{p_K \in \mathcal D}[\,p_K\,]_{\mathcal T} = \left\{ Q_{K,g} \;\middle|\; p_K\in\mathcal D,\;g\in G \right\}.
\end{align*} 
Throughout Section~4, we assume that the target distribution
$P\in\mathcal P$, with density $p$, satisfies $\mathbb E_P[|\log p(X)|]<\infty$, where $X\sim P$. The distribution-fitting problem of interest is then given by
\begin{align}
\inf_{Q_{K,g} \in \mathcal D_{\mathrm{orb}}}
D_{\mathrm{KL}}(P \| Q_{K,g})
=
\inf_{\substack{K\in\mathcal S^d\\ \mathrm{vol}_d(K)=1}}
\left(
\inf_{g \in G}
D_{\mathrm{KL}}(P \| Q_{K,g})
\right).
\label{eq:inner-and-outer-problems}
\end{align}
We consider the KL divergence in this direction to maintain a direct connection with prior results in the optimal regularization literature \cite{leong2025optimal}. Problem \eqref{eq:inner-and-outer-problems} can be viewed as a forward KL projection of $P$ onto the family $\mathcal D_{\mathrm{orb}}$.
Moreover, it possesses a natural inner--outer structure: we first fix $K$ and optimize over $g$, and then optimize the resulting objective over $K$. Both optimizations admit exact population-level characterizations.
We begin by rewriting $D_{\mathrm{KL}}(P\|Q_{K,g})$ in a form that will be useful throughout the subsequent analysis.

\begin{lemma}
\label{lem:KL-objective}
Fix $P\in\mathcal P$ satisfying the standing integrability assumption. Then minimizing $D_{\mathrm{KL}}(P\|Q_{K,g})$ with respect to $K\in\mathcal S^d$ and $g\in G$ is equivalent to minimizing the functional
$$
    \mathcal J(K,g)
    :=
    \mathbb E_{P_U}[\phi_{U,K}(g(U))]
    +
    \log\mathrm{vol}_d(K)
$$
over the same parameters, where the continuous integrand $\phi_{u,K}:(0,\infty)\to\mathbb R$ is given by
$$
    \phi_{u,K}(t)
    :=
    d\log t+\frac{\|u\|_K m_P(u)}{t}.
$$
\end{lemma}

\begin{proof}
By Theorem~\ref{thm:structure-of-orbit}, the pushforward $Q_{K,g}$ has density
$$
    q_{K,g}(x)
    =
    Z_K^{-1}g(x/\|x\|_2)^{-d}
    \exp\left(
    -\frac{\|x\|_K}{g(x/\|x\|_2)}
    \right),
    \qquad x\neq0.
$$
Since $m_P$ is bounded, $\mathbb E_P[\|X\|_2]=\mathbb E_{P_U}[m_P(U)]<\infty$. Together with the bounds on $g$ and $\rho_K$, this gives $\mathbb E_P[|\log q_{K,g}(X)|]<\infty$. Considering the standing integrability assumption on $p$, this ensures that $D_{\mathrm{KL}}(P\|Q_{K,g})=\mathbb E_P[\log p(X)]-\mathbb E_P[\log q_{K,g}(X)]$ is finite. Because the first term is independent of $K$ and $g$, minimizing the divergence is equivalent to minimizing the cross-entropy
$$
    \mathbb E_P[-\log q_{K,g}(X)]
    =
    \log Z_K
    +
    \mathbb E_P\left[
    d\log g(X/\|X\|_2)
    +
    \frac{\|X\|_K}{g(X/\|X\|_2)}
    \right].
$$
Using the identity $Z_K=\Gamma(d+1)\mathrm{vol}_d(K)$ from Lemma~\ref{lem:gibbs-polar-densities}, we write $\log Z_K=\log\mathrm{vol}_d(K)+\log\Gamma(d+1)$ and discard the constant $\log\Gamma(d+1)$ from the optimization objective.

To evaluate the remaining expectation, decompose $X$ into its polar components $X=RU$, where $U=X/\|X\|_2$ and $R=\|X\|_2$. By the positive homogeneity of the Minkowski functional, $\|X\|_K=\|RU\|_K=R\|U\|_K$. Applying the law of total expectation conditional on $U$ gives
\begin{align*}
    \mathbb E_P\left[
    d\log g(U)+\frac{R\|U\|_K}{g(U)}
    \right]
    &=
    \mathbb E_{P_U}\left[
    \mathbb E_P\left[
    d\log g(U)+\frac{R\|U\|_K}{g(U)}
    \,\middle|\,U
    \right]
    \right]\\
    &=
    \mathbb E_{P_U}\left[
    d\log g(U)+\frac{\|U\|_K m_P(U)}{g(U)}
    \right]\\
    &=
    \mathbb E_{P_U}[\phi_{U,K}(g(U))].
\end{align*}
Conditional on $U=u$, $g(U)$ and $\|U\|_K$ are fixed, leaving the conditional mean radius $m_P(u)=\mathbb E_P[R\mid U=u]$. Combining this with the $\log\mathrm{vol}_d(K)$ term yields the functional $\mathcal J(K,g)$.
\end{proof}
\paragraph{Inner Optimization: Radial Matching}
We first characterize the optimal scaling profile for the inner optimization problem in \eqref{eq:inner-and-outer-problems}. For a fixed target distribution $P\in\mathcal P$ and star body $K\in\mathcal S^d$, define the \emph{radial discrepancy ratio of $K$ relative to $P$} by
\begin{equation}
    \alpha_K(u)
    :=
    \frac{m_P(u)}{\rho_K(u)},
    \qquad
    u\in\mathbb S^{d-1}.
    \label{eq:radial-discrepancy-ratio}
\end{equation}
This ratio compares the conditional mean radius of the target distribution with the radial geometry of the base star body $K$. By Lemma~\ref{lem:gibbs-polar-densities}, $m_{p_K}(u)=d\rho_K(u)$, so $\alpha_K(u)=d\,m_P(u)/m_{p_K}(u)$. Thus, $\alpha_K(u)/d$ is precisely the direction-dependent scaling required to match the conditional mean radius of $p_K$ to that of the target distribution $P$. The following theorem shows that this radial matching profile is also the unique solution to the inner KL optimization problem.

\begin{theorem}[Optimal Scaling Profile]
\label{thm:optimal-g-K}
Suppose $P\in\mathcal P$. For a fixed $K\in\mathcal S^d$, the unique minimizer $g_K^*$ of the inner optimization problem in \eqref{eq:inner-and-outer-problems} is given by
\begin{equation}
    g_K^*(u)
    =
    \frac{\alpha_K(u)}{d},
    \qquad u\in\mathbb S^{d-1}.
    \label{eq:optimal-g}
\end{equation}
\end{theorem}
\begin{proof}
    Fix a star body $K$. By Lemma~\ref{lem:KL-objective}, the term $\log\mathrm{vol}_d(K)$ is independent of $g$, so minimizing $\mathcal J(K,g)$ over $g\in G$ reduces to minimizing $\mathbb E_{P_U}[\phi_{U,K}(g(U))]$. We proceed by minimizing the integrand pointwise. For fixed $u\in\mathbb S^{d-1}$, differentiating $\phi_{u,K}(t)$ with respect to $t>0$ gives
    $$
    \phi'_{u,K}(t) = \frac{d}{t} - \frac{\|u\|_Km_P(u)}{t^2} = \frac{dt - \|u\|_Km_P(u)}{t^2}.
    $$
    Since $K$ is a star body and $m_P(u)>0$, we have $\|u\|_K m_P(u)>0$. Thus, $\phi'_{u,K}(t)<0$ for $0<t<\|u\|_K m_P(u)/d$ and $\phi'_{u,K}(t)>0$ for $t>\|u\|_K m_P(u)/d$. Hence, $\phi_{u,K}$ decreases strictly up to the critical point $t_u^*:=\|u\|_K m_P(u)/d$ and increases strictly thereafter, so $t_u^*$ is its unique global minimizer. Using the identity $\|u\|_K=\rho_K(u)^{-1}$, define $g_K^*(u):=t_u^* = m_P(u)/(d\rho_K(u))$ for $u\in\mathbb S^{d-1}$. For $P\in\mathcal P$, \eqref{eq:rho-decomposition} gives $m_P(u)=\rho_P(u)^{d+1}/h_P(u)$, which is positive and continuous as the ratio of positive continuous functions. Since $\rho_K$ is also positive and continuous, it follows that $g_K^*\in G$.

    The pointwise minimization gives $\phi_{u,K}(g(u))\geq\phi_{u,K}(g_K^*(u))$ for every $u\in\mathbb S^{d-1}$ and every $g\in G$. Integrating this inequality with respect to $P_U$ establishes optimality. To establish uniqueness, suppose another $g\in G$ differs from $g_K^*$ at some $u_0\in\mathbb S^{d-1}$. By continuity, the two functions differ on an open neighborhood of $u_0$. Since $h_P$ is strictly positive on the sphere, every nonempty open subset of $\mathbb S^{d-1}$ has positive $P_U$-measure. The pointwise inequality is therefore strict on a set of positive $P_U$-measure, yielding a strictly larger objective. Hence, $g_K^*$ is the unique minimizer in $G$.
\end{proof}
The radial matching interpretation also admits a geometric description. Since $m_P$ is positive and continuous for every $P\in\mathcal P$, it defines the radial function of a unique star body, $K_{m_P}:=\{ru:0\leq r\leq m_P(u),\;u\in\mathbb S^{d-1}\}$. Furthermore, since $m_{p_K}(u)=d\rho_K(u)$, we have $K_{m_{p_K}}=dK$. Therefore, $g_K^*(u)=\alpha_K(u)/d=m_P(u)/m_{p_K}(u)$ is exactly the star-body matching profile illustrated in Figure~\ref{fig:star-body-matching} that maps $dK$ to $K_{m_P}$. Equivalently, $\mathcal T_{g_K^*}$ rescales $p_K$ along each ray so that, for every $u\in\mathbb S^{d-1}$,
$$
m_{Q_{K,g_K^*}}(u) = g_K^*(u)m_{p_K}(u) = m_P(u).
$$

\paragraph{Outer Optimization: Directional Matching}
Having characterized the optimal scaling profile $g_K^*=\alpha_K/d$, we now turn to the star body that solves the outer optimization problem. Since the inner optimization exactly matches the conditional mean radius of the Gibbs model to that of the target distribution, the remaining optimization determines how the directional mass is represented by the base star body. For each fixed $K$, the invariance of KL divergence under invertible transformations \cite{polyanskiy_wu_info_theory} allows us to apply $\mathcal T_{g_K^*}^{-1}$ to both $P$ and $Q_{K,g_K^*}$. Thus, the outer optimization problem can be written as
\begin{equation}
    \inf_{\substack{K\in\mathcal S^d\\ \mathrm{vol}_d(K)=1}}
    D_{\mathrm{KL}}\left(
    \left(\mathcal T_{g_K^*}^{-1}\right)_\#P
    \,\middle\|\,
    p_K
    \right).
    \label{eq:outer-opt-problem}
\end{equation}
The following theorem shows that, up to an additive constant independent of $K$, this objective is a KL divergence between directional marginals, and explicitly characterizes the unique minimizing star body.

\begin{theorem}[Optimal Directional Star Body]
\label{thm:optimal-starbody}
Let $P\in\mathcal P$. Given the optimal scaling profile \eqref{eq:optimal-g}, the outer optimization problem \eqref{eq:outer-opt-problem} has the same minimizers as
\begin{equation}
    \inf_{\substack{K\in\mathcal S^d\\ \mathrm{vol}_d(K)=1}}
    D_{\mathrm{KL}}(h_P\|h_{p_K}).
    \label{eq:simplified-outer-opt-problem}
\end{equation}
Furthermore, the unique solution $K^*$ is defined by the radial function 
\begin{equation}
    \rho_{K^*}(u)
    =
    (d\,h_P(u))^{1/d},
    \qquad u\in\mathbb S^{d-1}.
    \label{eq:optimal-rho_K^*}
\end{equation}
\end{theorem}

\begin{proof}
Evaluating the integrand $\phi_{u,K}$ from Lemma~\ref{lem:KL-objective} at the optimal scale $g_K^*(u)=\alpha_K(u)/d$ gives
$$
\phi_{u,K}(g_K^*(u)) = d\log\left(\frac{\alpha_K(u)}{d}\right)+d = d\log\|u\|_K+C_P(u),
$$
where we used $\alpha_K(u)=m_P(u)/\rho_K(u)=\|u\|_K m_P(u)$ and defined $C_P(u):=d\log(m_P(u)/d)+d$, which is independent of $K$. Since $\mathrm{vol}_d(K)=1$ throughout the outer optimization, the term $\log\mathrm{vol}_d(K)$ in $\mathcal J(K,g_K^*)$ vanishes. After discarding terms independent of $K$, the outer problem therefore reduces to
\begin{align}
\inf_{\substack{K\in\mathcal S^d\\
                \mathrm{vol}_d(K)=1}}
\mathbb E_{P_U}[d\log\|U\|_K].
\label{eq:outer-obj}
\end{align}
By Lemma~\ref{lem:gibbs-polar-densities}, $h_{p_K}(u)=\|u\|_K^{-d}/d$ when $\mathrm{vol}_d(K)=1$. Taking logarithms gives $d\log\|u\|_K=-\log h_{p_K}(u)-\log d$. Consequently, minimizing \eqref{eq:outer-obj} is equivalent to minimizing the cross-entropy $-\mathbb E_{P_U}[\log h_{p_K}(U)]$. This differs from $D_{\mathrm{KL}}(h_P\|h_{p_K})$ by the finite term $-\mathbb E_{P_U}[\log h_P(U)]$, which is independent of $K$, establishing \eqref{eq:simplified-outer-opt-problem}.

Since $h_P$ is positive and continuous, the function $(d\,h_P)^{1/d}$ defines the radial function of a star body $K^*$. Moreover, the polar volume formula gives
$$
\mathrm{vol}_d(K^*) = \frac{1}{d} \int_{\mathbb S^{d-1}} \rho_{K^*}(u)^d\,\mathrm d\sigma(u) = \int_{\mathbb S^{d-1}} h_P(u)\,\mathrm d\sigma(u) = 1.
$$
Thus, $K^*$ is admissible and $h_{p_{K^*}}=\rho_{K^*}^d/(d\,\mathrm{vol}_d(K^*))=h_P$, attaining the minimum value zero in \eqref{eq:simplified-outer-opt-problem}.

Conversely, any minimizing star body $K$ must satisfy $h_{p_K}=h_P$ $\sigma$-almost everywhere, by the equality condition for KL divergence. Using $h_{p_K}(u)=\rho_K(u)^d/d$, we obtain $\rho_K(u)=(d\,h_P(u))^{1/d}$ $\sigma$-almost everywhere. Since both sides are continuous on $\mathbb S^{d-1}$ and every nonempty open subset has positive surface measure, the equality holds everywhere. Hence, $K=K^*$, proving uniqueness.
\end{proof}
The reduction of \eqref{eq:outer-opt-problem} to \eqref{eq:simplified-outer-opt-problem} makes the decoupling effect of anisotropic dilation explicit. Once the inner optimization matches the conditional mean radius via \eqref{eq:optimal-g}, the remaining dependence on $K$ is entirely through the directional marginal. Consequently, the optimal star body $K^*$ is determined solely by the directional density $h_P$ and corresponds to the unique Gibbs density in $\mathcal D$ whose directional marginal matches that of the target distribution $P$. Together with Theorem~\ref{thm:optimal-g-K}, this yields the unique optimal pair $(K^*,g^*)$, where $g^*:=g_{K^*}^*$, for \eqref{eq:inner-and-outer-problems}.

\subsection{Provable Benefits of Anisotropic Dilation} \label{sec:provable-benefits}
The previous section derived closed-form solutions to the inner and outer optimization problems in \eqref{eq:inner-and-outer-problems} for approximating a target distribution $P\in\mathcal P$. Theorem~\ref{thm:optimal-g-K} gives the optimal direction-dependent scaling profile $g_K^*$ for each fixed star body $K$, while Theorem~\ref{thm:optimal-starbody} identifies the optimal star body $K^*$. A foundational question remains: \emph{when and by how much does the directional flexibility of anisotropic dilation improve upon isotropic scaling?}

Here, isotropic scaling refers to restricting the profile to be constant, $g(u)\equiv c>0$. For a fixed star body $K$, let $c_K^*$ denote the optimal constant scaling factor, defined by the restricted inner optimization problem
\begin{align}
    c_K^* \in \argmin_{c > 0} D_{\mathrm{KL}}(P \| Q_{K,c}) \label{eq:optimal-c}
\end{align} 
where $Q_{K,c} := (\Tcal_c)_{\#} p_K$. Since the positive constant profiles form a proper subclass of $G=C(\mathbb S^{d-1},(0,\infty))$, restricting $g$ to be constant cannot yield a smaller KL divergence than the optimal anisotropic profile $g_K^*$. The resulting nonnegative divergence gap measures the approximation penalty incurred by imposing isotropic scaling. We first analyze this gap within a single orbit $[\,p_K\,]_{\mathcal T}$ associated with a fixed star body $K$.

\begin{definition}[Isotropy Cost of a Star Body]
\label{def:star-isotropy-cost} Fix a target distribution $P\in\mathcal P$ and a star body $K\in\mathcal S^d$.  Let $c_K^*$ denote the optimal isotropic scaling factor defined in \eqref{eq:optimal-c}. The isotropy cost of $K$ with respect to $P$ is defined by 
\begin{align} 
    \Delta_{\mathrm{star}}(K) := D_{\mathrm{KL}}(P\|Q_{K,c_K^*}) - D_{\mathrm{KL}}(P\|Q_{K,g_K^*}).
    \label{eq:isotropy-cost-of-K}
\end{align}
\end{definition}
Recall from \eqref{eq:radial-discrepancy-ratio} that $\alpha_K$ is the radial discrepancy ratio of $K$ relative to $P$, while Theorem~\ref{thm:optimal-g-K} gives
$g_K^*(u)=\alpha_K(u)/d$. Thus, as a function of $u\in \mathbb S^{d-1}$, $\alpha_K$ encodes the directional variation of the optimal scaling profile. In particular, the optimal anisotropic profile is constant precisely when $\alpha_K$ is constant on the sphere. The following theorem formalizes this observation and quantifies the cost when $\alpha_K$ is nonconstant.

\begin{theorem}[Isotropy Cost of a Fixed Star Body]
\label{thm:fixed-K-correction}
Fix a target distribution $P\in\mathcal P$ and a star body $K\in\mathcal S^d$. Let $c_K^*$ denote the optimal constant scaling factor defined in \eqref{eq:optimal-c}. The isotropy cost of $K$ with respect to $P$ is defined by
\begin{align}
\Delta_{\mathrm{star}}(K) = d \left[ \log \mathbb{E}_{P_U}[\alpha_K(U)] - \mathbb{E}_{P_U}[\log \alpha_K(U)] \right]. \label{eq:delta-star-jensen-gap}
\end{align}
Consequently, $\Delta_{\mathrm{star}}(K) \ge 0$, with equality if and only if $\alpha_K(U)$ is constant $P_U$-almost surely.
\end{theorem}
\begin{proof}
    Substituting the constant profile $g(u)\equiv c$ into the variational objective from Lemma~\ref{lem:KL-objective}, and using $\|u\|_K m_P(u)=\alpha_K(u)$, gives
    $$
    \mathcal J(K,c) = d\log c + \frac{1}{c} \mathbb E_{P_U}[\alpha_K(U)] + \log\mathrm{vol}_d(K).
    $$
    We note that $\alpha_K$ is continuous and strictly positive on the compact sphere and thus $0 < \mathbb E_{P_U}[\alpha_K(U)] < \infty.$
    Differentiating with respect to $c$ and setting the derivative equal to zero yields the unique critical point
    \begin{align*}
        c_K^* = \frac{1}{d}\,\mathbb E_{P_U}\big[\alpha_K(U)\big] = \mathbb E_{P_U}[g_K^*(U)].
    \end{align*} 
    Since $\lim_{c\to0^+}\mathcal J(K,c)=+\infty$ and $\lim_{c\to\infty}\mathcal J(K,c)=+\infty$, while the objective decreases strictly before $c_K^*$ and increases strictly thereafter, this critical point is the unique global minimizer on $(0,\infty)$. Substituting $c_K^*$ into the objective yields the optimal isotropic value
    \begin{align}
    \mathcal J(K,c_K^*) &= d\log\mathbb E_{P_U}[\alpha_K(U)] - d\log d + d + \log\mathrm{vol}_d(K). \label{eq:minimal-c-objective}
    \end{align}
    On the other hand, substituting $g_K^*(u)=\alpha_K(u)/d$ into $\mathcal{J}(K, g)$ yields
    \begin{align}
    \mathcal J(K,g_K^*) &= d\,\mathbb E_{P_U}[\log\alpha_K(U)] - d\log d + d + \log\mathrm{vol}_d(K).
    \label{eq:minimal-g-objective}
    \end{align}
    Subtracting the optimal anisotropic value \eqref{eq:minimal-g-objective} from the optimal isotropic value \eqref{eq:minimal-c-objective}  yields the desired isotropy cost of $K$ \eqref{eq:delta-star-jensen-gap}. Since the logarithm is strictly concave, Jensen's inequality guarantees that $\Delta_{\mathrm{star}}(K)\geq0$, with equality if and only if $\alpha_K(U)$ is constant $P_U$-almost surely.
\end{proof}

Theorem~\ref{thm:fixed-K-correction} characterizes the benefit of anisotropic scaling when the underlying star body $K$ is held fixed. We now allow the star body itself to vary over the unit-volume star bodies in $\mathcal S^d$ and compare the best isotropically scaled Gibbs model with the best anisotropically scaled Gibbs model.

\begin{definition}[Isotropy Cost of a Distribution]
\label{def:distribution-isotropy-cost}
For a target distribution $P\in\mathcal P$, the isotropy cost of $P$ is defined by
\begin{align}
    \Delta_{\mathrm{dist}}(P) := \inf_{\substack{K\in\mathcal S^d\\ \mathrm{vol}_d(K)=1}} D_{\mathrm{KL}}(P\|Q_{K,c_K^*}) - \inf_{\substack{K\in\mathcal S^d\\ \mathrm{vol}_d(K)=1}}D_{\mathrm{KL}}(P\|Q_{K,g_K^*}). \label{eq:distribution-isotropy-cost}
\end{align}
\end{definition}
The quantity $\Delta_{\mathrm{dist}}(P)$ measures the approximation improvement provided by anisotropic scaling after the unit-volume star body is optimized separately within each model class. In particular, it is not simply the minimum of $\Delta_{\mathrm{star}}(K)$ over $K$, since the optimal star body for the isotropic family need not coincide with that of the anisotropic family.

To characterize this cost, we first identify the solution to the isotropic outer optimization problem. Let $K_P$ be the canonical data-driven star body defined by $\rho_P$, and let $\widehat K_P$ be its unit-volume dilation:
$$ 
K_P := \left\{ ru: 0\leq r\leq\rho_P(u),\; u\in\mathbb S^{d-1} \right\}, \qquad \widehat K_P := \mathrm{vol}_d(K_P)^{-1/d}K_P.
$$
In the following lemma, we prove that $\widehat K_P$ uniquely solves the isotropic objective, allowing us to compare its optimal value with the anisotropic optimum from Theorem~\ref{thm:optimal-starbody} to explicitly characterize $\Delta_{\mathrm{dist}}(P)$.
\begin{lemma}[Optimal Isotropic Star Body]
\label{lem:optimal-isotropic-starbody}
Let $P\in\mathcal P$. Then $\widehat K_P$ is the unique solution to the isotropic outer optimization problem:
\begin{align}
\widehat K_P
=
\argmin_{\substack{K\in\mathcal S^d\\ \mathrm{vol}_d(K)=1}}
D_{\mathrm{KL}}(P\|Q_{K,c_K^*}).
\label{eq:optimal-isotropic-body}
\end{align}
\end{lemma}
\begin{proof}
By \eqref{eq:minimal-c-objective}, under the unit-volume constraint $\mathrm{vol}_d(K)=1$, minimizing the isotropic objective is equivalent to minimizing $\mathbb E_{P_U}[\alpha_K(U)]$. This quantity satisfies
\begin{align*}
    \mathbb E_{P_U}[\alpha_K(U)]
    &= \int_{\mathbb S^{d-1}}
    \rho_K(u)^{-1}m_P(u)h_P(u)\,\mathrm d\sigma(u) \\
    &= \int_{\mathbb S^{d-1}}
    \rho_K(u)^{-1}\rho_P(u)^{d+1}\,\mathrm d\sigma(u)
    =: d\widetilde V_{-1}(K,K_P) \\
    &\geq
    d\,\mathrm{vol}_d(K)^{-1/d}
    \mathrm{vol}_d(K_P)^{(d+1)/d}.
\end{align*}
The first equality follows from $\alpha_K(u)=m_P(u)/\rho_K(u)$ and $P_U(\mathrm du)=h_P(u)\,\mathrm d\sigma(u)$. The second line uses \eqref{eq:rho-decomposition} and defines the dual mixed volume $\widetilde V_{-1}(K,K_P)$ in the convention of \cite{leong2025optimal}. The final inequality is the dual mixed-volume inequality \cite{Lutwak1975}, with equality if and only if $K$ is a positive dilation of $K_P$. Under the unit-volume constraint, the lower bound is independent of $K$, and equality uniquely determines $K=\widehat K_P=\mathrm{vol}_d(K_P)^{-1/d}K_P$. Thus, $\widehat K_P$ is the unique isotropic minimizer.
\end{proof}

\begin{theorem}[Isotropy Cost of a Distribution]
\label{thm:distribution-isotropy-cost}
Let $P\in\mathcal P$, and let $K_P$ be its canonical star body. Then
\begin{align}
\Delta_{\mathrm{dist}}(P) = (d+1) D_{\mathrm{KL}} (h_P\|h_{p_{K_P}}). \label{eq:distribution-cost-directional-kl}
\end{align}
\end{theorem}
\begin{proof}
We first compute the optimal isotropic value. By Lemma~\ref{lem:optimal-isotropic-starbody}, the unique isotropic optimizer over unit-volume star bodies is $\widehat K_P=\mathrm{vol}_d(K_P)^{-1/d}K_P$, so $\rho_{\widehat K_P}(u)=\mathrm{vol}_d(K_P)^{-1/d}\rho_P(u)$. Consequently, $\alpha_{\widehat K_P}(u)=\mathrm{vol}_d(K_P)^{1/d}m_P(u)/\rho_P(u)$. Taking the expectation with respect to $P_U$ and using \eqref{eq:rho-decomposition}, we obtain
\begin{align*}
    \mathbb E_{P_U}[\alpha_{\widehat K_P}(U)]
    &=
    \mathrm{vol}_d(K_P)^{1/d}
    \int_{\mathbb S^{d-1}}
    \frac{m_P(u)}{\rho_P(u)}h_P(u)\,\mathrm d\sigma(u) \\
    &=
    \mathrm{vol}_d(K_P)^{1/d}
    \int_{\mathbb S^{d-1}}\rho_P(u)^d\,\mathrm d\sigma(u)
    =
    d\,\mathrm{vol}_d(K_P)^{(d+1)/d}.
\end{align*}
Since $\mathrm{vol}_d(\widehat K_P)=1$, substituting this expression into \eqref{eq:minimal-c-objective} gives the optimal isotropic value
\begin{equation}
    \mathcal J(\widehat K_P,c_{\widehat K_P}^*)
    =
    (d+1)\log\mathrm{vol}_d(K_P)+d.
    \label{eq:optimal-isotropic-value}
\end{equation}
We next compute the optimal anisotropic value. By Theorem~\ref{thm:optimal-starbody}, the optimal star body $K^*$ satisfies $\|u\|_{K^*}=(d\,h_P(u))^{-1/d}$. Since $K^*$ has unit volume, substituting this expression into \eqref{eq:minimal-g-objective} yields
\begin{align}
    \mathcal J(K^*,g_{K^*}^*)
    &=
    -\mathbb E_{P_U}[\log h_P(U)]
    +d\,\mathbb E_{P_U}[\log m_P(U)]
    -(d+1)\log d+d
    \notag\\
    &=
    (d+1)\mathbb E_{P_U}
    \left[
    \log\left(\frac{\rho_P(U)^d}{d\,h_P(U)}\right)
    \right]+d
    \notag\\
    &=
    (d+1)\log\mathrm{vol}_d(K_P)
    -(d+1)D_{\mathrm{KL}}(h_P\|h_{p_{K_P}})+d.
    \label{eq:optimal-anisotropic-value}
\end{align}
The second equality uses \eqref{eq:rho-decomposition}. For the third, the directional-density formula above gives $\rho_P(u)^d/(d\,h_P(u))=\mathrm{vol}_d(K_P)\,h_{p_{K_P}}(u)/h_P(u)$, whose logarithm separates into a volume term and a log-density ratio. Finally, subtracting the optimal anisotropic value \eqref{eq:optimal-anisotropic-value} from the optimal isotropic value \eqref{eq:optimal-isotropic-value} cancels the common terms and yields \eqref{eq:distribution-cost-directional-kl}.
\end{proof}
Theorem~\ref{thm:distribution-isotropy-cost} shows that, up to the dimension-dependent factor $d+1$, the benefit of anisotropic scaling is determined entirely by the target distribution $P$ through the directional discrepancy between $h_P$ and the canonical Gibbs marginal $h_{p_{K_P}}$. This motivates specializing the radial discrepancy ratio to the canonical choice $K=K_P$. The resulting quantity depends only on $P$ and provides an intrinsic diagnostic for determining when anisotropic scaling is beneficial.

\begin{definition}[Diagnostic Ratio of a Distribution]
\label{def:diagnostic-ratio}
For $P\in\mathcal P$, define the diagnostic ratio of $P$ by
\begin{align}
\alpha_P(u)
:=
\alpha_{K_P}(u)
=
\frac{m_P(u)}{\rho_P(u)},
\qquad
u\in\mathbb S^{d-1}.
\label{eq:diagnostic-ratio}
\end{align}
\end{definition}
The diagnostic ratio $\alpha_P$ compares the conditional mean radius of $P$ with the canonical radial geometry $\rho_P$ associated with its optimal Gibbs regularizer. We now show that its logarithmic Jensen gap under $P_U$ determines the isotropy cost of $P$.

\begin{corollary}[Relationship Between the Isotropy Costs]
\label{cor:nullcostequiv}
Let $K_P$ be the canonical star body associated with the radial summary statistic of $P\in\mathcal P$. Then
\begin{equation}
    \Delta_{\mathrm{dist}}(P)
    =
    \frac{d+1}{d}\Delta_{\mathrm{star}}(K_P).
    \label{eq:delta-dist-delta-star-equivalence}
\end{equation}
In particular, $\Delta_{\mathrm{dist}}(P)=0$ if and only if $\alpha_P$ is constant on $\mathbb S^{d-1}$.
\end{corollary}
\begin{proof}[Proof of Corollary~\ref{cor:nullcostequiv}]
By construction, $\rho_{K_P}(u)=\rho_P(u)$. Using \eqref{eq:rho-decomposition} and the polar volume formula, we obtain
\begin{align*}
    \alpha_P(u)
    =
    \frac{m_P(u)}{\rho_P(u)}
    =
    \frac{\rho_P(u)^d}{h_P(u)},\qquad 
    \mathbb E_{P_U}[\alpha_P(U)] =
    \int_{\mathbb S^{d-1}}\rho_P(u)^d\,\mathrm d\sigma(u)
    =
    d\,\mathrm{vol}_d(K_P).
    \notag
\end{align*}
Lemma~\ref{lem:gibbs-polar-densities} therefore gives $h_{p_{K_P}}(u)=\alpha_P(u)h_P(u)/\mathbb E_{P_U}[\alpha_P(U)]$. Substituting this expression into the characterization from Theorem~\ref{thm:distribution-isotropy-cost} yields
\begin{align*}
    \Delta_{\mathrm{dist}}(P)
    &=
    (d+1)\int_{\mathbb S^{d-1}}
    h_P(u)
    \log\left(
    \frac{\mathbb E_{P_U}[\alpha_P(U)]}{\alpha_P(u)}
    \right)\,\mathrm d\sigma(u) \\
    &=
    (d+1)\left(
    \log\mathbb E_{P_U}[\alpha_P(U)]
    -
    \mathbb E_{P_U}[\log\alpha_P(U)]
    \right) \\
    &=
    \frac{d+1}{d}\Delta_{\mathrm{star}}(K_P),
\end{align*}
where the final equality follows from \eqref{eq:delta-star-jensen-gap}. This proves \eqref{eq:delta-dist-delta-star-equivalence}, and the equivalence of the zero-cost conditions follows because the two costs differ by a positive scalar factor. By Theorem~\ref{thm:fixed-K-correction}, these conditions hold precisely when $\alpha_P$ is constant on the sphere.
\end{proof}
Corollary~\ref{cor:nullcostequiv} shows that the isotropy cost of $P$ and that of its canonical star body $K_P$ measure the same underlying directional discrepancy. Both are determined by the logarithmic Jensen gap of the intrinsic diagnostic ratio $\alpha_P$ under $P_U$, differing only by the dimension-dependent factor $(d+1)/d$. Consequently, anisotropic scaling provides a strict improvement precisely when $\alpha_P$ varies across directions. We explore this criterion for several concrete distribution classes in the following subsection. Before doing so, we complete our analysis with a quantitative lower bound on the isotropy cost in terms of the variability of $\alpha_P$, which provides a guarantee on the improvement afforded by anisotropic scaling.

\begin{corollary}[Quantitative Lower Bound on Isotropy Cost]
\label{cor:quant-lower-bound-iso-cost}
For $P\in\mathcal P$, the isotropy cost satisfies
\begin{equation}
    \Delta_{\mathrm{dist}}(P)
    \geq
    \frac{d+1}{2}
    \frac{\mathrm{Var}_{P_U}(\alpha_P(U))}
    {\lVert\alpha_P\rVert_\infty^2},
    \label{eq:isotropy-cost-lower-bound}
\end{equation}
where $\lVert\alpha_P\rVert_\infty:=\sup_{u\in\mathbb S^{d-1}}|\alpha_P(u)|$.
\end{corollary}
\begin{proof}[Proof of Corollary~\ref{cor:quant-lower-bound-iso-cost}]
Let $M:=\lVert\alpha_P\rVert_\infty$ and $Y:=\alpha_P(U)/M$, where $U\sim P_U$. Since $\alpha_P$ is continuous and strictly positive on the sphere, $M$ is finite and positive, and $Y$ is bounded away from zero with $Y\leq1$. For $f(y):=-\log y$, we have $f''(y)=y^{-2}\geq1$ on $(0,1]$, so the standard Jensen-gap bound gives $\log\mathbb E[Y]-\mathbb E[\log Y]\geq\frac12\mathrm{Var}(Y)$; see \cite[Theorem 2.1]{rajba2011probabilistic}. Using \eqref{eq:delta-dist-delta-star-equivalence} and \eqref{eq:delta-star-jensen-gap}, together with the cancellation of $\log M$ in the Jensen gap, we obtain
\begin{align*}
    \Delta_{\mathrm{dist}}(P)
    =
    (d+1)\left(
    \log\mathbb E[Y]-\mathbb E[\log Y]
    \right) \geq
    \frac{d+1}{2}\mathrm{Var}(Y)
    =
    \frac{d+1}{2}
    \frac{\mathrm{Var}_{P_U}(\alpha_P(U))}{M^2}.
\end{align*}
\end{proof}
    
\subsection{Examples Analyzing the Isotropy Cost}\label{sec:examples-isotropy-cost}
We now apply the diagnostic framework developed in Section~\ref{sec:provable-benefits} to several common classes of distributions. Recall from Definition~\ref{def:diagnostic-ratio} that the diagnostic ratio of a target distribution $P$ is
$$
\alpha_P(u)
=
\frac{m_P(u)}{\rho_P(u)},
\qquad
u\in\mathbb S^{d-1},
$$
and Corollary~\ref{cor:nullcostequiv} shows that anisotropic scaling provides a strictly better approximation than isotropic scaling precisely when $\alpha_P$ varies across the sphere. Thus, determining whether anisotropic dilation is beneficial reduces to computing the radial statistics $m_P$ and $\rho_P$, forming their ratio, and analyzing its constancy.

The examples below illustrate this criterion across distributions of increasing geometric complexity. We begin with densities generated by the gauge of a single star body. Since their directional and mean radial structures are coupled through the same star body, we expect the diagnostic ratio to be constant, and thus anisotropic scaling offers no additional approximation benefit. We then consider non-centered elliptical distributions and centered Gaussian mixture models, where translations of the mean and multiple covariance geometries introduce additional radial structure, respectively. 

\paragraph{Densities Induced by Star Bodies} \label{ex: starbody-densities}
    Fix a star body $L\in\mathcal S^d$ and consider a distribution $P$ with density
    \begin{align}
        p(x) = c \cdot \varphi(\|x\|_{L}),
        \label{eq:starbody-densities}
    \end{align}
    where $c>0$ is the normalizing constant and $\varphi:[0,\infty)\to[0,\infty)$ is measurable. Such densities depend only on the gauge $\|x\|_L$ and are therefore constant on each dilation of the boundary of $L$. This class includes the standard centered elliptical families, such as multivariate Gaussian, Student's $t$, and elliptical Laplace distributions, for which $L$ is an ellipsoid associated with the Mahalanobis gauge $\|x\|_L=\sqrt{x^\top\Sigma^{-1}x}$ with $\Sigma\in\mathbb R^{d\times d}$ symmetric positive definite.
 
    Previous work \cite[Example 2]{leong2025optimal} shows that the canonical star body $K_P$ is a positive dilation of the generating body $L$. The following corollary shows that the conditional mean radius has the same directional dependence, resulting in a constant diagnostic ratio.
    
    \begin{corollary}[Constant Diagnostic Ratio of Star Body Densities]
    Let $P\in\mathcal P$ have density of the form \eqref{eq:starbody-densities}. Then
    \begin{align}
        \alpha_P(u)  = \frac{\displaystyle \int_0^\infty s^d\varphi(s)\,\mathrm ds}{\displaystyle\left(\int_0^\infty s^{d-1}\varphi(s)\,\mathrm ds\right) \left(c\int_0^\infty s^d\varphi(s)\,\mathrm ds\right)^{1/(d+1)}} \text{ is constant on } \mathbb S^{d-1}. \label{eq:starbody-diagnostic-ratio}
    \end{align}
    \end{corollary}

    \begin{proof}
    By positive homogeneity of the gauge, $p(ru)=c\,\varphi(r\|u\|_L)$. For $k\geq0$, define $M_k:=\int_0^\infty s^k\varphi(s)\,\mathrm ds$. Since $p$ is a probability density, integrating in polar coordinates and applying the substitution $s=r\|u\|_L$ gives $1=cM_{d-1}\int_{\mathbb S^{d-1}}\|u\|_L^{-d}\,\mathrm d\sigma(u)$, so $M_{d-1}<\infty$. The finite first-moment assumption similarly guarantees $M_d<\infty$.
    
    Applying the same substitution to the directional marginal and unnormalized first radial moment gives
    \begin{align*}
        h_P(u) &= c \int_0^\infty r^{d-1} \varphi(r\|u\|_{L}) \,\mathrm{d}r = c M_{d-1} \|u\|_{L}^{-d}, \\
        \int_0^\infty r^d p(ru) \,\mathrm{d}r &= c \int_0^\infty r^d \varphi(r\|u\|_{L}) \,\mathrm{d}r = c M_d \|u\|_{L}^{-(d+1)}.
    \end{align*}
    The conditional mean radius is the ratio of these two quantities $m_P(u) = M_d(M_{d-1} \|u\|_{L})^{-1}$. The radial summary statistic is the $1/(d+1)$-th root of the unnormalized first radial moment: $\rho_P(u)=(cM_d)^{1/(d+1)}\|u\|_L^{-1}$. Taking their ratio, the common directional factor $\|u\|_L^{-1}$ cancels, so \eqref{eq:starbody-diagnostic-ratio} is constant on $\mathbb S^{d-1}$. Corollary~\ref{cor:nullcostequiv} therefore implies $\Delta_{\mathrm{dist}}(P)=0$.
    \end{proof}

\paragraph{Non-Centered Elliptical Distributions}
    The previous example \eqref{eq:starbody-densities} includes centered elliptical distributions, for which the radial and directional structures are governed by the same underlying ellipsoid. We now examine what happens when this geometry is translated away from the origin. Let $P$ be a non-centered elliptical distribution with density
    \begin{align}
    p(x) = c\,\psi\!\left( (x-\mu)^\top\Sigma^{-1}(x-\mu) \right), \qquad \mu\neq0, \label{eq:noncentered-elliptical-density}
\end{align}
where $\Sigma\in\mathbb R^{d\times d}$ is symmetric positive definite, $c>0$ is the normalizing constant, and
$\psi:[0,\infty)\to[0,\infty)$ is a measurable density generator.
Unlike in the centered case, the displacement $\mu$ introduces an additional direction-dependent term into the radial integrals.
\begin{corollary}[Diagnostic Ratio of Non-Centered Elliptical Distributions]
Let $P\in\mathcal P$ have density of the form \eqref{eq:noncentered-elliptical-density} and finite first moment. Define $\kappa:=\sqrt{\mu^\top\Sigma^{-1}\mu}$ and $z(u):=u^\top\Sigma^{-1}\mu/\sqrt{u^\top\Sigma^{-1}u}$ for $u\in\mathbb S^{d-1}$. For $t\in z(\mathbb S^{d-1})$, define
$$
    F_\psi(t):=
    \frac{\displaystyle
    \left(\int_0^\infty s^d
    \psi\!\left(s^2-2ts+\kappa^2\right)\,\mathrm ds
    \right)^{d/(d+1)}}
    {\displaystyle
    \int_0^\infty s^{d-1}
    \psi\!\left(s^2-2ts+\kappa^2\right)\,\mathrm ds},
$$
where the dependence on the fixed parameters $d$ and $\kappa$ is suppressed in the notation. Then the diagnostic ratio satisfies
\begin{align}
\alpha_P(u)=c^{-1/(d+1)}F_\psi(z(u)),
    \qquad u\in\mathbb S^{d-1}.
    \label{eq:elliptical-diagnostic-ratio}
\end{align}
Hence, all directional dependence of $\alpha_P$ is captured by the single scalar quantity $z(u)$.
\end{corollary}
\begin{proof}
    Let $a(u):=u^\top\Sigma^{-1}u>0$ and $b(u):=u^\top\Sigma^{-1}\mu$. Then $p(ru)=c\,\psi(r^2a(u)-2rb(u)+\kappa^2)$ for $r\geq0$ and $u\in\mathbb S^{d-1}$. For $k\in\{d-1,d\}$ and $t\in z(\mathbb S^{d-1})$, define $I_k(t):=\int_0^\infty s^k\psi(s^2-2ts+\kappa^2)\,\mathrm ds$. Applying the change of variables $s=r\sqrt{a(u)}$ to the directional marginal and unnormalized first radial moment gives

    \begin{align*}
    h_P(u)&=c\,a(u)^{-d/2}I_{d-1}(z(u)),\\
    \int_0^\infty r^d p(ru)\,\mathrm dr
    &=c\,a(u)^{-(d+1)/2}I_d(z(u)).
\end{align*}
Both quantities are finite and strictly positive for every $u$ by the defining properties of $P\in\mathcal P$. In particular, $F_\psi$ is well-defined on $z(\mathbb S^{d-1})$. It follows that $m_P(u)=a(u)^{-1/2}I_d(z(u))/I_{d-1}(z(u))$. The radial summary statistic is $\rho_P(u)=c^{1/(d+1)}a(u)^{-1/2}I_d(z(u))^{1/(d+1)}$. Taking their ratio cancels the common factor $a(u)^{-1/2}$ and yields \eqref{eq:elliptical-diagnostic-ratio}.
\end{proof}
This corollary reduces all directional dependence of the diagnostic ratio to the signed scalar projection $z(u)$ \cite{mahalanobis1936generalized}. The next lemma shows that, for $d\geq2$, $z$ is a continuous surjection from $\mathbb S^{d-1}$ onto $[-\kappa,\kappa]$. This range depends only on $\mu$ and $\Sigma$, independently of the generator $\psi$. Consequently, $\alpha_P$ is constant on $\mathbb S^{d-1}$ if and only if $F_\psi$ is constant on $[-\kappa,\kappa]$. We then illustrate this reduction with the Gaussian generator $\psi(t) = e^{-t/2}$, for which $F_\psi$ has a nonzero derivative at the origin. This establishes nonconstancy of $\alpha_P$ and hence a strictly positive isotropy cost for non-centered Gaussians.

\begin{lemma}[Range of the Mahalanobis Projection]
\label{lem:range-of-z}
Let $d\geq2$, let $\Sigma\in\mathbb R^{d\times d}$ be symmetric positive definite, and let $\mu\in\mathbb R^d\setminus\{0\}$. For the function $z$ defined above,
$$
z(\mathbb S^{d-1})=[-\kappa,\kappa].
$$
Consequently, for any function $F:[-\kappa,\kappa]\to\mathbb R$, the composition $u\mapsto F(z(u))$ is constant on $\mathbb S^{d-1}$ if and only if $F$ is constant on $[-\kappa,\kappa]$.
\end{lemma}

\begin{proof}
    Let $\Sigma^{1/2}$ denote the symmetric positive definite square root of $\Sigma$, and define $v(u):=\Sigma^{-1/2}u/\|\Sigma^{-1/2}u\|_2$ and $w:=\Sigma^{-1/2}\mu$. Then $z(u)=\langle v(u),w\rangle$ and $\kappa = \|w\|_2>0$. The denominator in $v(u)$ is strictly positive on $\mathbb S^{d-1}$, so $v$ and $z$ are continuous.
    
    For any $y\in\mathbb S^{d-1}$, the choice $u=\Sigma^{1/2}y/\|\Sigma^{1/2}y\|_2\in\mathbb S^{d-1}$
    satisfies $v(u)=y$. Hence, $v$ maps $\mathbb S^{d-1}$ surjectively onto itself, and therefore
    $ z(\mathbb S^{d-1}) = \{\langle y,w\rangle:y\in\mathbb S^{d-1}\}.$
    By the Cauchy--Schwarz inequality, this set is contained in $[-\kappa,\kappa]$. Conversely, because $d\geq2$, choose a unit vector $e$ orthogonal to $w$. For any $t\in[-\kappa,\kappa]$, the unit vector
    $$
    y_t = \frac{tw}{\kappa^2} + \sqrt{1-\frac{t^2}{\kappa^2}}\,e
    $$
    satisfies $\langle y_t,w\rangle=t$. Thus every value in $[-\kappa,\kappa]$ is attained, completing the proof.
\end{proof}

\begin{proposition}[Nonzero Isotropy Cost for Non-Centered Gaussians] \label{prop:noncentered-gaussians}
Let $d\geq2$, let $\Sigma\in\mathbb R^{d\times d}$ be symmetric positive definite, and let $P=\mathcal N(\mu,\Sigma)$ with $\mu\neq0$. Then $\alpha_P$ is nonconstant on $\mathbb S^{d-1}$ and consequently $\Delta_{\mathrm{dist}}(P)>0$.

\end{proposition}

\begin{proof}
    The Gaussian density is continuous and strictly positive, and its Gaussian decay ensures that the radial integrals are finite and continuous. Thus $P\in\mathcal P$. Let $\kappa:=\sqrt{\mu^\top\Sigma^{-1}\mu}>0$ and $c:=(2\pi)^{-d/2}(\det\Sigma)^{-1/2}$. For the Gaussian generator $\psi(t):=e^{-t/2}$, define $J_k(t):=\int_0^\infty s^k e^{-s^2/2+ts}\,\mathrm ds$ for $k\geq0$ and $t\in\mathbb R$, and set $R_d(t):=J_d(t)^{d/(d+1)}/J_{d-1}(t)$. Factoring $e^{-\kappa^2/2}$ out of each radial integral defining $F_\psi$, we obtain
    $$
    F_\psi(t)=e^{\kappa^2/(2(d+1))}R_d(t),
    \qquad t\in[-\kappa,\kappa].
    $$ 
    Consequently, \eqref{eq:elliptical-diagnostic-ratio} gives $\alpha_P(u)=(c e^{-\kappa^2/2})^{-1/(d+1)}R_d(z(u))$. The prefactor is positive and independent of $u$, so it suffices to show that $R_d$ is nonconstant on $[-\kappa,\kappa]$.

    We first show that $R_d$ is nonconstant near the origin. Fix $M>0$. For $t\in[-M,M]$, the derivative of the integrand defining $J_k$ is bounded by the integrable function $s^{k+1}e^{-s^2/2+Ms}$. Dominated convergence therefore shows that $J_k$ is continuously differentiable, with $J_k'(t)=J_{k+1}(t)$. Since $R_d(t)>0$, its logarithmic derivative is
    \begin{align*}
    \frac{\mathrm d}{\mathrm dt}\log R_d(t)
    &=\frac{d}{d+1}\frac{J_d'(t)}{J_d(t)}
      -\frac{J_{d-1}'(t)}{J_{d-1}(t)}=\frac{d}{d+1}\frac{J_{d+1}(t)}{J_d(t)}
      -\frac{J_d(t)}{J_{d-1}(t)}.
\end{align*}
    At $t=0$, the gamma integral gives $J_k(0)=2^{(k-1)/2}\Gamma((k+1)/2)$, where $\Gamma$ denotes the gamma function. Applying the recurrence $\Gamma(x+1)=x\Gamma(x)$ at $x=d/2$ yields $J_{d+1}(0)=dJ_{d-1}(0)$. Writing $B_d:=J_d(0)/J_{d-1}(0)$, we obtain
    \begin{align}
    \left.\frac{\mathrm d}{\mathrm dt}\log R_d(t)\right|_{t=0}
    =\frac{d^2-(d+1)B_d^2}{(d+1)B_d}.
    \label{eq: proof-prop18}
\end{align}
    The strict lower bound in Wendel's inequality for gamma ratios \cite{Wendel1948NoteOT}, applied with parameters $x=d/2$ and $s=1/2$, gives $B_d=\sqrt{2}\,\Gamma((d+1)/2)/\Gamma(d/2)>d/\sqrt{d+1}$. Here strictness follows from strict log-convexity of the gamma function. Hence $(d+1)B_d^2>d^2$, so \eqref{eq: proof-prop18} is strictly negative.

    Since $\kappa>0$, the origin lies in the interior of $[-\kappa,\kappa]$. Therefore, $R_d$ is nonconstant on this interval, and by Lemma~\ref{lem:range-of-z}, $\alpha_P$ is nonconstant on $\mathbb S^{d-1}$. Moreover, $\alpha_P$ is continuous and $h_P$ is strictly positive, so $\alpha_P$ cannot be constant $P_U$-almost surely. Corollary~\ref{cor:nullcostequiv} therefore implies $\Delta_{\mathrm{dist}}(P)>0$.
\end{proof}

\paragraph{Centered Gaussian Mixture Models}
    The previous example examines how translation introduces additional directional dependence into the radial integrals forming the diagnostic ratio, suggesting a richer radial geometry relative to the origin. Proposition~\ref{prop:noncentered-gaussians} proves that, for non-centered Gaussians, this leads to a nonconstant diagnostic ratio. We now explore how mixing centered Gaussian components with different covariance matrices can introduce additional directional structure. Let $P=\sum_{j=1}^k w_j\mathcal N(0,\Sigma_j)$ be a $k$-component Gaussian mixture model in $\mathbb R^d$ with density
    \begin{align}
        p(x) = \sum_{j=1}^k w_j \phi (x \mid 0, \Sigma_j), \label{eq:GMM-density}
    \end{align}
    where $d,k\geq1$ are integers, $w_j>0$, $\sum_{j=1}^k w_j=1$, and each $\Sigma_j\in\mathbb R^{d\times d}$ is symmetric positive definite. Here $\phi(\,\cdot\mid 0,\Sigma_j)$ denotes the density of $\mathcal N(0,\Sigma_j)$.

    \begin{corollary}[Diagnostic Ratio of Centered Gaussian Mixtures]
    \label{cor:gmm-diagnostic}
    For $u\in\mathbb S^{d-1}$, let $a_j(u):=u^\top\Sigma_j^{-1}u$ and $|\Sigma_j|:=\det\Sigma_j$. Then $P\in\mathcal P$, and its diagnostic ratio is
    \begin{align}
    \alpha_P(u) = C_d\, \frac{\displaystyle \left(\sum_{j=1}^k w_j|\Sigma_j|^{-1/2} a_j(u)^{-(d+1)/2}\right)^{d/(d+1)}}{\displaystyle \sum_{j=1}^k w_j|\Sigma_j|^{-1/2}a_j(u)^{-d/2}},
    \label{eq: GMM-diagnostic-ratio}
\end{align}
where
$C_d:=2^{(d+2)/(2(d+1))}\pi^{d/(2(d+1))}
\Gamma((d+1)/2)^{d/(d+1)}/\Gamma(d/2)$
and $\Gamma$ denotes the gamma function.
\end{corollary}

\begin{proof}
The mixture density is continuous and strictly positive, and its Gaussian decay ensures that the radial integrals are finite and continuous. Thus $P\in\mathcal P$. By linearity of the integral, the directional marginal and unnormalized first radial moment are weighted sums of the corresponding quantities for the individual components. Standard gamma integral evaluations give
\begin{align*}
    h_P(u)
    &=\frac{\Gamma(d/2)}{2\pi^{d/2}}
    \sum_{j=1}^k w_j|\Sigma_j|^{-1/2}a_j(u)^{-d/2},\\
    \int_0^\infty r^d p(ru)\,\mathrm dr
    &=\frac{\Gamma((d+1)/2)}{\sqrt{2}\,\pi^{d/2}}
    \sum_{j=1}^k w_j|\Sigma_j|^{-1/2}
    a_j(u)^{-(d+1)/2}.
\end{align*}
The conditional mean $m_P(u)$ is the ratio of these two quantities, while $\rho_P(u)$ is the $1/(d+1)$-th power of the second. Taking their ratio yields \eqref{eq: GMM-diagnostic-ratio}.
\end{proof}
The following theorem characterizes exactly when the directional dependence in \eqref{eq: GMM-diagnostic-ratio} cancels.

\begin{theorem}[Constant Diagnostic Ratio of Centered Gaussian Mixtures]
 \label{thm:gaussianmixture}
Let $P=\sum_{j=1}^k w_j\mathcal N(0,\Sigma_j)$ satisfy the assumptions of Corollary~\ref{cor:gmm-diagnostic}. Then
$$
 \alpha_P\text{ is constant on }\mathbb S^{d-1}
    \quad\Longleftrightarrow\quad
    \Sigma_1,\ldots,\Sigma_k\text{ are all proportional}.
$$
Here proportionality means that there exist a symmetric positive definite matrix $\Sigma_0$ and scalars $c_1,\ldots,c_k>0$ such that $\Sigma_j=c_j\Sigma_0$ for every $j$.
\end{theorem}

\begin{proof}
The characterization of constancy requires an analytic continuation argument, which is deferred to Appendix~\ref{proof:gaussianmixture}. Since $\alpha_P$ is continuous and $h_P$ is strictly positive, Corollary~\ref{cor:nullcostequiv} gives the stated criterion for positive isotropy cost.
\end{proof}
Theorem~\ref{thm:gaussianmixture} shows that anisotropic dilations improve the optimal fit for indeed a large set of distributions. Mixtures of Gaussians have been shown to be useful models for multi-modal, complex data distributions, as they can approximate smooth distributions \cite{dalal1983approximating} and are useful models for natural image distributions \cite{chen2021learning, deledalle2018image}. Our result shows that any centered Gaussian mixture whose covariance matrices are not all proportional has a nonconstant diagnostic ratio, ensuring a provable benefit from anisotropic dilation. Geometrically, proportional covariances describe a common ellipsoidal shape at different scales; non-proportional covariances introduce competing directional scales.

For geometric intuition, Figure~\ref{fig:gaussian_samples} illustrates the two-component mixture in $\mathbb R^2$ defined below, whose components concentrate near orthogonal coordinate axes as $\varepsilon$ decreases to zero. While Theorem~\ref{thm:gaussianmixture} establishes when $\Delta_{\mathrm{dist}}(P)>0$, the following corollary gives a quantitative lower bound for this example using Corollary~\ref{cor:quant-lower-bound-iso-cost}.
\begin{figure}[h!]
        \centering
        \includegraphics[width=0.9\linewidth]{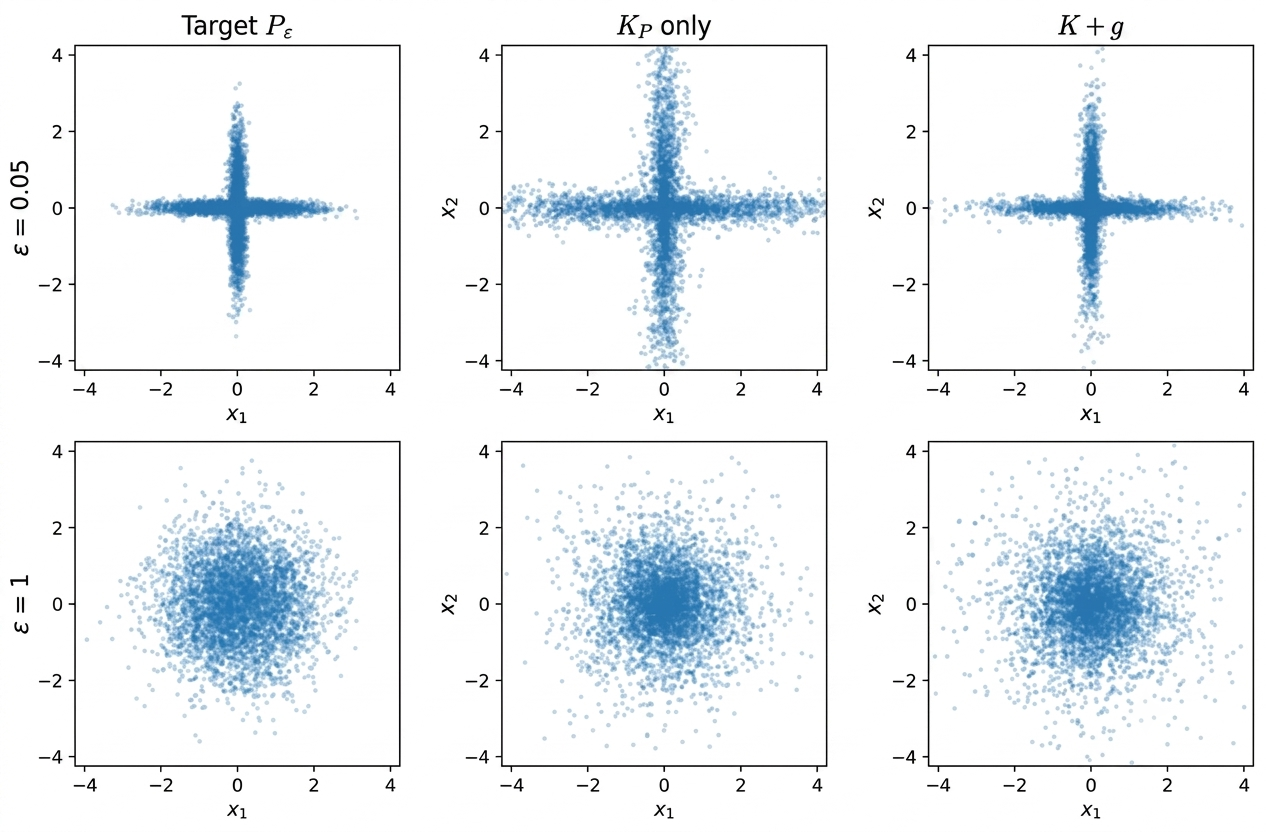}
        \caption{Approximations of a centered Gaussian mixture with varying anisotropy $\varepsilon$. Bottom ($\varepsilon = 1$): Proportional covariances yield a constant $\alpha_P(u)$, rendering the isotropic $K_P$ approximation sufficient. Top ($\varepsilon = 0.05$): Non-proportional covariances cause $\alpha_P(u)$ to fluctuate. The superior performance of the $(K,g)$ sampling highlights the necessity of anisotropic correction.}
        \label{fig:gaussian_samples}
    \end{figure}
\begin{corollary}\label{prop:examplegaussmix}
    Let $P_\varepsilon$ be a two-component mixture in $\mathbb{R}^2$ given by
    $P_\varepsilon = \frac12 \mathcal N(0,\Sigma_{\varepsilon,1}) + \frac12 \mathcal N(0,\Sigma_{\varepsilon,2}),$
    where $\Sigma_{\varepsilon,1} = \mathrm{diag}(1, \varepsilon)$ and $\Sigma_{\varepsilon,2} = \mathrm{diag}(\varepsilon, 1)$ for $\varepsilon > 0$. Then $\Delta_{\mathrm{dist}}(P_{\varepsilon}) =0$  if and only if $\varepsilon=1$. Moreover, there exists a numerical constant $c > 0$ such that  the following lower bound holds for every $0 < \varepsilon \leq 1$: $$\Delta_{\mathrm{dist}}(P_{\varepsilon}) \geq c \frac{\sqrt{\varepsilon}(1-\varepsilon)^4}{(1+\varepsilon)^5}.$$
\end{corollary}
    
\begin{proof}
    By Theorem~\ref{thm:gaussianmixture}, the diagnostic ratio is constant on $\mathbb{S}^1$ if and only if the covariance matrices are proportional. The relation $\Sigma_{\varepsilon,1}=\lambda\Sigma_{\varepsilon,2}$ for some $\lambda>0$ requires $1=\lambda\varepsilon$ and $\varepsilon=\lambda$. For $\varepsilon>0$, these identities hold if and only if $\varepsilon=1$.

    For the quantitative lower bound, fix $0<\varepsilon\leq1$. Corollary~\ref{cor:quant-lower-bound-iso-cost} gives
     \begin{align*}
        \Delta_{\mathrm{dist}}(P_{\varepsilon}) & = 3 (\log \mathbb{E}_{P_{\varepsilon,U}} [\alpha_{P_{\varepsilon}}(U)] - \mathbb{E}_{P_{\varepsilon,U}}[\log \alpha_{P_{\varepsilon}}(U)]) \geq \frac{3}{2}\cdot\frac{\mathrm{Var}_{P_{\varepsilon,U}}(\alpha_{P_{\varepsilon}}(U))}{\|\alpha_{P_{\varepsilon}}\|_{\infty}^2}.
    \end{align*}
    where $P_{\varepsilon,U}$ is the directional marginal of $P_\varepsilon$. Thus it suffices to control the variability of the diagnostic ratio relative to its maximum.

    Parameterize $u_\theta=(\cos\theta,\sin\theta)^\top$ for $\theta\in[0,2\pi)$, and put $a_1(\theta):= u_{\theta}^\top\Sigma_{\varepsilon,1}^{-1}u_{\theta} = \cos^2 \theta + \sin^2 \theta / \varepsilon$ and $a_2(\theta) := u_{\theta}^\top\Sigma_{\varepsilon,2}^{-1}u_{\theta} = \cos^2 \theta / \varepsilon + \sin^2 \theta$. Using the identities $\cos^2 \theta = \frac{1+\cos 2\theta}{2} $ and $\sin^2 \theta = \frac{1-\cos 2\theta}{2}$, note that we can simplify $$a_1(\theta) = \frac{1+\varepsilon}{2\varepsilon}(1 + \delta_{\varepsilon}\cos 2\theta)\ \text{and}\ a_2(\theta) = \frac{1+\varepsilon}{2\varepsilon}(1-\delta_{\varepsilon}\cos 2\theta)$$ where we set $\delta_{\varepsilon} := (\varepsilon-1)/(\varepsilon+1).$ Then we can use \eqref{eq: GMM-diagnostic-ratio} to show that $$\alpha_{P_{\varepsilon}}(u_{\theta}) \propto \left((1+\delta_{\varepsilon}\cos2\theta)^{3/2} + (1-\delta_{\varepsilon}\cos2\theta)^{3/2}\right)^{2/3} =: \phi\left((\delta_{\varepsilon}\cos 2\theta)^2\right)$$ where $\phi(t):= \left((1+\sqrt{t})^{3/2} + (1-\sqrt{t})^{3/2}\right)^{2/3}.$ Note that the multiplicative constant does not appear in the Jensen gap since $\log \mathbb{E}[cY] - \mathbb{E}[\log cY] = \log \mathbb{E}[Y] - \mathbb{E}[\log Y]$ for all $c > 0$. Elementary calculus reveals \begin{align*}
        & 2^{-4/3} \leq \phi'(t) \leq 1/2,\ \quad t \in (0,1], \\
        & 2^{2/3} \leq \phi(t) \leq 2,\ \quad t \in [0,1].
    \end{align*} As a consequence, $\phi$ satisfies the following lower-Lipschitz bound $|\phi(t_1) - \phi(t_2)| \geq 2^{-4/3}|t_1 - t_2|$ for $t_1,t_2\in [0,1]$. Since the ratio $\mathrm{Var}_{P_{\varepsilon,U}}(\alpha_{P_{\varepsilon}}(U)) / \|\alpha_{P_{\varepsilon}}\|_{\infty}^2$ is invariant under multiplication of $\alpha_{P_{\varepsilon}}$, it suffices to analyze the variance of the random variable $Y_{\varepsilon} := \phi(\delta_{\varepsilon}^2 \cos^2 (2\Theta))/ \phi(\delta_{\varepsilon}^2)$. Here, $\Theta$ is drawn from the angular marginal, denoted by $P_{\varepsilon,\Theta}$, so that $u_{\Theta} \sim P_{\varepsilon,U}$ when $\Theta \sim P_{\varepsilon,\Theta}$. Indeed, since $\phi$ is increasing and $0 \leq \cos^2 2\theta \leq 1 $, $\|\alpha_{P_{\varepsilon}}\|_{\infty} = C_{\varepsilon}\phi(\delta_{\varepsilon}^2)$. Hence $Y_{\varepsilon} = \alpha_{P_{\varepsilon}}(u_{\Theta}) / \|\alpha_{P_{\varepsilon}}\|_{\infty}$ so that \begin{align*}
    \frac{\mathrm{Var}_{P_{\varepsilon,U}}(\alpha_{P_{\varepsilon}}(U))}{\|\alpha_{P_{\varepsilon}}\|_{\infty}^2} = \mathrm{Var}_{P_{\varepsilon,U}}\left(\frac{\alpha_{P_{\varepsilon}}(U) }{\|\alpha_{P_{\varepsilon}}\|_{\infty}}\right)= \mathrm{Var}_{P_{\varepsilon,\Theta}}(Y_{\varepsilon}).
    \end{align*} Note that the random variable $\Theta \sim P_{\varepsilon,\Theta}$ has density $h_{\varepsilon}(\theta) = \frac{\sqrt{1-\delta_{\varepsilon}^2}}{2\pi(1-\delta_{\varepsilon}^2\cos^2(2\theta))}$. Drawing an independent copy $\Theta' \sim h_{\varepsilon}(\theta)$ yields \begin{align*}
        \mathrm{Var}_{P_{\varepsilon,\Theta}}(Y_{\varepsilon}) & = \frac{1}{2\phi(\delta_{\varepsilon}^2)^2} \mathbb{E}_{\Theta,\Theta' \sim P_{\varepsilon,\Theta}}\left[\left(\phi(\delta_{\varepsilon}^2 \cos^2(2\Theta)) - \phi(\delta_{\varepsilon}^2 \cos^2(2\Theta'))\right)^2\right] \\
        & \geq \frac{2^{-8/3}}{2\phi(\delta_{\varepsilon}^2)^2}\mathbb{E}_{\Theta,\Theta'\sim P_{\varepsilon,\Theta}}\left[\left(\delta_{\varepsilon}^2 \cos^2(2\Theta) - \delta_{\varepsilon}^2 \cos^2(2\Theta')\right)^2\right] \\ 
        & \geq 2^{-17/3}\delta_{\varepsilon}^4\mathrm{Var}_{P_{\varepsilon,\Theta}}(\cos^2(2\Theta))
    \end{align*} where in the first inequality we used the lower-Lipschitz bound and in the last inequality, we used $\phi \leq 2.$ Finally, since the density $h_{\varepsilon}(\theta) \geq \sqrt{1-\delta_{\varepsilon}^2}/(2\pi)$ for all $\theta$, note that we have \begin{align*}
        \mathrm{Var}_{P_{\varepsilon,\Theta}}(\cos^2(2\Theta)) \geq \sqrt{1-\delta_{\varepsilon}^2}/(2\pi)\mathrm{Var}_{\mathrm{Unif}[0,2\pi]}(\cos^2(2\Theta)).
    \end{align*} An elementary calculation gives $\mathrm{Var}_{\mathrm{Unif}[0,2\pi]}(\cos^2(2\Theta)) = \frac{1}{2\pi}\int_0^{2\pi}(\cos^2(2\theta) - 1/2)^2\mathrm{d}\theta= 1/8$. Finally, since $\sqrt{1-\delta_{\varepsilon}^2} = 2\sqrt{\varepsilon}/(\varepsilon+1)$, we have shown that there exists a universal constant $c > 0$ such that \begin{align*}
        \mathrm{Var}_{P_{\varepsilon,\Theta}}(Y_{\varepsilon}) \geq c \delta_{\varepsilon}^4 \sqrt{1- \delta_{\varepsilon}^2} \ge c\frac{\sqrt{\varepsilon}(\varepsilon-1)^4}{(\varepsilon + 1)^5}.
    \end{align*}
\end{proof}

\subsection{Generalized Gibbs Orbits}
We now extend the previous analysis to generalized Gibbs densities associated with regularizers that are positively homogeneous of degree $\tau>0$. For a star body $K\in\mathcal S^d$, the regularizer $x\mapsto\|x\|_K^\tau$ satisfies $\|ru\|_K^\tau=r^\tau\|u\|_K^\tau$ for $r>0$ and $u\in\mathbb S^{d-1}$. This generalization introduces an additional degree of freedom in how radial magnitude is penalized, while preserving the underlying directional geometry induced by $K$. For fixed $\tau>0$ and $K\in\mathcal S^d$, define
$$
    p_{K,\tau}(x)
    :=
    \frac{e^{-\|x\|_K^\tau}}{Z_{K,\tau}},
    \qquad
    Z_{K,\tau}
    :=
    \int_{\mathbb R^d}e^{-\|x\|_K^\tau}\,\mathrm dx.
$$
For $g\in G$, write $Q_{K,g,\tau}:=(\mathcal T_g)_\#p_{K,\tau}$. The corresponding family of generalized Gibbs orbits is
\begin{equation}
    \mathcal D_{\mathrm{orb},\tau}
    :=
    \bigcup_{K\in\mathcal S^d}[\,p_{K,\tau}\,]_{\mathcal T}
    =
    \left\{
    Q_{K,g,\tau}
    \mid K\in\mathcal S^d,\; g\in G
    \right\}.
    \label{eq:generalized-gibbs-orbit}
\end{equation}
As in the $\tau=1$ case, the unit-volume constraint in the projection below removes the scaling ambiguity between $K$ and $g$ without changing this family.

For regularizers that are positively homogeneous of degree $\tau$, prior work \cite{leong2025optimal} characterizes the optimal regularizer, under suitable regularity assumptions, through a generalization of the radial summary statistic. For a target distribution $P\in\mathcal P$ with density $p$, we define the \emph{generalized radial summary statistic} (GRSS) by
\begin{equation}
    \rho_{P,\tau}(u)
    :=
    \left(
    \int_0^\infty r^{d+\tau-1}p(ru)\,\mathrm dr
    \right)^{1/(d+\tau)},
    \qquad u\in\mathbb S^{d-1}.
    \label{eq:generalized-RSS}
\end{equation}
The directional density $h_P(u)=\int_0^\infty r^{d-1}p(ru)\,\mathrm dr$ does not involve $\tau$ and, as before, is invariant under anisotropic dilation. Because the radial penalty is now governed by $r^\tau$, the relevant conditional statistic becomes the conditional $\tau$-moment of the radius. Writing $X=RU\sim P$ in polar coordinates, define
\begin{equation}
    m_{P,\tau}(u)
    :=
    \mathbb E_P[R^\tau\mid U=u].
    \label{eq:expected-alpha-moment}
\end{equation}
Throughout this subsection, we additionally assume $m_{P,\tau}\in G$. Using the conditional radial density, we obtain an analogue of~\eqref{eq:rho-decomposition} for degree $\tau$:
$$
    m_{P,\tau}(u)
    =
    \frac{1}{h_P(u)}
    \int_0^\infty r^{d+\tau-1}p(ru)\,\mathrm dr
    =
    \frac{\rho_{P,\tau}(u)^{d+\tau}}{h_P(u)}.
$$
This identity connects the conditional radial moment, the directional marginal, and the GRSS. Since $h_P\in G$, the assumption $m_{P,\tau}\in G$ is equivalent to $\rho_{P,\tau}\in G$.

The following theorem shows that projecting $P$ onto the generalized Gibbs orbit family yields a decoupled solution analogous to the $\tau=1$ case: the inner optimization matches the conditional radial $\tau$-moments, while the outer optimization matches the directional marginal. In particular, the optimal outer star body is independent of $\tau$.

\begin{theorem}[Solution to Generalized-Orbit Projection]
\label{thm:generalized-moment-projection}
Fix $P\in\mathcal P$ and $\tau>0$. Suppose $m_{P,\tau}\in G$ and consider the KL projection problem
\begin{equation}
    \inf_{\substack{K\in\mathcal S^d\\ \mathrm{vol}_d(K)=1}}
    \left(
    \inf_{g\in G}
    D_{\mathrm{KL}}(P\|Q_{K,g,\tau})
    \right).
    \label{eq:alpha-inner-outer}
\end{equation}
For fixed $K\in\mathcal S^d$, the unique optimal inner profile is
\begin{equation}
    g^*_{K,\tau}(u)
    =
    \left(
    \frac{\tau}{d}m_{P,\tau}(u)
    \right)^{1/\tau}
    \|u\|_K,
    \qquad u\in\mathbb S^{d-1},
    \label{eq:alpha-optimal-g}
\end{equation}
which matches the conditional radial $\tau$-moment of the transformed base distribution to that of $P$. The unique optimal outer star body is the same $K^*$ as in Theorem~\ref{thm:optimal-starbody}, defined by
$$
\rho_{K^*}(u) = (d\,h_P(u))^{1/d}, \qquad u\in\mathbb S^{d-1}.
$$
\end{theorem}

\begin{proof}[Proof of Theorem~\ref{thm:generalized-moment-projection}]
The proof follows the inner--outer structure of Theorems~\ref{thm:optimal-g-K} and \ref{thm:optimal-starbody}. We begin with the inner optimization over $g\in G$. By the pushforward density formula, the density $q_{K,g,\tau}$ of $Q_{K,g,\tau}$ satisfies
$$
    q_{K,g,\tau}(ru)
    =
    \frac{1}{Z_{K,\tau}}g(u)^{-d}
    \exp\left(
    -\frac{r^\tau\|u\|_K^\tau}{g(u)^\tau}
    \right),
    \qquad r>0,\; u\in\mathbb S^{d-1}.
$$
Since $m_{P,\tau}\in G$, we have $\mathbb E_P[R^\tau]=\mathbb E_{P_U}[m_{P,\tau}(U)]<\infty$. The boundedness of $g$, $g^{-1}$, and the gauge on the sphere therefore gives $\mathbb E_P[|\log q_{K,g,\tau}(X)|]<\infty$. Together with the standing integrability assumption on $p$, this ensures that the KL divergence is finite and that minimizing it is equivalent to minimizing the cross-entropy. Conditioning on $U$ gives
$$
    \mathbb E_P[-\log q_{K,g,\tau}(X)]
    =
    \log Z_{K,\tau}
    +
    \mathbb E_{P_U}
    \left[
    d\log g(U)
    +
    \frac{\|U\|_K^\tau m_{P,\tau}(U)}
    {g(U)^\tau}
    \right].
$$
For fixed $K$, the normalizing term is independent of $g$, so we minimize the integrand pointwise. For fixed $u\in\mathbb S^{d-1}$, the derivative of $d\log t+\|u\|_K^\tau m_{P,\tau}(u)t^{-\tau}$ with respect to $t>0$ has the same sign as $d\,t^\tau-\tau\|u\|_K^\tau m_{P,\tau}(u)$. It is therefore negative for $t<g^*_{K,\tau}(u)$ and positive for $t>g^*_{K,\tau}(u)$, so \eqref{eq:alpha-optimal-g} gives the unique pointwise minimizer. Since $m_{P,\tau}\in G$ and $u\mapsto\|u\|_K$ is continuous and strictly positive, $g^*_{K,\tau}\in G$. Integrating the pointwise inequality establishes optimality. Any other continuous profile differs from $g^*_{K,\tau}$ on a nonempty open set, which has positive $P_U$-measure because $h_P>0$. The integrated inequality is then strict, proving uniqueness.

To verify the moment-matching interpretation, Gamma integration gives $\mathbb E_{p_{K,\tau}}[R^\tau\mid U=u]=(d/\tau)\rho_K(u)^\tau$. Under $\mathcal T_g$, this conditional moment is multiplied by $g(u)^\tau$. Substituting \eqref{eq:alpha-optimal-g} therefore yields the target moment $m_{P,\tau}(u)$ in every direction.

We now turn to the outer optimization over $K$. At the optimal profile, the ratio term in the integrand equals $d/\tau$, while $\log g^*_{K,\tau}(u)=\log\|u\|_K+\tau^{-1}\log((\tau/d)m_{P,\tau}(u))$. Discarding terms independent of $K$ and using $\log\|u\|_K=-\log\rho_K(u)$ reduces the outer problem to
\begin{equation}
    \inf_{\substack{K\in\mathcal S^d\\ \mathrm{vol}_d(K)=1}}
    \left\{
    \log Z_{K,\tau}
    -
    d\,\mathbb E_{P_U}[\log\rho_K(U)]
    \right\}.
    \label{eq:outer-alpha-obj}
\end{equation}
To evaluate the normalizer and identify the directional marginal, use the substitution $t=r^\tau\|u\|_K^\tau$ in the Gamma integral:
\begin{align*}
    h_{p_{K,\tau}}(u)
    &=
    \frac{1}{Z_{K,\tau}}
    \int_0^\infty
    r^{d-1}e^{-r^\tau\|u\|_K^\tau}\,\mathrm dr =
    \frac{\Gamma(d/\tau)}
    {\tau Z_{K,\tau}}\rho_K(u)^d,
\end{align*}
where $\Gamma$ denotes the Gamma function. Integrating over the sphere and using the polar volume formula yields $Z_{K,\tau}=(d/\tau)\Gamma(d/\tau)\mathrm{vol}_d(K)$. Consequently, $h_{p_{K,\tau}}(u)=\rho_K(u)^d/(d\,\mathrm{vol}_d(K))$. For unit-volume star bodies, $Z_{K,\tau}$ is independent of $K$, and the directional density simplifies to $h_{p_{K,\tau}}(u)=\rho_K(u)^d/d$, independently of $\tau$.

For unit-volume $K$, the identity $-d\log\rho_K(u)=-\log h_{p_{K,\tau}}(u)-\log d$ shows that minimizing \eqref{eq:outer-alpha-obj} is equivalent to minimizing the directional cross-entropy. This cross-entropy differs from $D_{\mathrm{KL}}(h_P\|h_{p_{K,\tau}})$ by the finite term $-\mathbb E_{P_U}[\log h_P(U)]$, which is independent of $K$. Thus, the two objectives have the same minimizers.

Because $h_P\in G$, the proposed radial function $\rho_{K^*}=(d\,h_P)^{1/d}$ defines a star body. The polar volume formula gives $\mathrm{vol}_d(K^*)=\int_{\mathbb S^{d-1}}h_P(u)\,\mathrm d\sigma(u)=1$, so $K^*$ is admissible. Its directional density equals $h_P$, and hence it attains the minimum directional KL divergence of zero. Conversely, any minimizer must satisfy $h_{p_{K,\tau}}=h_P$ $\sigma$-almost everywhere. By continuity, this equality holds everywhere on the sphere, giving $\rho_K(u)=(d\,h_P(u))^{1/d}$ and therefore $K=K^*$. This proves uniqueness, and the formula for $K^*$ is independent of $\tau$.
\end{proof}
\subsection{Orbits of Radial Exponential Families}\label{sec:exponential-family-model}
We have established that for both Gibbs and generalized Gibbs orbits, the inner KL projection matches a specific conditional radial $\tau$-moment, while the outer KL projection matches the directional marginal. In particular, both KL projection problems \eqref{eq:inner-and-outer-problems} and \eqref{eq:alpha-inner-outer} have the same outer solution $K^*$, with radial function $\rho_{K^*}(u)=(d\,h_P(u))^{1/d}$. By Theorem~\ref{thm:optimal-starbody}, this star body is the unique solution to the directional marginal problem \eqref{eq:simplified-outer-opt-problem}. This suggests that differences in the expressivity of these models lie entirely in their conditional radial laws. The following proposition shows that, for fixed $\tau$, the conditional radial laws in generalized Gibbs orbits belong to a generalized Gamma family with fixed shape. 

\begin{proposition}[Polar Decomposition of Generalized Gibbs Orbits]
\label{prop:polardecomp-gibbs}
Fix $K\in\mathcal S^d$, $\tau>0$, and $g\in G$. For a random vector $X=RU\sim Q_{K,g,\tau}$, the directional marginal is that of the base generalized Gibbs density:
$$
    h_{Q_{K,g,\tau}}(u)
    =
    h_{p_{K,\tau}}(u)
    =
    \frac{\rho_K(u)^d}{d\,\mathrm{vol}_d(K)},
    \qquad u\in\mathbb S^{d-1}.
$$
Furthermore, the conditional radial law satisfies
$$
    \left.
    \left(\frac{R}{g(u)\rho_K(u)}\right)^\tau
    \,\right|\, U=u
    \sim
    \operatorname{Gamma}\left(\frac d\tau,1\right),
$$
where the Gamma parameters are shape and scale, respectively.
\end{proposition}

\begin{proof}
The directional marginal follows from its invariance under anisotropic dilation and the normalizing constant computed in the preceding subsection. Writing $x=ru$, the joint polar density is $f_{R,U}(r,u)=r^{d-1}q_{K,g,\tau}(ru)$. Dividing by the directional marginal gives
$$
    f_{R\mid U}(r\mid u)
    =
    \frac{\tau r^{d-1}}
    {\Gamma(d/\tau)(g(u)\rho_K(u))^d}
    \exp\left(
    -\left(\frac{r}{g(u)\rho_K(u)}\right)^\tau
    \right),
    \qquad r>0.
$$
For fixed $u$, the substitution $t=(r/(g(u)\rho_K(u)))^\tau$ yields the density $t^{d/\tau-1}e^{-t}/\Gamma(d/\tau)$ for $t>0$, proving the claimed Gamma law.
\end{proof}
For fixed $\tau$, this proposition shows that $K$ and $g$
affect the conditional radial law only through its scale
$g(u)\rho_K(u)$. This limitation motivates our construction
of a conditional radial exponential family that offers
greater flexibility in radial shape while preserving our
decoupled geometric framework. Choosing powers of the radius as sufficient statistics allows us to incorporate higher-order radial moment information while retaining a convex log-partition function whose derivatives encode these moments \cite{wainwright2008graphical}. Furthermore, recent work has used neural networks to learn sufficient statistics for expressive energy-based generative models
\cite{abdulsalam2025exponential} and natural-gradient updates on exponential-family manifolds to guide the training of time-varying generative models \cite{liu2025guiding}. These developments further motivate our choice of exponential families as a basis for modeling more expressive conditional radial laws.

\paragraph{Radial Exponential Base Densities and Their Orbits} For each unit-volume star body $K\in\mathcal S^d$, we retain the directional density $h_K(u):=\rho_K(u)^d/d$ and choose the conditional radial law from an exponential family. Fix an integer $L\geq1$ and exponents $0<\tau_1<\cdots<\tau_L$. Consider the coefficient vector $\eta=(\eta_1,\ldots,\eta_L)\in\Theta:=\mathbb R^{L-1}\times(0,\infty)$, and define the radial density 
$$
    f_\eta(r)
    :=
    \frac{r^{d-1}}{A(\eta)}
    \exp\left(-\sum_{\ell=1}^L\eta_\ell r^{\tau_\ell}\right),
    \qquad r>0,
$$
where $A(\eta):=\int_0^\infty r^{d-1}\exp(-\sum_{\ell=1}^L\eta_\ell r^{\tau_\ell})\,\mathrm dr$ is the normalizing constant. We write $\mathbb E_\eta$ for expectation under $R\sim f_\eta$.
For any moment order $s>0$, the corresponding radial moment is
$$
    \mathbb E_\eta[R^s]
    =
    \frac{1}{A(\eta)}
    \int_0^\infty
    r^{d+s-1}
    \exp\left(-\sum_{\ell=1}^L\eta_\ell r^{\tau_\ell}\right)
    \,\mathrm dr.
$$
On every compact subset of $\Theta$, the leading coefficient
$\eta_L$ is bounded away from zero and the remaining coefficients
are bounded. For both the normalizing integral and the moment
numerator, the highest-power term provides an integrable tail
bound for each parameter derivative, while near zero these
derivatives are bounded by a constant multiple of $r^{d-1}$.
Differentiation under the integral therefore shows that $A$ and
$\eta\mapsto\mathbb E_\eta[R^s]$ are smooth on $\Theta$.

To allow the radial law to vary with direction, take a continuous coefficient field $\eta:\mathbb S^{d-1}\to\Theta$ and let $Q_{K,\eta}^{\mathrm{base}}$ be the distribution of $X=RU$ whose joint polar density, with respect to $\mathrm dr\,\mathrm d\sigma(u)$, is
$$
    q_{K,\eta}^{\mathrm{base}}(r,u)
    :=
    h_K(u)f_{\eta(u)}(r),
    \qquad r>0,\;u\in\mathbb S^{d-1}.
$$
We generate its orbit by setting $Q_{K,\eta,g}:=(\mathcal T_g)_\#Q_{K,\eta}^{\mathrm{base}}$ for $g\in G$. Since $\mathcal T_g$ preserves direction and maps $r$ to $g(u)r$, the scaled conditional density is $g(u)^{-1}f_{\eta(u)}(r/g(u))$ by Theorem~\ref{thm:structure-of-orbit}, or explicitly,
\begin{equation}
    f_{\eta(u),g(u)}(r\mid u)
    =
    \frac{r^{d-1}}{g(u)^dA(\eta(u))}
    \exp\left(
    -\sum_{\ell=1}^L
    \eta_\ell(u)
    \left(\frac{r}{g(u)}\right)^{\tau_\ell}
    \right),
    \qquad r>0.
    \label{eq:separated-orbit-density}
\end{equation}
The resulting joint polar density is $q_{K,\eta,g}(r,u)=h_K(u)f_{\eta(u),g(u)}(r\mid u)$. In the following analysis, we study KL projection onto the union of the orbits $[\,Q_{K,\eta}^{\mathrm{base}}\,]_{\mathcal T}$ as $K$ and $\eta$ vary. Before doing so, we address the scaling ambiguity between $\eta$ and the external profile $g$. 

\begin{remark}[Identifiability via Baseline Moment Constraint]
\label{rem:identifiability}
Fix $K$. Without constraint, the coefficient field $\eta$ and an external scaling profile $g\in G$ are not jointly identifiable. Indeed, for any $\widetilde g\in G$, define $\widetilde\eta_\ell(u):=\eta_\ell(u)\widetilde g(u)^{-\tau_\ell}$ for $\ell=1,\ldots,L$. This again defines a continuous coefficient field $\widetilde\eta:\mathbb S^{d-1}\to\Theta$. At each fixed $u$, the substitution $r=\widetilde g(u)s$ in the normalizing integral gives $A(\widetilde\eta(u))=\widetilde g(u)^dA(\eta(u))$, and hence
$$
f_{\widetilde\eta(u)}(r) = \frac{1}{\widetilde g(u)} f_{\eta(u)}\left(\frac{r}{\widetilde g(u)}\right).
$$
This is the density of $\widetilde g(u)R$ when $R\sim f_{\eta(u)}$, so the parameter mapping $\eta\mapsto\widetilde\eta$ scales the conditional radius by $\widetilde g(u)$. Simultaneously mapping the external profile $g\mapsto g/\widetilde g$ compensates for this scaling and leaves the model distribution unchanged:
$$    (\mathcal T_{g/\widetilde g})_\#
    Q_{K,\widetilde\eta}^{\mathrm{base}}
    =
    (\mathcal T_g)_\#
    Q_{K,\eta}^{\mathrm{base}}.
$$
When $\widetilde g\not\equiv1$, the distinct parameter pairs $(\eta,g)$ and $(\widetilde\eta,g/\widetilde g)$ therefore produce the same model. Thus, $\eta$ can absorb radial scaling that would otherwise be represented by the external profile $g$. Conversely, if $(\eta,g)$ and $(\widehat\eta,\widehat g)$ define the same model, their base laws differ by the scale ratio $\widetilde g:=g/\widehat g$. Linear independence of $1,r^{\tau_1},\ldots,r^{\tau_L}$, together with positivity of $h_K$ and continuity of the coefficient fields, then gives $\widehat\eta_\ell(u)=\eta_\ell(u)\widetilde g(u)^{-\tau_\ell}$ for every direction $u$ and $\ell=1,\ldots,L$.

We separate the roles of $\eta$ and $g$ by imposing the baseline moment constraint
\begin{equation}
    \mathbb E_{\eta(u)}[R^{\tau_1}]
    =
    \frac{d}{\tau_1}\rho_K(u)^{\tau_1},
    \qquad u\in\mathbb S^{d-1},
    \label{eq:baseline-constraint}
\end{equation}
which matches the conditional $\tau_1$-moment of the base radial family to that of the generalized Gibbs distribution with exponent $\tau_1$. If both $\eta$ and $\widetilde\eta$ satisfy \eqref{eq:baseline-constraint}, then the relationship $\mathbb E_{\widetilde\eta(u)}[R^{\tau_1}]=  \widetilde g(u)^{\tau_1}\mathbb E_{\eta(u)}[R^{\tau_1}]$, obtained by taking the $\tau_1$-moment of $\widetilde g(u)R$ with $R\sim f_{\eta(u)}$, forces $\widetilde g(u)=1$. Consequently, the two parameter pairs coincide, establishing joint identifiability for fixed $K$.
\end{remark}
When $L=1$, we have $\mathbb E_{\eta(u)}[R^{\tau_1}]=d/(\tau_1\eta_1(u))$, so the baseline constraint forces $\eta_1(u)=\rho_K(u)^{-\tau_1}$. Thus, ordinary and generalized Gibbs orbits are recovered by taking $\tau_1=1$ and $\tau_1=\tau$, respectively. In this case, the baseline constraint fixes $\eta_1$ through $K$, while $g$ remains the external radial scaling profile. For $L>1$, the radial energy can combine distinct powers, such as $r$ and $r^2$, to modify the conditional radial shape.

\paragraph{Projection onto the Separated Radial Orbit}
We now study KL projection onto this orbit family. Let $P\in\mathcal P$ and assume that its conditional radial moments $m_{P,\tau_\ell}(u):=\mathbb E_P[R^{\tau_\ell}\mid U=u]$ belong to $G$ for $\ell=1,\ldots,L$. For each unit-volume star body $K\in\mathcal S^d$, admissible pairs consist of $g\in G$ and continuous fields $\eta:\mathbb S^{d-1}\to\Theta$ satisfying \eqref{eq:baseline-constraint}.

\begin{proposition}[KL Decomposition]
\label{prop:KL-decomposition}
Let $P\in\mathcal P$ satisfy the moment assumptions above, and write $P_{R,U}=P_U P_{R\mid U}$ for its polar decomposition. For every unit-volume star body $K\in\mathcal S^d$ and admissible pair $(\eta,g)$, the KL divergence to the pushforward model is finite and decomposes as
$$
    D_{\mathrm{KL}}(P\|Q_{K,\eta,g})
    =
    D_{\mathrm{KL}}(h_P\|h_K)
    +
    \mathbb E_{P_U}\left[
    D_{\mathrm{KL}}\left(
    P_{R\mid U}
    \,\middle\|\,
    f_{\eta(U),g(U)}(\cdot\mid U)
    \right)
    \right],
$$
where densities are identified with their associated probability measures.
\end{proposition}
\begin{proof}
We first show that this KL divergence is finite. Let $\bar q$ denote the Lebesgue density of $Q_{K,\eta,g}$. Dividing the joint polar density by $r^{d-1}$ and taking logarithms gives
\begin{equation}
        \log\bar q(ru)
    =
    \log h_K(u)
    -
    d\log g(u)
    -
    \log A(\eta(u))
    -
    \sum_{\ell=1}^L
    \eta_\ell(u)
    \left(\frac{r}{g(u)}\right)^{\tau_\ell}. \label{eq:log-expfam-density}
\end{equation}
Since $A$ is positive and continuous on $\Theta$, the functions $h_K$, $g$, and $A\circ\eta$ are positive and continuous on the compact sphere, so their logarithms are bounded. The coefficients $\eta_\ell(u)/g(u)^{\tau_\ell}$ are also bounded. Using $r^{\tau_\ell}\leq1+r^{\tau_L}$ for $r>0$, we therefore obtain $|\log\bar q(ru)|\leq C(1+r^{\tau_L})$ for some finite constant $C$ depending on $K$, $\eta$, and $g$. Since $m_{P,\tau_L}\in G$ is bounded, we have $\mathbb E_P[R^{\tau_L}]=\mathbb E_{P_U}[m_{P,\tau_L}(U)]<\infty$, and hence $\mathbb E_P[|\log\bar q(X)|]<\infty$. Together with the standing assumption $\mathbb E_P[|\log p(X)|]<\infty$, this shows that $D_{\mathrm{KL}}(P\|Q_{K,\eta,g})=\mathbb E_P[\log p(X)-\log\bar q(X)]$ is finite.

Finally, the polar-coordinate map is a measurable bijection away from the origin, which has probability zero under both distributions. The decomposition follows from invariance of KL divergence and the chain rule for relative entropy applied to the joint polar distributions \cite{polyanskiy_wu_info_theory}.
\end{proof}
\noindent
For fixed $K$, the directional term $D_{\mathrm{KL}}(h_P\|h_K)$ is constant with respect to $(\eta,g)$, leaving the residual radial problem
\begin{equation}
    \inf_{g,\eta}
    \mathbb E_{P_U}\left[
    D_{\mathrm{KL}}\left(
    P_{R\mid U}
    \,\middle\|\,
    f_{\eta(U),g(U)}(\cdot\mid U)
    \right)
    \right],
    \label{eq:separated-KL}
\end{equation}
where the infimum is over the admissible pairs defined above. The following theorem characterizes its solution when a member of the radial exponential family can match all prescribed moments in each direction.

\begin{theorem}[Optimal Radial Parameters]
\label{thm:optimal-eta-g}
Fix a unit-volume star body $K\in\mathcal S^d$ and exponents $0<\tau_1<\cdots<\tau_L$. Let $P\in\mathcal P$ satisfy $m_{P,\tau_\ell}\in G$ for $\ell=1,\ldots,L$. Suppose that for every $u\in\mathbb S^{d-1}$ there exists $a\in\Theta$ satisfying $\mathbb E_a[R^{\tau_\ell}]=m_{P,\tau_\ell}(u)$ for $\ell=1,\ldots,L$. Then \eqref{eq:separated-KL}, subject to the baseline constraint \eqref{eq:baseline-constraint}, has a unique admissible minimizer $(\eta^*,g^*)$. The optimal scale profile is
$$
    g^*(u)
    =
    \left(
    \frac{\tau_1m_{P,\tau_1}(u)}{d}
    \right)^{1/\tau_1}
    \|u\|_K,
    \qquad u\in\mathbb S^{d-1},
$$
and the optimal shape parameter is uniquely determined by the
scaled moment-matching equations
$$
    \mathbb E_{\eta^*(u)}[R^{\tau_\ell}]
    =
    g^*(u)^{-\tau_\ell}m_{P,\tau_\ell}(u),
    \qquad \ell=1,\ldots,L.
$$
\end{theorem}

\begin{proof}
Let $\widehat\eta_\ell(u):=\eta_\ell(u)g(u)^{-\tau_\ell}$. By Remark~\ref{rem:identifiability}, the conditional density is $f_{\widehat\eta(u)}$, and $A(\widehat\eta(u))=g(u)^dA(\eta(u))$. Using \eqref{eq:log-expfam-density}, the part of the objective depending on $\widehat\eta$ is $\mathbb E_{P_U}[F_U(\widehat\eta(U))]$, where
$$
F_u(a):=\log A(a)+\sum_{\ell=1}^L a_\ell m_{P,\tau_\ell}(u)
$$ 
for $a\in\Theta$. By the smoothness established above, the derivative identities for exponential families \cite[Proposition 3.1]{wainwright2008graphical} give
\begin{align*}
    \frac{\partial F_u}{\partial a_\ell}(a)
    &=
    m_{P,\tau_\ell}(u)-\mathbb E_a[R^{\tau_\ell}],\\
    \frac{\partial^2F_u}{\partial a_j\partial a_\ell}(a)
    &=
    \operatorname{Cov}_a(R^{\tau_j},R^{\tau_\ell}),
    \qquad j,\ell=1,\ldots,L.
\end{align*}
The covariance matrix is positive definite because $f_a$ is strictly positive on $(0,\infty)$ and $1,r^{\tau_1},\ldots,r^{\tau_L}$ are linearly independent, confirming that $F_u$ is strictly convex on $\Theta$. By assumption, there exists a parameter $\widehat\eta^*(u)\in\Theta$ satisfying $\mathbb E_{\widehat\eta^*(u)}[R^{\tau_\ell}]=m_{P,\tau_\ell}(u)$ for $\ell=1,\ldots,L$. Hence $\nabla F_u(\widehat\eta^*(u))=0$, and strict convexity makes $\widehat\eta^*(u)$ the unique global minimizer. The Jacobian of the moment map $a\mapsto (\mathbb E_a[R^{\tau_\ell}])_{\ell=1}^L$ is the negative of this covariance matrix and is therefore invertible. The implicit function theorem gives locally smooth matching coefficients as functions of the target moments, and uniqueness makes these local solutions agree. Since the target moments are continuous in $u$, the field $u\mapsto\widehat\eta^*(u)$ is continuous.

To represent $\widehat\eta^*$ by a pair satisfying \eqref{eq:baseline-constraint}, the $\tau_1$-moment must obey
\begin{align*}
    m_{P,\tau_1}(u)
    &=
    \mathbb E_{\widehat\eta^*(u)}[R^{\tau_1}]
    =g(u)^{\tau_1}\mathbb E_{\eta(u)}[R^{\tau_1}] =g(u)^{\tau_1}\frac{d}{\tau_1}\rho_K(u)^{\tau_1}.
\end{align*}
Its unique positive solution is the stated profile $g^*$. Define $\eta_\ell^*(u):=\widehat\eta_\ell^*(u)g^*(u)^{\tau_\ell}$. Then $g^*\in G$, $\eta^*$ is a continuous field taking values in $\Theta$, and
$$
    \mathbb E_{\eta^*(u)}[R^{\tau_\ell}]
    =
    g^*(u)^{-\tau_\ell}
    \mathbb E_{\widehat\eta^*(u)}[R^{\tau_\ell}]
    =
    g^*(u)^{-\tau_\ell}m_{P,\tau_\ell}(u).
$$
For $\ell=1$, the formula for $g^*$ makes this expectation equal to $(d/\tau_1)\rho_K(u)^{\tau_1}$, verifying the baseline constraint.

For any admissible pair with coefficient field $\widehat\eta_\ell=\eta_\ell g^{-\tau_\ell}$, pointwise optimality gives $\mathbb E_{P_U}[F_U(\widehat\eta(U))-F_U(\widehat\eta^*(U))]\geq0$. If $\widehat\eta$ differs from $\widehat\eta^*$ anywhere, continuity makes them differ on a nonempty open set. Since $h_P>0$ and the pointwise minimizer is unique, the inequality is then strict. Finally, Remark~\ref{rem:identifiability} makes the representation by $(\eta,g)$ unique, proving uniqueness of the admissible minimizer.
\end{proof}
This solution highlights the decoupled structure of the orbit family.
Under the hypotheses of Theorem~\ref{thm:optimal-eta-g}, the optimal
transformed conditional density $f_{\widehat\eta^*(u)}$ depends only
on the target moments and the chosen exponents. Thus, the minimized
radial term in Proposition~\ref{prop:KL-decomposition} is independent
of $K$, and the remaining outer optimization reduces to minimizing
$D_{\mathrm{KL}}(h_P\|h_K)$ over unit-volume star bodies.
By Theorem~\ref{thm:optimal-starbody}, its unique minimizer is $K^*$,
with radial function $\rho_{K^*}(u)=(d\,h_P(u))^{1/d}$.

At each direction $u$, Theorem~\ref{thm:optimal-eta-g} requires a single coefficient vector $a\in\Theta$ satisfying $\mathbb E_a[R^{\tau_\ell}]=m_{P,\tau_\ell}(u)$ for all $\ell=1,\ldots,L$. For $L=1$, this holds automatically for positive target moments, since $\mathbb E_{a_1}[R^{\tau_1}]=d/(\tau_1a_1)$ ranges over $(0,\infty)$. For $L>1$, finiteness of the target moments alone does not ensure that such a vector exists. Without this moment-matching hypothesis, the radial infimum need not be attained by an admissible pair, and exact matching of all prescribed moments is not guaranteed.

\subsection{A Transport Perspective on Anisotropic Dilation}
So far, we have mainly focused on analyzing the optimal dilation with respect to the KL divergence. We show in this section that it is also possible to interpret the map $\Tcal_g$ through the lens of optimal transport under ray-preserving couplings, which will allow us to analyze the optimal dilation with respect to a raywise Wasserstein distance. Classical work established optimality results for radial transformations in multivariate coupling problems \cite{cuestaalbertos1993optimal}, while fibered and conditional Wasserstein formulations optimize over couplings constrained to remain within fibers of a prescribed map \cite{kloeckner2021extensions, peszek2023heterogeneous, chemseddine2025conditional}. In this spirit, we show that $\Tcal_g$ is optimal among transports constrained to preserve Euclidean rays when it maps $P$ to $Q$.

Recall that for two distributions $P,Q\in\mathcal P$ with finite second moments, their squared $2$-Wasserstein distance is defined by
\begin{align*}
    W_2^2(P,Q) := \inf_{\gamma \in \Gamma(P,Q)} \int_{\R^d \times \R^d} \|x-y\|_2^2 \,\mathrm{d}\gamma(x,y)
\end{align*} where $\Gamma(P,Q)$ is the set of all couplings between $P$ and $Q$. Define the class of ray-preserving couplings
\begin{align*}
    \Gamma_{\mathrm{ray}}(P,Q)
    :=\left\{\gamma\in\Gamma(P,Q):
    \frac{x}{\|x\|_2}=\frac{y}{\|y\|_2},\ \gamma\text{-a.s.}\right\}.
\end{align*}
The directions are well-defined almost surely because $P$ and $Q$ assign zero mass to the origin. Such couplings can exist only when $P_U=Q_U$. For distributions with this common directional marginal, define the squared raywise Wasserstein distance by
\begin{align*}
    W_{2,\mathrm{ray}}^2(P,Q) :=\inf_{\gamma\in\Gamma_{\mathrm{ray}}(P,Q)}
    \int_{\R^d\times\R^d}\|x-y\|_2^2\,\mathrm d\gamma(x,y).
\end{align*}
For a scalar $c>0$, define $\omega_c:\R_+\to\R_+$ by $\omega_c(r):=cr$; in particular, $\omega_{g(u)}(r)=g(u)r$ for $g\in G$ and $u\in\mathbb S^{d-1}$.

\begin{proposition} \label{prop:OT-interpretation}
Consider two distributions $P,Q\in\mathcal P$ with finite second moments. Assume their directional distributions agree, and write $P_U=Q_U=:\Psi$. Consider their polar disintegrations
\begin{align*}
    P_{R,U}(\mathrm{d}r, \mathrm{d}u) = \Psi(\mathrm{d}u)P_{R\mid U=u}(\mathrm{d}r)\quad\text{and}\quad Q_{R,U}(\mathrm{d}r, \mathrm{d}u) = \Psi(\mathrm{d}u)Q_{R\mid U=u}(\mathrm{d}r).
\end{align*} Then \begin{align*}
    \inf_{\gamma \in \Gamma_{\mathrm{ray}}(P,Q)} \int \|x-y\|_2^2 \mathrm{d}\gamma(x,y) = \int_{\mathbb{S}^{d-1}} W_2^2(P_{R\mid U=u}, Q_{R\mid U=u})\Psi(\mathrm{d}u).
\end{align*} 
Moreover, if $g\in G$ satisfies
$Q_{R\mid U=u}=(\omega_{g(u)})_{\#}P_{R\mid U=u}$ for $\Psi$-a.e. $u$, then the coupling $\gamma_g:=(\mathrm{Id},\Tcal_g)_{\#}P$ attains the infimum defining $W_{2,\mathrm{ray}}^2(P,Q)$. Thus $\Tcal_g$ is a quadratic-cost optimal transport map from $P$ to $Q$ among all Borel measurable ray-preserving transports
$T:\R^d\to\R^d$ in the class $\mathcal{T}_{\mathrm{ray}}(P,Q) :=\{T : T_{\#}P=Q\ \text{and}\ T(x)/\|T(x)\|_2 = x/\|x\|_2\ P\text{-a.s.}\}$, i.e., \begin{align*}
\Tcal_g\in\argmin_{T\in\mathcal T_{\mathrm{ray}}(P,Q)}
    \int_{\R^d}\|x-T(x)\|_2^2\,\mathrm dP(x).
\end{align*}
\end{proposition}
\begin{proof}[Proof of Proposition \ref{prop:OT-interpretation}] Observe that for $\gamma\in\Gamma_{\mathrm{ray}}(P,Q)$, disintegration gives a measurable family $\{\gamma_u\}_{u\in\mathbb S^{d-1}}$ with $\gamma_u\in\Gamma(P_{R\mid U=u},Q_{R\mid U=u})$ for $\Psi$-a.e. $u$ such that $$\gamma=H_{\#}\bigl(\Psi(\mathrm du)\gamma_u(\mathrm dr,\mathrm ds)\bigr),$$ where $H:\mathbb S^{d-1}\times\R_+\times\R_+\to\R^d\times\R^d$ is defined by $H(u,r,s):=(ru,su)$. For $\gamma$-a.e. $(x,y)$, we can write $x=ru$ and $y=su$. By integration against the pushforward measure and Tonelli's theorem,
\begin{align*}
    \int_{\R^d\times\R^d}\|x-y\|_2^2\,\mathrm d\gamma(x,y)
    =\int_{\mathbb S^{d-1}}\int_{\R_+\times\R_+}
    (r-s)^2\,\mathrm d\gamma_u(r,s)\,\Psi(\mathrm du).
\end{align*} 
For $\Psi$-a.e. $u$, the inclusion $\gamma_u\in\Gamma(P_{R\mid U=u},Q_{R\mid U=u})$ gives the lower bound
\begin{align*}
    \int_{\R_+\times\R_+}(r-s)^2\,\mathrm d\gamma_u(r,s)
    \geq W_2^2(P_{R\mid U=u},Q_{R\mid U=u}).
\end{align*}
Integrating over the sphere with respect to $\Psi$ and taking the infimum over $\gamma\in\Gamma_{\mathrm{ray}}(P,Q)$ yields the lower bound
\begin{align*}
    W_{2,\mathrm{ray}}^2(P,Q)
    \geq\int_{\mathbb S^{d-1}}
    W_2^2(P_{R\mid U=u},Q_{R\mid U=u})\,\Psi(\mathrm du).
\end{align*} 
To show the inequality in the other direction, recall that the squared $2$-Wasserstein distance in one dimension admits a quantile representation \cite{santambrogio2015optimal}. For a probability measure $\eta$ on $\R_+$, let $F_\eta(r):=\eta([0,r])$ for $r\geq0$ and define its quantile function by
$$F_\eta^{-1}(t):=\inf\{r\geq0:F_\eta(r)\geq t\}, \qquad t\in(0,1).
$$ 
Since $P$ and $Q$ have finite second moments, their conditional radial distributions have finite second moments for $\Psi$-a.e. $u$, and
\begin{align*}
    W_2^2(P_{R\mid U=u},Q_{R\mid U=u})
    =\int_0^1
    \left|F^{-1}_{P_{R\mid U=u}}(t)-F^{-1}_{Q_{R\mid U=u}}(t)\right|^2\,\mathrm dt.
\end{align*} 
As noted in \cite[Proposition 2.2 and Definition 2.3]{santambrogio2015optimal}, the coupling
$$
\pi_u^* :=\bigl(F^{-1}_{P_{R\mid U=u}}(\cdot),F^{-1}_{Q_{R\mid U=u}}(\cdot)\bigr)_{\#} \bigl(\mathcal L^1\rvert_{(0,1)}\bigr)
$$
is the co-monotone transport plan between $P_{R\mid U=u}$ and $Q_{R\mid U=u}$. Here, $\mathcal L^1\rvert_{(0,1)}$ is one-dimensional Lebesgue measure restricted to $(0,1)$. For $\Psi$-a.e. $u$, this coupling satisfies
\begin{align*}
    \int_{\R_+\times\R_+}(r-s)^2\,\mathrm d\pi_u^*(r,s)
    &=\int_0^1
    \left|F^{-1}_{P_{R\mid U=u}}(t)-F^{-1}_{Q_{R\mid U=u}}(t)\right|^2\,\mathrm dt =W_2^2(P_{R\mid U=u},Q_{R\mid U=u}).
\end{align*}
Since the conditional distributions form measurable probability kernels, their quantile functions are jointly measurable in $(u,t)$; hence $u\mapsto\pi_u^*$ is a measurable family. Consequently,
 $\gamma^* := H_{\#}(\Psi(\mathrm{d}u)\pi^*_u(\mathrm{d}r,\mathrm{d}s)) \in \Gamma_{\mathrm{ray}}(P,Q)$ and \begin{align*}
    \int_{\R^d \times \R^d} \|x-y\|^2_2 \mathrm{d}\gamma^*(x,y) = \int_{\mathbb{S}^{d-1}}\int_{\R_+ \times \R_+} (r-s)^2 \mathrm{d}\pi^*_u(r,s)\Psi(\mathrm{d}u)=\int_{\mathbb{S}^{d-1}}W_2^2(P_{R\mid U=u},Q_{R\mid U=u})\Psi(\mathrm{d}u).
\end{align*} This shows that the inequality is indeed an equality.

It remains to prove the claim concerning $\Tcal_g$. Let $\widetilde Q:=(\Tcal_g)_{\#}P$. Recall that Theorem~\ref{thm:structure-of-orbit} establishes that if $X=RU\sim P$, then $\Tcal_g(X)=g(U)RU$. Thus $\widetilde Q$ has directional distribution $\Psi$, with
 $$
    \widetilde Q_{R\mid U=u}
    =(\omega_{g(u)})_{\#}P_{R\mid U=u}
    =Q_{R\mid U=u}
    \qquad\text{for $\Psi$-a.e. $u$},
$$
where the last equality is by assumption. Hence $\widetilde Q$ and $Q$ have the same polar disintegration, so $(\Tcal_g)_{\#}P=Q$ and $\gamma_g=(\mathrm{Id},\Tcal_g)_{\#}P\in\Gamma_{\mathrm{ray}}(P,Q)$. For $\Psi$-almost every $u$, the relation $Q_{R\mid U=u} = (\omega_{g(u)})_{\#}P_{R\mid U=u}$ gives $F_{Q_{R\mid U=u}}(s)=F_{P_{R\mid U=u}}(s/g(u))$ which can be used to show $(\mathrm{Id},\omega_{g(u)})_{\#}P_{R\mid U=u} = \pi^*_u$. This ensures that the coupling induced by $r \mapsto g(u)r$ is the one-dimensional co-monotone optimal coupling on
$\Psi$-almost every ray. It follows that \begin{align*}
    \int_{\R^d}\|x - \Tcal_g(x)\|_2^2\mathrm{d}P(x) & = \int_{\mathbb{S}^{d-1}} \int_0^{\infty} |r - g(u)r|^2 P_{R\mid U=u}(\mathrm{d}r)\Psi(\mathrm{d}u) = \int_{\mathbb{S}^{d-1}} W_2^2(P_{R\mid U=u},Q_{R\mid U=u})\Psi(\mathrm{d}u).
\end{align*}
Finally, let $T \in \mathcal{T}_{\mathrm{ray}}(P,Q)$. Then $\gamma_T:= (\mathrm{Id},T)_{\#}P \in \Gamma_{\mathrm{ray}}(P,Q)$, and hence \begin{align*}
    \int_{\R^d}\|x - T(x)\|_2^2 \mathrm{d}P(x) = \int_{\R^d \times \R^d} \|x-y\|_2^2 \mathrm{d}\gamma_T(x,y) &\geq \inf_{\gamma\in\Gamma_{\mathrm{ray}}(P,Q)}\int_{\R^d \times \R^d} \|x-y\|_2^2 \mathrm{d}\gamma(x,y) \\
    & = \int_{\R^d} \|x-\mathcal{T}_g(x)\|^2_2\mathrm{d}P(x).
\end{align*}
\end{proof}

\begin{remark}[Optimal Dilations]
    The previous result shows that when $P$ and $Q$ have the same directional marginal and their conditional radial distributions are related by a profile $g\in G$, then  $\Tcal_g$ is optimal for quadratic cost among ray-preserving transports from $P$ to $Q$. A natural question is whether one can characterize the optimal dilation $g^*$ when the conditional radial distributions are not related in this way. For general $P,Q\in\mathcal P$ with finite second moments and common directional marginal $\Psi$, consider the minimization problem
    \begin{align*}
        \inf_{g \in G} W^2_{2,\mathrm{ray}}((\Tcal_g)_{\#}P, Q)
    \end{align*} 
    where $W_{2,\mathrm{ray}}^2((\Tcal_g)_{\#}P,Q)
    =\int_{\mathbb S^{d-1}} W_2^2\bigl((\omega_{g(u)})_{\#}P_{R\mid U=u},Q_{R\mid U=u}\bigr)\,\Psi(\mathrm du)$ by Proposition~\ref{prop:OT-interpretation}. This raywise Wasserstein distance is a special case of the conditional Wasserstein distance defined in \cite{chemseddine2025conditional}, with $U$ as the conditioning variable. 
    
    To find the pointwise minimizing scale, first note that for a random variable $X\sim\eta$ with $\eta$ a probability measure on $\R_+$, and for any $c>0$, the CDF satisfies 
    $F_{(\omega_c)_{\#}\eta}(s)
    = \mathbb P(cX\leq s)
    = \mathbb P(X\leq s/c)
    = F_\eta(s/c).$
     Thus, for $t\in(0,1)$,
     $$F^{-1}_{(\omega_c)_{\#}\eta}(t) = \inf\{s \geq 0 : F_{\eta}(s/c) \geq t\} = c\inf\{h \geq 0 : F_{\eta}(h) \geq t\}=cF^{-1}_{\eta}(t).$$ Hence for fixed $u \in \mathbb{S}^{d-1}$, we have $F^{-1}_{(\omega_{g(u)})_{\#}P_{R\mid U=u}}(t) = g(u)F^{-1}_{P_{R\mid U=u}}(t)$. Using the one-dimensional quantile formula for $W_2^2(P_{R\mid U=u},Q_{R\mid U=u})$ and expanding the square gives \begin{align*}
        \int_{\mathbb{S}^{d-1}}W_2^2((\omega_{g(u)})_{\#}P_{R\mid U=u},Q_{R\mid U=u})\Psi(\mathrm{d}u) & = \int_{\mathbb{S}^{d-1}}\int_0^1(g(u)F_{P_{R\mid U=u}}^{-1}(t) - F_{Q_{R\mid U=u}}^{-1}(t))^2\mathrm{d}t \Psi(\mathrm{d}u) \\
        & = \int_{\mathbb{S}^{d-1}}\left(g(u)^2 A_P(u) - 2g(u) B_{P,Q}(u) + C_Q(u)\right) \Psi(\mathrm{d}u)
    \end{align*} where we defined the quantities 
    \begin{align*}
    A_P(u)
    &:=\int_0^1\bigl(F^{-1}_{P_{R\mid U=u}}(t)\bigr)^2\,\mathrm dt
     = m_{P,2}(u),\\
    B_{P,Q}(u)
    &:=\int_0^1F^{-1}_{P_{R\mid U=u}}(t)F^{-1}_{Q_{R\mid U=u}}(t)\,\mathrm dt,\\
    C_Q(u) &:=\int_0^1\bigl(F^{-1}_{Q_{R\mid U=u}}(t)\bigr)^2\,\mathrm dt = m_{Q,2}(u)
    .
\end{align*}
Thus $A_P$ and $C_Q$ are the conditional radial second moments. 
Furthermore, these quantities are finite and strictly positive for $\Psi$-almost every $u$, with $B_{P,Q}(u)^2\leq A_P(u)C_Q(u)$ by Cauchy--Schwarz. One can minimize the integrand pointwise over positive scales to obtain
 \begin{align*}
        g^*(u) = \frac{B_{P,Q}(u)}{A_{P}(u)} = \frac{\int_0^1F_{P_{R\mid U=u}}^{-1}(t)F_{Q_{R\mid U=u}}^{-1}(t)\mathrm{d}t}{\int_0^1 (F_{P_{R\mid U=u}}^{-1}(t))^2 \mathrm{d}t}
    \end{align*}
    for $\Psi$-almost every $u$. On each such ray, $g^*(u)$ is the least-squares slope that best aligns the radial quantile function of $P_{R\mid U=u}$ with that of $Q_{R\mid U=u}$. Interchanging $P$ and $Q$ gives the analogous pointwise scale
$B_{P,Q}(u)/C_Q(u)$ for minimizing the distance between
$(\Tcal_g)_{\#}Q$ and $P$. For either problem, any profile in $G$
that agrees $\Psi$-almost everywhere with the corresponding
pointwise scale minimizes the objective over $G$. A sufficient
condition is that the numerator and denominator of the
corresponding ratio each agree $\Psi$-almost everywhere with
a profile in $G$. An interesting avenue of future work would be to analyze the use of such optimal dilations in the context of regularization and compare the properties of profiles learned via the raywise Wasserstein distance and the KL divergence.
\end{remark}

\section{Computational Illustrations}\label{sec:denoising-experiments}
This section illustrates our theoretical framework through denoising
experiments on a two-dimensional synthetic distribution and MNIST.
In the synthetic setting, we generate clean training samples from a prescribed distribution and substitute kernel estimates of its directional density and conditional radial moments into the closed-form characterizations derived above. For MNIST,
we learn parameterized regularizers and scaling profiles by minimizing
empirical objectives. We first describe the learning procedure used
in the MNIST experiments (Section~\ref{sec:comp-learning-procedure}),
then present the synthetic and MNIST denoising results
(Sections~\ref{sec:two-d-denoising} and~\ref{sec:mnist-denoising},
respectively).

\subsection{Learning Procedure} \label{sec:comp-learning-procedure} 

When the population quantities required by our theoretical characterizations are unavailable, we estimate the base regularizer and scaling profile within parameterized function classes. Specifically, we represent the star body through its gauge $\|\cdot\|_{K_\phi}$ and the scaling profile by a positive function $g_\theta$, and optimize an empirical counterpart of the functional in Lemma~\ref{lem:KL-objective}. The resulting pair defines a learned regularizer for a downstream inverse problem. We describe the training objective and parameterizations below. 

\paragraph{Training and Loss Function} Suppose we have nonzero training samples $x_1,\ldots,x_n\sim P$, and
write $u_i=x_i/\|x_i\|_2$. We consider the anisotropically dilated Gibbs
model \begin{align*}
    Q_{K_\phi,g_\theta} = (\Tcal_{g_\theta})_\#p_{K_\phi}, \qquad p_{K_\phi}(x) = Z_{K_\phi}^{-1}\exp(-\|x\|_{K_\phi}).
\end{align*}
By Theorem~\ref{thm:structure-of-orbit}, the density $q_{\phi,\theta}$
of the pushforward measure $Q_{K_\phi,g_\theta}$ can be written as
$$
q_{\phi,\theta}(x) = p_{K_\phi}(\mathcal T_{1/g_{\theta}}(x))g_{\theta}(x/\|x\|_2)^{-d}.
$$ 
Consequently, the empirical negative log-likelihood, up to constants independent of $(\phi,\theta)$, is given by
\begin{align}
    \widehat{\mathcal L}(\phi,\theta)
    =\frac1n\sum_{i=1}^n
      \left[
          \frac{\|x_i\|_{K_\phi}}{g_\theta(u_i)}
          +d\log g_\theta(u_i)
      \right]
      +\log Z_{K_\phi}.
    \label{eq:experiments-NLL}
\end{align}
When $K_\phi$ is volume-normalized, the partition function $Z_{K_\phi}$
is constant and can be omitted from the optimization. Otherwise, we
approximate $Z_{K_\phi}$ by Monte Carlo integration over
$\mathbb S^{d-1}$, using the relationship
$Z_K\propto\mathrm{vol}_d(K)$ and the polar volume formula
$\mathrm{vol}_d(K)=d^{-1}\int_{\mathbb S^{d-1}}\rho_K(u)^d\,
\mathrm d\sigma(u)$. In practice, we observe that alternating
minimization over $\phi$ and $\theta$ yields stable convergence.
Once the parameters are learned, the resulting regularizer is the
negative log-density, up to an additive constant independent of $x$:
\begin{align*}
    E_{(K_{\phi},g_{\theta})}(x) := \frac{\|x\|_{K_{\phi}}}{g_{\theta}(x/\|x\|_2)} + d \log g_{\theta}(x/\|x\|_2).
\end{align*} 

\paragraph{Parameterization}
Parameterizing $K_\phi$ is equivalent to parameterizing its gauge
$\|\cdot\|_{K_\phi}$. The gauge must be positively homogeneous, positive
except at the origin, and continuous. In the MNIST experiments, we set
$$
\|x\|_{K_\phi} :=E(x)+\lambda_{\phi_1}\|x\|_2+f_{\phi_2}(x),
$$ 
where $E$ is a continuous, nonnegative, positively homogeneous base
regularizer, $\lambda_{\phi_1}>0$ is learned, and $f_{\phi_2}$ is a
neural network with linear convolutional layers, no biases, and
positively homogeneous activation functions, such as ReLU. We also incorporate a non-negative activation in the final output layer. This ensures that $f_{\phi_2}$ is continuous,
nonnegative, and positively homogeneous, while the Euclidean norm
term ensures positivity away from the origin. For $g_\theta$, we
require continuity and positivity on the sphere. We therefore
parameterize it as $g_\theta(u)=\exp(a_\theta(u))$ for a continuous
map $a_\theta$.

\begin{remark}[Improving hand-crafted regularizers]
    Another way to exploit $\Tcal_g$ is to improve a handcrafted or fixed base regularizer by learning a direction-dependent correction. For a continuous, nonnegative, positively homogeneous base regularizer $E$, such as total variation \cite{Rudinetal92}, we learn this correction by solving \begin{align}
        \min_\theta \sum_{i=1}^n \left[\frac{E(x_i)}{g_{\theta}(u_i)} + d \log g_{\theta}(u_i)\right]. \label{eq:experiments-improving-HC-loss}
    \end{align} 
    Provided $Z_E:=\int_{\mathbb R^d}\exp(-E(x))\,\mathrm dx<\infty$,
this is the fixed-base specialization of
\eqref{eq:experiments-NLL} for anisotropic dilations of the density
$p_E(x)=Z_E^{-1}\exp(-E(x))$, with the constant $\log Z_E$ omitted. The correction also defines
an energy when the base regularizer does not induce a normalizable
density. For example, for regularizers such as TV, $Z_E$ may be infinite. Nevertheless, \eqref{eq:experiments-improving-HC-loss} remains meaningful as an empirical learning objective and is the pointwise limit, up to constants independent of $\theta$, of the negative log-likelihood objectives obtained from the nomralizable Gibbs energies $E_{\varepsilon}(x) = E(x) + \varepsilon\|x\|_2$ as $\varepsilon\downarrow 0$. For a fitted profile $g_\theta$, the resulting regularizer is
\begin{align}
    \widehat E_{g_\theta}(x)
    :=\frac{E(x)}{g_\theta(x/\|x\|_2)}
      +d\log g_\theta(x/\|x\|_2),
    \qquad x\ne0.
    \label{eq:learned-Ehat-gtheta}
\end{align}
We explore this approach in the MNIST experiments.
\end{remark}

\begin{remark}[Existence of minimizers] Regarding the well-posedness of optimization using a regularizer of the form \eqref{eq:learned-Ehat-gtheta}, the existence of minimizers follows under natural assumptions. Suppose $g:\mathbb S^{d-1}\to(0,\infty)$ is continuous and $E$ is
continuous, nonnegative, and positively homogeneous. Let $\widehat E_g$
be defined as in \eqref{eq:learned-Ehat-gtheta}, extending its definition
to the origin by $\widehat E_g(0):=d\log g_{\min}$, where
$g_{\min}:=\min_{u\in\mathbb S^{d-1}}g(u)>0$. For any
$y\in\mathbb R^d$ and $\lambda\ge0$, the functional
\begin{align*}
    \mathcal{L}(x) := \|x-y\|_2^2 + \lambda \widehat{E}_g(x)
\end{align*} attains its minimum. Indeed, positive homogeneity and continuity of $E$ imply
$$0 \leq \frac{E(x)}{g(x/\|x\|_2)} \leq \frac{\max_{u\in\mathbb{S}^{d-1}}E(u)}{g_{\min}}\|x\|_2 \longrightarrow 0\ \quad\ \text{as}\ \quad x \longrightarrow 0.$$ Also, $\liminf_{x\rightarrow 0} d \log g(x/\|x\|_2) = d\log g_{\min}.$ This implies that $\widehat{E}_g$, and hence $\mathcal{L}$, is lower semicontinuous. Finally, $$\mathcal{L}(x) \geq (\|x\|_2 - \|y\|_2)^2 + \lambda d \log g_{\min} \longrightarrow \infty\ \quad \text{as}\ \quad \|x\|_2 \longrightarrow \infty.$$ Thus, $\mathcal{L}$ is proper, lower semicontinuous, and
coercive, and existence of a minimizer follows from the direct method.
\end{remark}

\subsection{Two-Dimensional Denoising Experiment} \label{sec:two-d-denoising}

We first investigate whether learning directional geometry, radial scale,
and higher-order radial shape can improve regularization in a controlled
two-dimensional denoising problem. We use $5000$ clean training samples
to train our models. At test time, samples are corrupted by additive
Gaussian noise,
\begin{align*}
    y=x+\sigma\xi,
    \qquad \xi\sim\mathcal N(0,I_2),
    \qquad \sigma>0,
\end{align*}
and are reconstructed by posterior averaging with the regularizer $\widehat E$:
\begin{align*}
    \widehat x_\gamma(y)
    =\mathbb E_{\pi_\gamma(\cdot\mid y)}[X], \qquad  \pi_\gamma(x\mid y)
\propto
    \exp\left(
        -\frac{\|y-x\|_2^2}{2\sigma^2}-\gamma \widehat E(x)
    \right).
\end{align*}
The regularization parameter $\gamma>0$ is selected separately for each model and noise level using validation mean squared error (MSE) over $320$ samples. To compute the posterior mean, we use numerical quadrature on a polar grid.

The clean distribution follows the radial construction of
Section~\ref{sec:exponential-family-model}. In particular, we generate
samples from a numerical approximation to the law
\begin{align*}
    X=\rho_0(\Theta)g_0(\Theta)S\,u(\Theta),
    \qquad
    u(\vartheta)=(\cos\vartheta,\sin\vartheta),
\end{align*}
where $\vartheta$ denotes the polar angle, $\rho_0$ is the radial
function of a baseline star body $K_0$, $g_0$ is the generating scaling
profile, and $\Theta$ has density
$h_0(\vartheta)\propto\rho_0(\vartheta)^2$. Conditional on
$\Theta=\vartheta$, the standardized radius $S$ follows a radial
exponential family with exponents $1,2,3$ and mean $2$. This construction
allows the directional marginal, radial scale, and higher-order radial
shape to vary separately. 

\paragraph{Models} 
In this experiment, we compare a fixed circular regularizer with data-driven
regularizers that incorporate directional geometry, radial scale, and
radial shape. In particular, the four Gibbs models we compare are
1) the circular baseline $\widehat E(x)=\|x\|_2$,
2) a $g$-only model $\widehat E=E_g$ with an estimated scaling profile and
$\rho_K\equiv1$,
3) a $K_P$-only model $\widehat E=E_{K_P}$ with an estimated star body and
$g\equiv1$,
and 4) a $(K,g)$ model $\widehat E=E_{(K,g)}$ with both components
estimated. In polar coordinates, the $(K,g)$ model has energy
\begin{align*}
    E_{(K,g)}(r,\vartheta)
    =2\log g(\vartheta)
     +\frac{r}{\rho_K(\vartheta)g(\vartheta)}.
\end{align*}
Note that these four models have fixed conditional radial shape and are
therefore misspecified for the clean distribution.

We also consider two models from the radial exponential families of
Section~\ref{sec:exponential-family-model}: a $(K,\eta)$ model that
learns a star body and radial shape with $g\equiv1$, and a full
$(K,g,\eta)$ model that also learns the radial scale. Writing the
coefficients $\eta$ in the standardized radius
$s=r/(\rho_K(\vartheta)g(\vartheta))$, the full model has energy
\begin{align*}
    E_{(K,g,\eta)}(r,\vartheta)
    =2\log g(\vartheta)+\log A(\eta(\vartheta))
      +\sum_{\ell=1}^3\eta_\ell(\vartheta)s^\ell.
\end{align*}
Here, $A(\eta)$ is the normalizing constant of the standardized radial
density proportional to $s\exp(-\sum_{\ell=1}^3\eta_\ell s^\ell)$ on $(0,\infty)$.
Setting $g\equiv1$ gives the energy of the $(K,\eta)$ model.

We find the parameter values of our models by computing empirical estimations of the closed-form
population characterizations. Specifically,
we use circular kernel smoothing of the clean training samples. Let $\widehat h_P$ denote the estimate of the ground-truth directional density 
$h_P$, and let $\widehat m_{P,\ell}$
denote the estimates of the conditional radial moments $m_{P,\ell}$ of
orders $\ell=1,2,3$. For the $(K,g)$ and $(K,g,\eta)$ models,
Theorems~\ref{thm:optimal-starbody} and~\ref{thm:optimal-g-K} give the
empirical profiles
$$
\widehat\rho_K(u)\propto\widehat h_P(u)^{1/2},
    \qquad
    \widehat g(u)=\frac{\widehat m_{P,1}(u)}{2\widehat\rho_K(u)},
    \qquad u\in\mathbb S^1.
$$
For the $K_P$-only model, \eqref{eq:rho-decomposition} gives
$\widehat\rho_K\propto(\widehat h_P\widehat m_{P,1})^{1/3}$,
which we also use for the $(K,\eta)$ comparison. For the $g$-only model,
we fix $\widehat\rho_K\equiv1$ and compute
$\widehat g=\widehat m_{P,1}/2$. These profiles are used directly on the
angular quadrature grid. For the radial-shape models, we fit $\eta$
using a loss that penalizes discrepancies in the standardized second
and third moments and encourages the standardized first moment to
equal $2$ by \eqref{eq:baseline-constraint} with $d=2$ and
$\tau_1=1$. Please see Appendix~\ref{sec:appendix-2d} for the specific clean
distribution, kernel estimation, moment fitting, and quadrature details.

\begin{table}[t]
\centering
\begin{tabular}{lccc}
\toprule
Model & $\sigma=0.25$ & $\sigma=0.50$ & $\sigma=0.80$ \\
\midrule
Noisy observation & $0.3532$ & $0.7033$ & $1.1514$ \\
Circular baseline & $0.3511$ & $0.6816$ & $1.0404$ \\
$g$-only & $0.3494$ & $0.6668$ & $0.9997$ \\
$K_P$-only & $0.3475$ & $0.6478$ & $0.9663$ \\
$(K,g)$ & $\mathbf{0.3449}$ & $\mathbf{0.6456}$ & $\mathbf{0.9614}$ \\
\bottomrule
\end{tabular}
\caption{Two-dimensional denoising performance over $650$ test samples
for the Gibbs models. Entries report test RMSE, with the regularization
parameter selected separately for each model and noise level using
validation MSE. The lowest RMSE at each noise level among these models
is shown in bold.}
\label{tab:toy-base-rmse}
\end{table}

\begin{table}[t]
\centering
\begin{tabular}{lccc}
\toprule
Model & $\sigma=0.25$ & $\sigma=0.50$ & $\sigma=0.80$ \\
\midrule
$(K,\eta)$ & $0.3499$ & $0.6475$ & $0.9826$ \\
$(K,g)$ & $0.3449$ & $0.6456$ & $0.9614$ \\
$(K,g,\eta)$ & $\mathbf{0.3428}$ & $\mathbf{0.6381}$ & $\mathbf{0.9506}$ \\
\bottomrule
\end{tabular}
\caption{Two-dimensional denoising performance over $650$ test samples
for models incorporating radial scale and higher-order shape. Entries
report test RMSE, with the regularization parameter selected separately
for each model and noise level using validation MSE. The lowest RMSE at
each noise level is shown in bold.}
\label{tab:toy-eta-rmse}
\end{table}

\paragraph{Results}
We evaluate the six regularization models in reconstructing $650$ test
samples. Tables~\ref{tab:toy-base-rmse} and~\ref{tab:toy-eta-rmse} report
test root mean squared Euclidean error (RMSE) for Gaussian noise levels
$\sigma\in\{0.25,0.50,0.80\}$. Adding the anisotropic profile $g$
consistently improves the circular baseline. Relative to that baseline,
the $g$-only model reduces RMSE by approximately $0.0017$, $0.0148$, and
$0.0407$ as the noise level increases. These results show that a
direction-dependent radial correction can improve denoising performance
while keeping the base star body fixed. The $K_P$-only model also improves
upon the circular baseline at every noise level. The $(K,g)$ model gives
the lowest RMSE among the Gibbs models, although the additional gains
over $K_P$-only are modest.

Table~\ref{tab:toy-eta-rmse} compares the models incorporating radial
scale and higher-order shape. The $(K,g)$ model consistently outperforms
the $(K,\eta)$ comparison, which uses the $K_P$ geometry estimate and fixes
$g\equiv1$. Adding radial shape flexibility to the $(K,g)$ model
further reduces RMSE by approximately $0.0020$, $0.0075$, and $0.0108$
across the three noise levels. The full $(K,g,\eta)$ model attains the
lowest test RMSE across all regimes.

\subsection{Denoising on MNIST} \label{sec:mnist-denoising}

We next investigate whether anisotropic dilation can improve both a classical handcrafted regularizer and a learned regularizer for a denoising problem on $14 \times 14$ MNIST images \cite{MNIST}. We use $20000$ clean training images to train our models. At test time, images are corrupted by additive Gaussian noise, \begin{align*}
    y = x + \sigma \xi,\ \qquad \xi \sim \mathcal{N}(0,I_d),\ \qquad \sigma >0,
\end{align*}and are reconstructed by solving the following maximum a posteriori
(MAP) estimation problem with the regularizer $\widehat E$: \begin{align*}
    \min_{x \in \R^d} \frac{1}{2}\|y - x\|_2^2 + \lambda \widehat{E}(x).
\end{align*} For the Gaussian likelihood, this normalization corresponds to
$\lambda=\sigma^2\gamma$ in the posterior formulation of
Section~\ref{sec:two-d-denoising}. The regularization parameter
$\lambda$ is selected separately for each model and noise level
using validation peak signal-to-noise ratio (PSNR). To minimize the
MAP objective, we use the nonmonotone accelerated-gradient method
of \cite{li2015accelerated}.

\paragraph{Models} In this experiment, we compare a handcrafted regularizer with a learned
one and examine the effect of learning an additional scaling profile.
In particular, the four models we compare are
1) TV, $\widehat E(x)=E(x)=\mathrm{TV}(x)$,
2) TV with a learned scaling profile,
$\widehat E=\widehat E_{g_\theta}$, obtained by solving
\eqref{eq:experiments-improving-HC-loss},
3) a base learned regularizer
$\widehat E(\cdot)=\|\cdot\|_{K_\phi}$ with $g_\theta\equiv1$, and
4) a jointly learned regularizer
$\widehat E=E_{(K_\phi,g_\theta)}$ obtained by solving
\eqref{eq:experiments-NLL}. For our parameterizations, we consider
a base regularizer of the form \begin{align*}
    \|x \|_{K_{\phi}} = \mathrm{TV}(x) + \lambda_{\phi_1}\|x\|_2 + f_{\phi_2}(x),\ \qquad \lambda_{\phi_1}>0,
\end{align*} where $f_{\phi_2}$ is a neural network with linear convolutional layers
without biases, LeakyReLU activations, and an absolute value output
activation to ensure nonnegativity. We include the TV term to retain
the structure of the fixed TV baseline. The resulting regularizer
$\|\cdot\|_{K_\phi}$ is positively homogeneous, positive except at the
origin, and continuous, and therefore satisfies the requirements to
be a gauge. As in previous work on adversarial regularization
\cite{Lunzetal18}, we also employ a gradient penalty to encourage
$\|\cdot\|_{K_\phi}$ to be Lipschitz.

For the anisotropic dilation profile $g_\theta$, we found that a simple learned linear model of image statistics was sufficient for good performance. Specifically, we use the exponential parameterization: \begin{align*}
    g_\theta(u)
    :=\exp\left(-\theta^\top\varphi(u)\right),
    \qquad \theta\in\mathbb R^3,
    \qquad \varphi(u)\in\mathbb R^3.
\end{align*} We note that such exponential models and Gibbs densities have a rich history for prior learning \cite{zhu1997prior, zhu1998filters, roth2009fields}. Here, $\varphi:\mathbb S^{d-1}\to\mathbb R^3$ consists of a constant
feature and two image statistics, spatial support and spread,
standardized using the clean training data. Please see the Appendix
for further details.

\begin{table}[t]
    \centering
    
    \setlength{\tabcolsep}{5.5pt}
    \small
    \begin{tabular}{lcccc}
        \toprule
        Model
        & $\sigma=0.10$
        & $\sigma=0.15$
        & $\sigma=0.20$
        & $\sigma=0.25$ \\
        \midrule
        TV
        & $24.245\;(0.003856)$
        & $20.995\;(0.008109)$
        & $18.485\;(0.014417)$
        & $16.514\;(0.022678)$ \\

        TV with $g_{\theta}$
        & $\underline{24.634\;(0.003515)}$
        & $\underline{21.779\;(0.006872)}$
        & $\underline{19.156\;(0.012775)}$
        & $\underline{18.318\;(0.015506)}$ \\

        Learned $\|\cdot\|_{K_{\phi}}$
        & $23.759\;(0.004307)$
        & $20.756\;(0.008563)$
        & $18.891\;(0.013883)$
        & $17.986\;(0.016850)$ \\

        Learned $\|\cdot\|_{K_{\phi}}$ with $g_{\theta}$
        & $\mathbf{25.010\;(0.003298)}$
        & $\mathbf{22.057\;(0.006633)}$
        & $\mathbf{20.158\;(0.010231)}$
        & $\mathbf{18.728\;(0.014336)}$ \\
        \bottomrule
    \end{tabular}
     \caption{
        MNIST denoising performance over $100$ test images using fixed TV and a learned nonconvex
        regularizer, with and without the two-feature anisotropic profile
        $g_{\theta}$. Entries report test PSNR in dB, with test MSE
        in parentheses. The regularization parameter is selected separately
        for each model and noise level using validation PSNR. The best result at each noise level is
        shown in bold and the second best result is underlined.
    }\label{tab:mnist-results}
\end{table}
\paragraph{Results}
We evaluate the four regularization models in reconstructing $100$
test images. Table~\ref{tab:mnist-results} reports test PSNR and MSE
for Gaussian noise levels $\sigma\in\{0.10,0.15,0.20,0.25\}$.
Adding the anisotropic profile $g_\theta$ consistently improves the
fixed TV regularizer. Relative to TV, TV with $g_\theta$ improves
PSNR by approximately $0.389$, $0.784$, $0.671$, and $1.804$ dB
as the noise level increases. These results show that a low-capacity
direction-dependent radial correction can improve performance even
when the underlying regularizer remains fixed. The anisotropic
correction also improves the learned nonconvex regularizer
$\|\cdot\|_{K_\phi}$. Learned $\|\cdot\|_{K_\phi}$ with $g_\theta$
improves over learned $\|\cdot\|_{K_\phi}$ by approximately
$1.251$, $1.301$, $1.267$, and $0.742$ dB across the four noise
levels. The combined model also attains the best test PSNR values
across all regimes. 

\section{Statistical Learning Guarantees}\label{sec:statistical-learning}

To complete our analysis, we establish sample complexity guarantees for learning a star body
$K$ and a scaling profile $g$ from data. We relate the complexity of
these hypothesis classes to the required sample size by establishing
uniform convergence of the empirical risk to its population counterpart. Let $P$ be a probability distribution on $\mathbb R^d$ supported on the
Euclidean unit ball $B^d$ with $P(\{0\})=0$, so that $\|X\|_2\le1$
almost surely for $X\sim P$. For a star body $K\in\mathcal S^d$ with
Minkowski functional $\|\cdot\|_K$ and a continuous scaling profile
$g:\mathbb S^{d-1}\to(0,\infty)$, the loss under the anisotropic radial
transform is defined for $x\in B^d\setminus\{0\}$ by
$$
\ell_{K,g}(x)
:=d\log g\left(\frac{x}{\|x\|_2}\right)
+\frac{\|x\|_K}{g\left(\frac{x}{\|x\|_2}\right)}.
$$
Given $m$ training samples
$X_1,\ldots,X_m\overset{\mathrm{i.i.d.}}{\sim}P$, let $P_m$ denote their
empirical distribution. The empirical and population anisotropic orbit
risks are respectively defined as
$$
F(K,g;P_m):=\frac1m\sum_{i=1}^m\ell_{K,g}(X_i),
\qquad
F(K,g;P):=\mathbb E_P[\ell_{K,g}(X)].
$$
To obtain uniform bounds on the losses, we consider classes of
non-degenerate star bodies and bounded, continuous scaling profiles.
Fix $r>0$ and $0<a\le b$, and let $\mathfrak C$ be a nonempty class
satisfying
$$
\mathfrak C\subseteq\{K\in\mathcal S^d:rB^d\subseteq K\},
\qquad
G_a^b:=\{g\in G:a\le g(u)\le b
\text{ for every }u\in\mathbb S^{d-1}\}.
$$
We choose a nonempty profile class $\mathcal G\subseteq G_a^b$ and use
$\mathfrak C\times\mathcal G$ as the product hypothesis class. The
associated anisotropic orbit loss class is generated by varying $(K, g)$ over this hypothesis class:
$$
\mathcal L_{\mathrm{orb}}
:=\{\ell_{K,g}:K\in\mathfrak C,\ g\in\mathcal G\}.
$$
For a bounded function $f$ on a set $\Omega$, write
$\|f\|_{\infty,\Omega}:=\sup_{x\in\Omega}|f(x)|$.
We omit the domain when it is clear from context. We equip
$\mathfrak C\times\mathcal G$ with the metric
$D_{\mathrm{orb}}((K,g),(L,h)):=\delta(K,L)+\|g-h\|_\infty$,
where $\delta(K,L):=\|\rho_K-\rho_L\|_\infty$ is the radial Hausdorff
metric on $\mathcal S^d$ and these supremum norms are taken over
$\mathbb S^{d-1}$.

Our goal is to identify conditions on $\mathfrak C$ and $\mathcal G$
under which, for each confidence parameter $\gamma\in(0,1)$, the bound
$$
\sup_{(K,g)\in\mathfrak C\times\mathcal G}
\left|F(K,g;P_m)-F(K,g;P)\right|
\le\eta(m,\mathfrak C\times\mathcal G,P,\gamma)
$$
holds with probability at least $1-\gamma$, where
$\eta(m,\mathfrak C\times\mathcal G,P,\gamma)\to0$ as $m\to\infty$
for fixed classes, $P$, and $\gamma$. We first show that the loss map
$(K,g)\mapsto\ell_{K,g}$ is Lipschitz continuous, allowing us to
transfer covering-number bounds from $\mathfrak C\times\mathcal G$
to $\mathcal L_{\mathrm{orb}}$.

\begin{lemma}[Lipschitz properties of the loss maps]
\label{lem:lipschitz-loss}
Let $\mathfrak C$ and $\mathcal G$ be defined as above. The isotropic
loss map $K\mapsto\ell_K$, where $\ell_K(x)=\|x\|_K$, is Lipschitz
continuous with respect to the radial metric $\delta$ over $\mathfrak C$,
satisfying $\|\ell_K-\ell_L\|_{\infty, B^d}\le r^{-2}\delta(K,L)$ for all
$K,L\in\mathfrak C$. Furthermore, the anisotropic
orbit loss map $(K,g)\mapsto\ell_{K,g}$ is Lipschitz continuous with
respect to $D_{\mathrm{orb}}$ over $\mathfrak C\times\mathcal G$,
satisfying
$$
\|\ell_{K,g}-\ell_{L,h}\|_{\infty, B^d\setminus\{0\}}
\le C_{\mathrm{orb}}D_{\mathrm{orb}}\big((K,g),(L,h)\big)
$$
for all $(K,g),(L,h)\in\mathfrak C\times\mathcal G$ and
$C_{\mathrm{orb}}:=\max\{(ar^2)^{-1},\,d/a+(ra^2)^{-1}\}$.
\end{lemma}

\begin{proof}
The Lipschitz continuity of the isotropic loss map is established in
\cite[Section 4, Proposition 2]{leong2025optimal}. We therefore focus
on the anisotropic claim. Let $(K,g),(L,h)\in\mathfrak C\times\mathcal G$
and fix $x\in B^d\setminus\{0\}$. Writing $u=x/\|x\|_2$, the triangle
inequality gives
$$
|\ell_{K,g}(x)-\ell_{L,h}(x)|
\le d|\log g(u)-\log h(u)|
+\left|\frac{\|x\|_K}{g(u)}-\frac{\|x\|_L}{h(u)}\right|.
$$
Since $g(u),h(u)\in[a,b]$, the logarithm is $(1/a)$-Lipschitz on this
interval, so the first term is at most $(d/a)\|g-h\|_\infty$.
For the second term, adding and subtracting $\|x\|_L/g(u)$ gives
$$
\left|\frac{\|x\|_K}{g(u)}-\frac{\|x\|_L}{h(u)}\right|
\le\frac{1}{g(u)}\left|\|x\|_K-\|x\|_L\right|
+\|x\|_L\left|\frac{1}{g(u)}-\frac{1}{h(u)}\right|.
$$
Using the isotropic bound, $g(u)\ge a$, $\|x\|_L\le r^{-1}$, and
$|g(u)^{-1}-h(u)^{-1}|\le a^{-2}\|g-h\|_\infty$, we obtain
$$
|\ell_{K,g}(x)-\ell_{L,h}(x)|
\le\frac{1}{ar^2}\delta(K,L)
+\left(\frac da+\frac{1}{ra^2}\right)\|g-h\|_\infty.
$$
Taking the supremum over $x\in B^d\setminus\{0\}$ proves the claim.
\end{proof}

Throughout this section, $N(\mathcal H,\mathsf d,\eta)$ denotes the
proper covering number of $\mathcal H$ at radius $\eta>0$ with respect
to a metric or pseudometric $\mathsf d$, where the centers belong to
$\mathcal H$. Let $\mathbf X_m:=(X_1,\ldots,X_m)$ denote the training
sample and define the empirical pseudo-norms
$$
\|f\|_{q,\mathbf X_m}
:=\left(\frac1m\sum_{i=1}^m|f(X_i)|^q\right)^{1/q},
\qquad q\in\{1,2\}.
$$

\begin{theorem}[Uniform convergence for anisotropic orbit regularizers] \label{thm:uniform-convergence}
Let $P$ be a distribution supported on $B^d$ with $P(\{0\})=0$, and let
$\mathfrak C$ and $\mathcal G\subseteq G_a^b$ be the hypothesis classes
defined above. Assume that, for some $\widetilde\varepsilon>0$,
$N(\mathfrak C\times\mathcal G,D_{\mathrm{orb}},\varepsilon)$ is finite
for every $\varepsilon\in(0,\widetilde\varepsilon]$.
Let $C_{\mathrm{orb}}$ be the Lipschitz constant from
Lemma~\ref{lem:lipschitz-loss}, and set
\[
\begin{aligned}
B_{\mathrm{orb}}
&:=d\max\{|\log a|,|\log b|\}+(ar)^{-1},
\qquad c_{\mathrm{orb}}:=4B_{\mathrm{orb}}, \qquad c_{P,m}^{\mathrm{orb}}:=\sup_{\ell\in\mathcal L_{\mathrm{orb}}}
\|\ell\|_{2,\mathbf X_m}.
\end{aligned}
\]
Then, for any $\gamma\in(0,1)$, with probability at least $1-\gamma$
over $m$ i.i.d.\ samples from $P$, the bound
\[
\begin{aligned}
|F(K,g;P)-F(K,g;P_m)|
&\le 2\inf_{\varepsilon\in(0,\widetilde\varepsilon]}
\left\{\varepsilon+c_{P,m}^{\mathrm{orb}}
\sqrt{\frac2m\log N\left(
\mathfrak C\times\mathcal G,D_{\mathrm{orb}},
\frac{\varepsilon}{C_{\mathrm{orb}}}\right)}\right\} +c_{\mathrm{orb}}\sqrt{\frac{2\log(8/\gamma)}m}
\end{aligned}
\]
holds simultaneously for all $(K,g)\in\mathfrak C\times\mathcal G$.
\end{theorem}

\begin{proof}
Since $rB^d\subseteq K$ for all $K\in\mathfrak C$, we have
$\|x\|_K\le r^{-1}$ for every $x\in B^d$. Together with
$a\le g\le b$, this gives
\[
|\ell_{K,g}(x)|
\le d\max\{|\log a|,|\log b|\}+\frac1{ar}
=:B_{\mathrm{orb}},
\]
for every $x\in B^d\setminus\{0\}$. Let
\[
\mathcal R(\mathcal L_{\mathrm{orb}};\mathbf X_m)
:=\mathbb E_\sigma\left[
\sup_{\ell\in\mathcal L_{\mathrm{orb}}}
\frac1m\sum_{i=1}^m\sigma_i\ell(X_i)\right]
\]
denote the empirical Rademacher complexity \cite{bartlett2005local},
where the $\sigma_i$ are mutually independent, independent of the
sample, and take values $\pm1$ with equal probability. The classes
$\mathcal L_{\mathrm{orb}}$ and $-\mathcal L_{\mathrm{orb}}$ have the
same empirical Rademacher complexity by symmetry of the $\sigma_i$.
Applying \cite[Theorem 26.5, part 2]{shalev2014understanding} to both
classes, each with failure probability $\gamma/2$, and taking a union
bound gives, with probability at least $1-\gamma$ over the choice of
$\mathbf X_m$,
\[
|F(K,g;P)-F(K,g;P_m)|
\le 2\mathcal R(\mathcal L_{\mathrm{orb}};\mathbf X_m)
+c_{\mathrm{orb}}\sqrt{\frac{2\log(8/\gamma)}m}
\]
simultaneously for all $\ell_{K,g}\in\mathcal L_{\mathrm{orb}}$,
where $c_{\mathrm{orb}}:=4B_{\mathrm{orb}}$. Furthermore, for every $\varepsilon>0$,
\cite[Proposition 5.2]{rebeschini2021covering} gives
\[
\mathcal R(\mathcal L_{\mathrm{orb}};\mathbf X_m)
\le\varepsilon+c_{P,m}^{\mathrm{orb}}
\sqrt{\frac2m\log N\left(
\mathcal L_{\mathrm{orb}},\|\cdot\|_{1,\mathbf X_m},
\varepsilon\right)},
\]
where $c_{P,m}^{\mathrm{orb}}:=
\sup_{\ell\in\mathcal L_{\mathrm{orb}}}
\|\ell\|_{2,\mathbf X_m}\le B_{\mathrm{orb}}$.

To evaluate this metric entropy, we map the covering numbers of
$\mathcal L_{\mathrm{orb}}$ back to the parameter space
$\mathfrak C\times\mathcal G$. Since
$X_i\in B^d\setminus\{0\}$ almost surely for every $i=1,\ldots,m$,
the empirical $L^1$ distance between any two losses
$\ell,\widetilde\ell\in\mathcal L_{\mathrm{orb}}$ is bounded by
their uniform distance on this domain:
\[
\begin{aligned}
\|\ell-\widetilde\ell\|_{1,\mathbf X_m}
&=\frac1m\sum_{i=1}^m|\ell(X_i)-\widetilde\ell(X_i)| \le\max_{1\le i\le m}|\ell(X_i)-\widetilde\ell(X_i)|
\le\|\ell-\widetilde\ell\|_{\infty,B^d\setminus\{0\}}.
\end{aligned}
\]
Hence, $
N(\mathcal L_{\mathrm{orb}},\|\cdot\|_{1,\mathbf X_m},\varepsilon)
\le N(\mathcal L_{\mathrm{orb}},
\|\cdot\|_{\infty,B^d\setminus\{0\}},\varepsilon).
$

Now fix $\varepsilon>0$ and set
$\eta:=\varepsilon/C_{\mathrm{orb}}$.
Let $\mathcal H_\eta\subseteq\mathfrak C\times\mathcal G$ be a proper
$\eta$-cover of minimum cardinality with respect to $D_{\mathrm{orb}}$,
so that $|\mathcal H_\eta|=
N(\mathfrak C\times\mathcal G,D_{\mathrm{orb}},\eta)$.
We construct a corresponding set of loss functions
$\mathcal L_{\mathcal H_\eta}:=
\{\ell_{K,g}:(K,g)\in\mathcal H_\eta\}$.
For any pair $(K,g)\in\mathfrak C\times\mathcal G$, there exists an
element $(K_\eta,g_\eta)\in\mathcal H_\eta$ such that
$D_{\mathrm{orb}}((K,g),(K_\eta,g_\eta))\le\eta$.
By Lemma~\ref{lem:lipschitz-loss},
\[
\|\ell_{K,g}-\ell_{K_\eta,g_\eta}\|_{\infty,B^d\setminus\{0\}}
\le C_{\mathrm{orb}}D_{\mathrm{orb}}((K,g),(K_\eta,g_\eta))
\le C_{\mathrm{orb}}\eta=\varepsilon.
\]
Consequently, $\mathcal L_{\mathcal H_\eta}$ forms a valid
$\varepsilon$-covering of $\mathcal L_{\mathrm{orb}}$ in the supremum
norm. We therefore conclude that
\[
\begin{aligned}
N(\mathcal L_{\mathrm{orb}},
\|\cdot\|_{\infty,B^d\setminus\{0\}},\varepsilon)
&\le |\mathcal L_{\mathcal H_\eta}| \le N\left(\mathfrak C\times\mathcal G,D_{\mathrm{orb}},
\frac{\varepsilon}{C_{\mathrm{orb}}}\right).
\end{aligned}
\]
Combining these covering estimates with the Rademacher and uniform
deviation bounds, and taking the infimum over
$\varepsilon\in(0,\widetilde\varepsilon]$, proves the theorem.
\end{proof}

\subsection{Application: Exponentially Parameterized Families} 

We illustrate the preceding uniform convergence bound for anisotropic profiles parameterized as $g=e^f$, where $f$ belongs to a structured class of bounded continuous functions. This parameterization allows for expressive families of functions while ensuring positivity of the anisotropic scaling profile. Fix $M\ge0$ and let
$\mathcal F_M$ be a nonempty class satisfying
\[
\mathcal F_M\subseteq
\{f\in C(\mathbb S^{d-1},\mathbb R):\|f\|_\infty\le M\}.
\]
We define the induced exponential class of scaling profiles by
\[
G_M:=\{e^f:f\in\mathcal F_M\},
\]
which satisfies $G_M\subseteq G_a^b$ with $a=e^{-M}$ and $b=e^M$.

The following result transfers metric-entropy bounds for $\mathcal F_M$ and $\mathfrak C$ to the product class $\mathfrak C\times G_M$, allowing them to be inserted directly into Theorem~\ref{thm:uniform-convergence}. We then apply this result to two examples: linear feature models and sparse
neural networks.

\begin{lemma}[Entropy transfer under exponential parameterization]\label{lem:exp-entropy-transfer}
Let $\mathcal{F}_M$ and $G_M$ be defined as above. Then for every $\eta > 0$,
\begin{align}
\log N\left(\mathfrak{C} \times G_M, D_\mathrm{orb}, \eta \right) \le \log N\left(\mathfrak{C}, \delta, \frac{\eta}{2} \right) + \log N\left(\mathcal{F}_M, \|\cdot\|_\infty, \frac{\eta}{2b} \right). \label{eq:exp-metric-entropy}
\end{align}
\end{lemma}
\begin{proof}
If either covering number on the right-hand side of
\eqref{eq:exp-metric-entropy} is infinite, the claim is immediate.
Assume therefore that both are finite.
Since the derivative of $x\mapsto e^x$ is bounded above by $b$ on
$[-M,M]$, the mean value theorem gives
$\|e^f-e^{\widetilde f}\|_\infty
\le b\|f-\widetilde f\|_\infty$
for all $f,\widetilde f\in\mathcal F_M$.
Thus, any proper $\eta/b$-cover of
$(\mathcal F_M,\|\cdot\|_\infty)$ induces a proper $\eta$-cover
of $(G_M,\|\cdot\|_\infty)$, so
\[
N(G_M,\|\cdot\|_\infty,\eta)
\le N(\mathcal F_M,\|\cdot\|_\infty,\eta/b).
\]
Choose proper $\eta/2$-covers
$\mathcal H_{\eta/2}^{\mathfrak C}$ and
$\mathcal H_{\eta/2}^{G}$ of $(\mathfrak C,\delta)$ and
$(G_M,\|\cdot\|_\infty)$ both of minimum cardinality.
For any $(K,g)\in\mathfrak C\times G_M$, there exists
$(L,h)\in\mathcal H_{\eta/2}^{\mathfrak C}
\times\mathcal H_{\eta/2}^{G}$ such that
$\delta(K,L)\le\eta/2$ and $\|g-h\|_\infty\le\eta/2$.
By definition of the product metric,
$D_{\mathrm{orb}}((K,g),(L,h))
=\delta(K,L)+\|g-h\|_\infty\le\eta/2+\eta/2=\eta$.
Consequently, this Cartesian product is a proper $\eta$-cover of
$(\mathfrak C\times G_M,D_{\mathrm{orb}})$, giving
$$
N(\mathfrak C \times G_M, D_\mathrm{orb}, \eta) \le N\left(\mathfrak C , \delta, \frac \eta 2\right) N\left(G_M , \| \cdot \|_\infty, \frac \eta 2\right) \le N\left(\mathfrak C , \delta, \frac \eta 2\right) N\left(\mathcal F_M , \| \cdot \|_\infty, \frac{\eta}{2b} \right).
$$
Taking logarithms yields \eqref{eq:exp-metric-entropy}.
\end{proof}

\paragraph{Linear Feature Models}

We consider the profile parameterization used in
Section~\ref{sec:mnist-denoising}, treating the feature map as fixed.
Fix an integer $p\ge1$ and $R\ge0$. Let
$\varphi:\mathbb S^{d-1}\to\mathbb R^p$ be a continuous feature map
with $B_\varphi:=\sup_{u\in\mathbb S^{d-1}}\|\varphi(u)\|_2<\infty$,
and restrict parameters $\theta\in\mathbb R^p$ to the ball
$\Theta_R:=\{\theta\in\mathbb R^p:\|\theta\|_2\le R\}$.
Define the base linear function class as
$$
\mathcal F_{\mathrm{lin}}(R) := \{f_\theta: u \mapsto -\theta^\top \varphi(u) : \theta \in \Theta_R \}.
$$
By Cauchy--Schwarz,
$|f_\theta(u)|\le\|\theta\|_2\|\varphi(u)\|_2\le RB_\varphi$
uniformly. Thus, taking
$\mathcal F_M=\mathcal F_{\mathrm{lin}}(R)$ and $M:=RB_\varphi$
satisfies the required uniform bound. The corresponding exponential
profile class
$$
G_{\mathrm{lin}}(R) := \{e^{f_\theta} : f_\theta \in \mathcal F_{\mathrm{lin}}(R)\}
$$ satisfies $G_{\mathrm{lin}}(R) \subseteq G_a^b$ with $a = e^{-M}$ and $b = e^{M}$. 

For the star-body hypothesis class, we consider finite unions of
convex bodies with fixed volume and a common inscribed ball.
For $d\ge2$, $L\in\mathbb N$ with $L\ge1$, and $r>0$, let
$$
    \mathcal C^d(r,1)
    :=\{K\in\mathcal C^d:rB^d\subseteq K,\ 
       \operatorname{vol}_d(K)=1\},
    \qquad
    \mathcal S_L^d(r)
    :=\left\{\bigcup_{i=1}^L K_i:
       K_i\in\mathcal C^d(r,1)\right\},
$$
where $\mathcal C^d$ denotes the class of convex bodies in
$\mathbb R^d$. We assume $\mathcal C^d(r,1)$ is nonempty.
\begin{corollary}[Uniform convergence for linear feature models]
\label{cor:linear-features}
Let $P$ be a probability distribution supported on $B^d$ with
$P(\{0\})=0$. Consider the product hypothesis class
$\mathcal S_L^d(r)\times G_{\mathrm{lin}}(R)$, where
$\mathcal S_L^d(r)$ is the class of unions of $L$ volume-one convex
bodies containing $rB^d$. Fix $d,L,r,R$ and the feature map $\varphi$
as above. For each sufficiently large $m$ and every $\gamma\in(0,1)$,
with probability at least $1-\gamma$ over $m$ i.i.d.\ samples from
$P$, the bound
$$
|F(K,g;P)-F(K,g;P_m)|
\le\mathcal O\left(
\sqrt{\frac{p\log m+\log(8/\gamma)}m}
+\left(\frac Lm\right)^{2/(d+3)}\right)
$$
holds simultaneously for all $(K,g) \in \mathcal S^d_L(r) \times G_\mathrm{lin}(R)$, where the implicit constant is independent of $m$ and $\gamma$.
\end{corollary}

\begin{proof}
For $\theta,\widetilde\theta\in\Theta_R$, we have
$\|f_\theta-f_{\widetilde\theta}\|_\infty
\le B_\varphi\|\theta-\widetilde\theta\|_2$, so a proper
$\rho$-cover of $\Theta_R$ in the Euclidean norm induces a proper
$B_\varphi\rho$-cover of $\mathcal F_{\mathrm{lin}}(R)$ in the
uniform norm. By a standard volumetric argument,
$N(\Theta_R,\|\cdot\|_2,\rho)\le(1+2R/\rho)^p$
\cite[Corollary 4.2.13]{vershynin2018high}.
If $B_\varphi=0$, the profile class is a singleton and its entropy
is zero. Otherwise, setting
$\rho=\varepsilon/(2bB_\varphi C_{\mathrm{orb}})$ and applying
Lemma~\ref{lem:exp-entropy-transfer} yields the product entropy bound
for general $\mathfrak C$:
$$
\log N\left(\mathfrak{C} \times G_\mathrm{lin}(R), D_\mathrm{orb}, \frac{\varepsilon}{C_\mathrm{orb}} \right) \le \log N\left(\mathfrak{C}, \delta, \frac{\varepsilon}{2C_\mathrm{orb}} \right) + p\log \left( 1 + \frac{4 R B_\varphi b C_\mathrm{orb}}{\varepsilon}\right).
$$
We now specialize this estimate to $\mathfrak C=\mathcal S_L^d(r)$. One can show that convexity, the inclusion $rB^d\subseteq K$, and the volume constraint imply that every $K\in\mathcal C^d(r,1)$ is contained in $R_rB^d$ for a constant $R_r$ depending only on $d$ and $r$. Let $d_H$ denote the usual Hausdorff metric. Bronshtein's entropy bound, as stated in \cite{guntuboyina2012optimal}, gives
$$
\log N\bigl(\{K\in\mathcal C^d:K\subseteq R_rB^d\},d_H,\eta\bigr)
\lesssim_d \left(\frac{R_r}{\eta}\right)^{(d-1)/2}
$$
for sufficiently small $\eta>0$. For $K,K'\in\mathcal C^d(r,1)$, writing $t=d_H(K,K')$, the inclusions
$$
K\subseteq K'+tB^d\subseteq (1+t/r)K',
\qquad
K'\subseteq (1+t/r)K
$$
imply $\delta(K,K')\le (R_r/r)d_H(K,K')$. Take a Hausdorff cover of radius $s>0$ of the ambient convex-body class, with centers $A_j$. Discard balls that do not intersect
$\mathcal C^d(r,1)$, and choose a representative
$K_j\in\mathcal C^d(r,1)$ from each remaining ball.
For any $K\in\mathcal C^d(r,1)$ in the same ball,
\[
d_H(K,K_j)
\le d_H(K,A_j)+d_H(A_j,K_j)
\le 2s.
\]
Thus, the representatives form a proper Hausdorff $2s$-cover of
$\mathcal C^d(r,1)$ with no more centers than the original cover. Combining this observation with the metric comparison yields
$$
\log N(\mathcal C^d(r,1),\delta,\eta)
\lesssim_{d,r}\eta^{-(d-1)/2}.
$$
Since $\rho_{\cup_{i=1}^L K_i}=\max_{1\le i\le L}\rho_{K_i}$,
taking unions of elements of a proper cover of $\mathcal C^d(r,1)$
gives a proper cover of $\mathcal S_L^d(r)$. Consequently,
\begin{equation}
\log N(\mathcal S_L^d(r),\delta,\eta)
\le L\log N(\mathcal C^d(r,1),\delta,\eta)
\lesssim_{d,r}L\eta^{-(d-1)/2}
\label{eq:union-radial-entropy}
\end{equation}
for sufficiently small $\eta>0$. Thus, for sufficiently small $\varepsilon>0$,
$$
\log N\left(\mathcal S_L^d(r)\times G_{\mathrm{lin}}(R),D_{\mathrm{orb}},\frac{\varepsilon}{C_{\mathrm{orb}}}\right)
\lesssim L\varepsilon^{-(d-1)/2}
+p\log\left(1+\frac{4RB_\varphi bC_{\mathrm{orb}}}{\varepsilon}\right),
$$
where the implicit constant depends only on the fixed structural parameters. For $m$ sufficiently large, choose $\varepsilon=\left(L/m\right)^{2/(d+3)}\le\min\{\widetilde\varepsilon,1\}.$ Then 
$$
\sqrt{\frac{L}{m}\varepsilon^{-(d-1)/2}}
=\left(\frac{L}{m}\right)^{2/(d+3)}=\varepsilon,
\quad \text{and} \quad 
p\log\left(1+\frac{4RB_\varphi bC_{\mathrm{orb}}}{\varepsilon}\right)
\lesssim p\log m.
$$
To complete the proof, we substitute these estimates into
Theorem~\ref{thm:uniform-convergence}. Note that
$c_{P,m}^{\mathrm{orb}}\le B_{\mathrm{orb}}$ and
$c_{\mathrm{orb}}=4B_{\mathrm{orb}}$ are uniformly bounded
independently of $m$. The additive $2\varepsilon$ term has the same
order as the geometric contribution. This yields the stated rate,
with the profile class contributing an additional
$\mathcal O(\sqrt{p\log m/m})$ term to the uniform convergence bound.
\end{proof}
\paragraph{Sparse ReLU Networks with Bounded Outputs}
Let $\mathcal F(L_\mathrm{NN},\mathbf p,s)$ denote the class of ReLU networks on $[0,1]^d$ with depth $L_\mathrm{NN}$, layer-width vector $\mathbf p=(p_0,p_1,\ldots,p_{L_\mathrm{NN}+1})$ with $p_0=d$ and $p_{L_\mathrm{NN}+1}=1$, at most $s$ nonzero parameters, and all weights and biases bounded in absolute value by one. Define the architectural constant $V:=\prod_{\ell=0}^{L_\mathrm{NN}+1}(p_\ell+1)$. To obtain functions on $\mathbb S^{d-1}$, let $T:\mathbb S^{d-1}\to[0,1]^d$ be the affine map $T(u):=(u+\mathbf 1)/2$, where $\mathbf 1=(1,\ldots,1)^\top\in\mathbb R^d$. We define the corresponding bounded network class on the sphere by
$$
\mathcal F_{\mathrm{NN}}(L_{\mathrm{NN}},\mathbf p,s,M) := \left\{ F\circ T: F\in\mathcal F(L_{\mathrm{NN}},\mathbf p,s), \ \|F\circ T\|_\infty\le M \right\},
$$
and the associated exponential profile class by $G_{\mathrm{NN}} := \left\{ e^f: f\in \mathcal F_{\mathrm{NN}}(L_\mathrm{NN},\mathbf p,s,M) \right\} \subseteq G_a^b$, with $a = e^{-M}$ and $b = e^M$.

\begin{corollary}[Uniform convergence for sparse ReLU networks]
\label{cor:relu-networks}
Let $P$ be a probability distribution supported on $B^d$ with
$P(\{0\})=0$. Consider the product hypothesis class
$\mathcal S^d_L(r)\times G_{\mathrm{NN}}$, where $G_{\mathrm{NN}}$
is the sparse ReLU profile class defined above. Under the assumptions
of Theorem~\ref{thm:uniform-convergence}, fix structural parameters
$d,L,r,L_{\mathrm{NN}},\mathbf p,M$ and let $b=e^M$. For each
sufficiently large $m$ and every $\gamma\in(0,1)$, with probability
at least $1-\gamma$ over $m$ i.i.d.\ samples from $P$, the bound
$$
|F(K,g;P)-F(K,g;P_m)| \le \mathcal{O}\left(  \sqrt{\frac{(s+1)\log m+\log(8/\gamma)}{m}}  + \left(\frac{L}{m}\right)^{2/(d+3)} \right)
$$
holds simultaneously for all $(K,g) \in \mathcal S^d_L(r) \times G_{\mathrm{NN}}$, where the implicit constant is independent of $m$ and $\gamma$.
\end{corollary}

\begin{proof}
To construct a proper cover of
$\mathcal F_{\mathrm{NN}}(L_{\mathrm{NN}},\mathbf p,s,M)$, consider
the subclass
\[
\mathcal F_{[0,1]^d,M}
:=
\left\{
F\in\mathcal F(L_{\mathrm{NN}},\mathbf p,s):
\|F\circ T\|_{\infty,\mathbb S^{d-1}}\le M
\right\}.
\]
For $0<\eta\le1$, let $\{Z_j\}$ be an external $\eta/2$-cover
of minimum cardinality of $\mathcal F(L_{\mathrm{NN}},\mathbf p,s)$
in $\|\cdot\|_{\infty,[0,1]^d}$. By
\cite[Lemma 5]{schmidthieber2020nonparametric}, the number of
covering balls is bounded by
$\left(4\eta^{-1}(L_{\mathrm{NN}}+1)V^2\right)^{s+1}$. For each ball $B(Z_j,\eta/2)$ that intersects $\mathcal F_{[0,1]^d,M}$, select a representative
$F_j\in B(Z_j,\eta/2)\cap\mathcal F_{[0,1]^d,M}$ and set
$f_j:=F_j\circ T\in
\mathcal F_{\mathrm{NN}}(L_{\mathrm{NN}},\mathbf p,s,M)$.
For any $f=F\circ T\in
\mathcal F_{\mathrm{NN}}(L_{\mathrm{NN}},\mathbf p,s,M)$,
choose a covering ball containing $F$. Since
$T(\mathbb S^{d-1})\subseteq[0,1]^d$, the triangle inequality gives
\[
\begin{aligned}
\|f-f_j\|_{\infty,\mathbb S^{d-1}}
&\le \|F-F_j\|_{\infty,[0,1]^d} \le \|F-Z_j\|_{\infty,[0,1]^d}
    +\|Z_j-F_j\|_{\infty,[0,1]^d}
\le\eta.
\end{aligned}
\]
Thus, $\{f_j\}$ forms a proper $\eta$-cover of
$\mathcal F_{\mathrm{NN}}(L_{\mathrm{NN}},\mathbf p,s,M)$
on $\mathbb S^{d-1}$, and
\[
\log N\left(
\mathcal F_{\mathrm{NN}}(L_{\mathrm{NN}},\mathbf p,s,M),
\|\cdot\|_\infty,\eta
\right)
\le
(s+1)\log\left(4\eta^{-1}(L_{\mathrm{NN}}+1)V^2\right).
\]
For $0<\varepsilon\le2bC_{\mathrm{orb}}$, applying
Lemma~\ref{lem:exp-entropy-transfer} with
$\eta=\varepsilon/(2bC_{\mathrm{orb}})$ gives
\[
\begin{aligned}
\log N\left(
\mathfrak C\times G_{\mathrm{NN}},D_{\mathrm{orb}},
\frac{\varepsilon}{C_{\mathrm{orb}}}
\right)
&\le
\log N\left(
\mathfrak C,\delta,\frac{\varepsilon}{2C_{\mathrm{orb}}}
\right) + (s+1)\log\left(
\frac{8bC_{\mathrm{orb}}(L_{\mathrm{NN}}+1)V^2}{\varepsilon}
\right).
\end{aligned}
\]
Now set $\mathfrak C=\mathcal S_L^d(r)$. Using the geometric
entropy bound \eqref{eq:union-radial-entropy} gives, for sufficiently
small $\varepsilon>0$,
\[
\begin{aligned}
\log N\left(
\mathcal S_L^d(r)\times G_{\mathrm{NN}},D_{\mathrm{orb}},
\frac{\varepsilon}{C_{\mathrm{orb}}}
\right)
&\lesssim
L\varepsilon^{-(d-1)/2} +(s+1)\log\left(
\frac{8bC_{\mathrm{orb}}(L_{\mathrm{NN}}+1)V^2}{\varepsilon}
\right).
\end{aligned}
\]
For $m$ sufficiently large that the preceding entropy estimates
apply, choose
$\varepsilon=(L/m)^{2/(d+3)}\le\widetilde\varepsilon$. Then
\[
\sqrt{\frac{L}{m}\varepsilon^{-(d-1)/2}}
=\left(\frac{L}{m}\right)^{2/(d+3)}
=\varepsilon,
\quad \text{and} \quad
(s+1)\log\left(
\frac{8bC_{\mathrm{orb}}(L_{\mathrm{NN}}+1)V^2}{\varepsilon}
\right)
\lesssim(s+1)\log m,
\]
since the architectural parameters are fixed as $m\to\infty$.
Substituting these estimates into
Theorem~\ref{thm:uniform-convergence}, using
$c_{P,m}^{\mathrm{orb}}\le B_{\mathrm{orb}}$ and
$c_{\mathrm{orb}}=4B_{\mathrm{orb}}$, gives the stated bound.
The additive term $2\varepsilon$ has the same order as the
geometric contribution.
\end{proof}

\subsection{Estimation Error Bounds}

The preceding results bound the deviation between empirical and
population risks. We now use these bounds to control the estimation
error of the population-optimal scaling profile for a fixed star body
$K\in\mathcal S^d$. Retain the standing assumptions that $P$ is
supported on $B^d$ with $P(\{0\})=0$, and suppose in addition that
$P\in\mathcal P$. By Theorem~\ref{thm:optimal-g-K}, the optimal
population scaling profile is
$$
g_K^*(u)=\frac{m_P(u)}{d\rho_K(u)}.
$$
We first show that the population excess risk controls the squared
$L^2(P_U)$ distance to $g_K^*$, where $P_U$ is the directional
marginal of $P$.
 \begin{lemma} \label{lem:g-gK-bound}
    Fix $P \in \mathcal{P}$ and suppose $g^*_K \in G_a^b$. Then there exists a constant $\kappa_{a,b} \in (0,\infty)$ such that \begin{align*}
        \|g - g^*_K\|^2_{L^2(P_U)}:= \mathbb{E}_{P_U}[(g(u)-g^*_K(u))^2]\leq \frac{b^2}{d\kappa_{a,b}}(F(K,g;P)-F(K,g^*_K;P)),\ \quad \forall g \in G_a^b.
    \end{align*}
\end{lemma}
\begin{proof}
    Observe that we may write \begin{align*}
        F(K,g;P) = \mathbb{E}_{u \sim P_U}\left[d \log g(u) + \frac{m_P(u)}{\rho_K(u)g(u)}\right] = d\mathbb{E}_{u\sim P_U}\left[\log g(u) + \frac{g_K^*(u)}{g(u)}\right].
    \end{align*} Set $t(u):=g(u)/g_K^*(u)$ and
$\psi(t):=\log t+t^{-1}-1$ for $t>0$. Then \begin{align*}
        F(K,g;P) - F(K,g^*_K;P) = d\mathbb{E}_{P_U}[\psi(t(u))]
    \end{align*} 
    Since $a\le g,g_K^*\le b$, we have $t(u)\in[a/b,b/a]$.
If $a=b$, every profile in $G_a^b$ equals $a$, and the conclusion
holds with $\kappa_{a,a}:=1/2$. Assume henceforth that $a<b$ and define
\[
\Psi(t):=
\begin{cases}
\displaystyle\frac{\psi(t)}{(t-1)^2},&t\ne1,\\
\displaystyle 1/2,&t=1.
\end{cases}
\]
We show that $\Psi$ is strictly decreasing on $(0,\infty)$.
Let $\widetilde\psi(t):=(t-1)^2/t^2-2\psi(t)$. Since
$\psi'(t)=(t-1)/t^2$, for $t\ne1$,
\[
\Psi'(t)
=
\frac{(t-1)\psi'(t)-2\psi(t)}{(t-1)^3}
=
\frac{\widetilde\psi(t)}{(t-1)^3}.
\]
The function $\widetilde\psi$ is strictly decreasing, since
$\widetilde\psi'(t)=-2(t-1)^2/t^3\le0$, with equality only at
$t=1$. As $\widetilde\psi(1)=0$, its value is positive for
$t<1$ and negative for $t>1$. Thus $\Psi'(t)<0$ on both
$(0,1)$ and $(1,\infty)$. Moreover,
$\psi(1)=\psi'(1)=0$ and $\psi''(1)=1$, so Taylor's theorem
gives $\lim_{t\to1}\Psi(t)=1/2$. The extension is therefore
continuous and strictly decreasing on $(0,\infty)$. Consequently,
\[
\kappa_{a,b}
:=
\inf_{t\in[a/b,b/a]}\Psi(t)
=
\Psi(b/a)
=
\frac{\log(b/a)+a/b-1}{(b/a-1)^2}
\in(0,\infty).
\]
Positivity follows from $\psi(t)>0$ for $t\ne1$, which is
immediate from $\psi'(t)=(t-1)/t^2$ and $\psi(1)=0$.
Using this scalar lower bound and $g_K^*\le b$, we obtain
\[
\begin{aligned}
F(K,g;P)-F(K,g^*_K;P) \ge d\kappa_{a,b}\,
\mathbb E_{P_U}\!\left[(t(U)-1)^2\right]
&=d\kappa_{a,b}\,
\mathbb E_{P_U}\!\left[
\left(\frac{g(U)-g_K^*(U)}{g_K^*(U)}\right)^2
\right]\\
&\ge \frac{d\kappa_{a,b}}{b^2}
\|g-g_K^*\|_{L^2(P_U)}^2.
\end{aligned}
\]
Rearranging proves the claim.
\end{proof}

The following theorem combines this result with a uniform deviation
bound to control the estimation error of approximate empirical risk
minimizers. A uniform Lipschitz bound on the profiles also yields a
sup-norm estimate under the directional mass condition below, which
holds for $P\in\mathcal P$ when $d\ge2$.

\begin{theorem}
Under the assumptions of Lemma~\ref{lem:g-gK-bound}, let
$\mathcal G\subseteq G_a^b$ be a nonempty profile class containing
$g_K^*$. Define the uniform deviation
$\Delta_m:=\sup_{h\in\mathcal G}|F(K,h;P_m)-F(K,h;P)|$ and the
empirical optimization error
$\varepsilon_{\mathrm{opt}}(g):=
F(K,g;P_m)-\inf_{h\in\mathcal G}F(K,h;P_m)$ for $g\in\mathcal G$.
Then
\[
\|g-g_K^*\|_{L^2(P_U)}^2
\le
\frac{b^2}{d\kappa_{a,b}}
\bigl(2\Delta_m+\varepsilon_{\mathrm{opt}}(g)\bigr),
\qquad g\in\mathcal G.
\]
For $\Lambda>0$, let $\mathcal G(\Lambda)$ consist of the
$\Lambda$-Lipschitz profiles in $\mathcal G$, with respect to
Euclidean distance on $\mathbb S^{d-1}$. Write $B(u,\varepsilon)$
for the closed Euclidean ball with center $u$ and radius
$\varepsilon$. Suppose $g_K^*\in\mathcal G(\Lambda)$ and, for
some $c,\nu,\varepsilon_0>0$,
$P_U(B(u,\varepsilon)\cap\mathbb S^{d-1})\ge c\varepsilon^\nu$
for every $u\in\mathbb S^{d-1}$ and $0<\varepsilon\le\varepsilon_0$.
Then
\[
\|g-g_K^*\|_\infty^{2+\nu}
\le
\frac{4b^2}{
cd\kappa_{a,b}
\min\{(4\Lambda)^{-\nu},\varepsilon_0^\nu b^{-\nu}\}
}
\bigl(2\Delta_m+\varepsilon_{\mathrm{opt}}(g)\bigr),
\qquad g\in\mathcal G(\Lambda).
\]
\end{theorem}

\begin{proof}
For any $g\in\mathcal G$, we have
\[
\begin{aligned}
F(K,g;P)-F(K,g^*_K;P)= \bigl[F(K,g;P)-F(K,g;P_m)\bigr] &+
\bigl[F(K,g;P_m)-F(K,g_K^*;P_m)\bigr]\\
&+ \bigl[F(K,g_K^*;P_m)-F(K,g_K^*;P)\bigr]\\
&\le 2\Delta_m+\varepsilon_{\mathrm{opt}}(g).
\end{aligned}
\]
Here, the first and third terms are bounded above by $\Delta_m$,
and the middle term is bounded above by
$\varepsilon_{\mathrm{opt}}(g)$ because $g_K^*\in\mathcal G$.
Applying Lemma~\ref{lem:g-gK-bound} gives the first bound.

For the second bound, fix $g\in\mathcal G(\Lambda)$ and set
$h:=g-g_K^*$. Then $h$ is $2\Lambda$-Lipschitz and
$\|h\|_\infty\le b-a\le b$. By continuity and compactness, there
exists $u_0\in\mathbb S^{d-1}$ such that
$|h(u_0)|=\|h\|_\infty$. The case $h\equiv0$ is immediate, so
assume $\|h\|_\infty>0$. Set
$\varepsilon:=\min\{\varepsilon_0,\|h\|_\infty/(4\Lambda)\}$
and $S_\varepsilon:=B(u_0,\varepsilon)\cap\mathbb S^{d-1}$.
For every $u\in S_\varepsilon$,
\[
|h(u)|
\ge |h(u_0)|-2\Lambda\|u-u_0\|_2
\ge \tfrac12\|h\|_\infty.
\]
Thus the error remains at least half its maximum throughout this cap.
Using the directional mass condition gives
\[
\begin{aligned}
\|h\|_{L^2(P_U)}^2
=\int_{\mathbb S^{d-1}}|h(u)|^2\,\mathrm dP_U(u)
\ge \int_{S_\varepsilon}|h(u)|^2\,\mathrm dP_U(u)
&\ge \frac14\|h\|_\infty^2P_U(S_\varepsilon)\\
&\ge \frac{c}{4}\varepsilon^\nu\|h\|_\infty^2\\
&\ge \frac{c}{4}
\min\{(4\Lambda)^{-\nu},\varepsilon_0^\nu b^{-\nu}\}
\|h\|_\infty^{2+\nu},
\end{aligned}
\]
where the final inequality uses the definition of $\varepsilon$
and $\|h\|_\infty\le b$. Consequently,
\[
\|g-g_K^*\|_\infty^{2+\nu}
\le
\frac{4}{
c\min\{(4\Lambda)^{-\nu},\varepsilon_0^\nu b^{-\nu}\}
}
\|g-g_K^*\|_{L^2(P_U)}^2,
\qquad g\in\mathcal G(\Lambda).
\]
Applying the first bound completes the proof.
\end{proof}

Whenever the hypotheses of Theorem~\ref{thm:uniform-convergence}
hold for the class $\{K\}\times\mathcal G$, its bound on $\Delta_m$
yields estimation error bounds on the same high-probability event.

\begin{remark}
For $d\ge2$, the directional mass condition holds whenever $P_U$
has a density with respect to surface measure that is uniformly
bounded below. In particular, if
$\mathrm dP_U(u)=h_P(u)\,\mathrm d\sigma(u)$ and
$h_{\min}:=\inf_{u\in\mathbb S^{d-1}}h_P(u)>0$, then the
spherical-cap estimate gives
\[
P_U(B(u,\varepsilon)\cap\mathbb S^{d-1})
\ge h_{\min}c_d\varepsilon^{d-1},
\qquad u\in\mathbb S^{d-1},\quad 0<\varepsilon\le1,
\]
for some dimension-dependent constant $c_d>0$.
For every $P\in\mathcal P$, the density $h_P$ is continuous and
strictly positive on the compact sphere, so $h_{\min}>0$.
Thus the condition holds with $\nu=d-1$, $c=h_{\min}c_d$,
and $\varepsilon_0=1$.
\end{remark}

\section{Conclusion}

In this work, we developed a geometric framework for adapting regularizers to data through anisotropic dilation. We showed that these direction-preserving transformations act explicitly on the radial summary statistic that determines optimal regularization, allowing any admissible target statistic to be reached through a unique scaling profile. At the distributional level, their expressivity is characterized by the preservation of directional marginals and conditional radial shape up to scaling. This combination of flexibility and structure makes anisotropic dilation a tractable mechanism for modifying the geometry relevant to regularization. More concretely, we established that the central benefit of anisotropic dilation at the level of distributional fitting is that it separates directional mass from conditional radial scale. This allows us to establish explicit population-optimal models and an exact characterization of the approximation error incurred by restricting to isotropic scaling. We further developed extensions to homogeneous Gibbs energies and richer conditional radial families, established a raywise transport interpretation, and obtained finite-sample learning guarantees. Experiments demonstrate that learning the scaling profile can improve both handcrafted and learned regularizers in simple settings.

There are several interesting avenues of future research. A central direction is to connect improvements in distributional approximation and profile estimation to quantitative reconstruction guarantees for inverse problems when reconstructing signals from the underlying data distribution. Empirically showcasing that learning anisotropic dilation maps helps for more high-dimensional inverse problems would also be important. Moreover, while we established favorable geometry of the underlying variational problem in a simple setting (Theorem \ref{thm:better-landscape}), extending this theory more broadly to different choices of target regularizers along with establishing corresponding reconstruction guarantees would strengthen the results. Finally, a natural next step would be to investigate other families of maps that continue to be interpretable in the context of optimal regularization while potentially increasing flexibility of anisotropic dilation.

\section*{Acknowledgements}

We thank Venkat Chandrasekaran for helpful discussions. ChatGPT was used to proofread and polish the manuscript. Certain ideas were also suggested by ChatGPT, such as using the family of ray-wise couplings in Proposition \ref{prop:OT-interpretation} and the analytic continuation argument in Theorem \ref{thm:gaussianmixture}. Part of this work was conducted while OL was with the Department of Statistics and Data Science at UCLA.

\bibliographystyle{abbrv}
\bibliography{neurips_2026}

\section{Appendix}

\subsection{Experimental Details: Two-Dimensional Denoising}

\label{sec:appendix-2d}

\paragraph{Clean Data Distribution}
We use a uniform angular grid $\vartheta_j=-\pi+2\pi j/J$,
$j=0,\ldots,J-1$, with $J=192$. Writing $\langle a\rangle_J$ for the
grid average of a function $a$, define
$$
    \rho_0(\vartheta)
    =0.90\frac{\widetilde\rho_0(\vartheta)}
                   {\langle\widetilde\rho_0\rangle_J},
    \qquad
    g_0(\vartheta)
    =\frac{\widetilde g_0(\vartheta)}{\langle\widetilde g_0\rangle_J},
$$
where
\begin{align*}
    \log\widetilde\rho_0(\vartheta)
    &=0.70\cos(2\vartheta-0.25)+0.28\sin(3\vartheta+0.80)\\
    &\quad-0.18\cos(5\vartheta-0.30)+0.10\sin(7\vartheta+0.40),\\
    \log\widetilde g_0(\vartheta)
    &=0.55\sin(4\vartheta+0.35)-0.38\cos(2\vartheta-0.70)
      +0.18\sin(8\vartheta-0.20).
\end{align*}
The target angle has density $h_0(\vartheta)\propto\rho_0(\vartheta)^2$.
Conditional on $\Theta=\vartheta$, the standardized radius $S$ has density
$$
    f_{\eta_0(\vartheta)}(s)
    =\frac{s}{A(\eta_0(\vartheta))}
      \exp\left(-\sum_{\ell=1}^3\eta_{0,\ell}(\vartheta)s^\ell\right),
    \qquad s>0,
$$
where $A$ is the normalizing constant. We set
\begin{align*}
    \eta_{0,2}(\vartheta)
    &=0.04\sin(3\vartheta-0.40)-0.03\cos(5\vartheta+0.30),\\
    \eta_{0,3}(\vartheta)
    &=0.030+0.018\bigl(1+\sin(2\vartheta+0.55)\bigr),
\end{align*}
and determine $\eta_{0,1}$ numerically by bisection so that
$\mathbb E[S\mid\Theta=\vartheta]=2$. Thus, the base radius
$R_0=\rho_0(\Theta)S$ satisfies the normalization
\eqref{eq:baseline-constraint}, and the clean sample is
$X=g_0(\Theta)R_0u(\Theta)$.

Numerically, we select an angular bin with probability proportional to
$\rho_0(\vartheta_j)^2$ and sample the angle uniformly within that bin.
We draw $S$ from the radial law at the bin center by numerical inverse
CDF sampling, using trapezoidal quadrature on $1200$ radial points in
$[10^{-4},16]$. The profiles $\rho_0$ and $g_0$ are evaluated at the
sampled angle.

\paragraph{Model Fitting}
The directional density and conditional radial moments are estimated
from the clean training samples using the circular kernel
$k(\delta)=\exp(36\cos\delta)$, where $\delta$ is the angular separation.
The directional estimate is normalized to have grid mean one, and the
conditional moments are kernel-weighted averages of the first three
powers of the observed radii. We compute the radial profiles using the
expressions in Section~\ref{sec:two-d-denoising} and normalize the
estimated profiles to have grid mean $0.90$. The resulting profile
values are used directly on the angular quadrature grid. Scaling
profiles are clipped to $[0.08,12]$; no values required clipping in this
experiment.

For the shape models, the standardized second- and third-moment targets
are $\widehat m_{P,\ell}/(\widehat\rho_K\widehat g)^\ell$ for
$\ell=2,3$, clipped to $[4.2,80]$ and $[8.4,500]$, respectively.
Here, $\widehat g\equiv1$ for $(K,\eta)$. We minimize squared differences
between the logarithms of model and target moments, averaged over the
angular grid, using weights $4$, $1$, and $1/4$ for the first three
moments and first-moment target $2$. A penalty on squared periodic
second differences of $(\log\eta_1,\eta_2,\log\eta_3)$ is included with
weight $0.002$. The numerical parameterization requires $\eta_1,\eta_3>0$
and permits signed $\eta_2$. Moments are evaluated by trapezoidal
quadrature on $520$ points in $[10^{-4},12]$. We use $1400$ AdamW steps
with learning rate $0.035$ and weight decay $10^{-5}$.
The first-moment normalization is encouraged through a penalty and is
therefore approximate.

\paragraph{Denoising and Evaluation}
We use random seed $28$, with the same data splits and noisy observations
for every model. Posterior means are approximated on the $192$-point
angular grid, using $64$ radial points for validation and $88$ for
testing. The radial grid spans $[10^{-3},R_{\max}]$, where
$R_{\max}=3+1.15\max_x\|x\|_2\approx15.76$ and the maximum is taken over
all clean training, validation, and test samples. Thus, clean test radii
are used to set the numerical integration domain. Quadrature weights
include the polar Jacobian factor $r$.

For each model and noise level, $\gamma$ is selected by minimizing
validation MSE over
$$
    \Gamma_{\mathrm{grid}}
    =\{0.20,0.35,0.55,0.80,1.10,1.50,2.00,2.70,3.60,4.80\}.
$$
For the $N=650$ test pairs $(x_i,y_i)$, we report
$\operatorname{RMSE}
=[N^{-1}\sum_{i=1}^N
\|\widehat x_\gamma(y_i)-x_i\|_2^2]^{1/2}$.

\subsection{Experimental Details: MNIST} \label{sec:appendix-mnist}

\paragraph{Explicit details on scaling profile parameterization} For the anisotropic profile, recall that $g_{\theta}$ was of the form \begin{align*}
    g_{\theta}(u) := \exp\left(- \theta^{\top}\varphi(u)\right),\ \qquad \varphi(u) \in \R^{3}.
\end{align*} The feature vector we use is of the following form: \begin{align*}
    \varphi(u) = \left(\begin{array}{c}
         1  \\
         \varphi_1(u)  \\
         \varphi_2(u)
    \end{array}\right),\ \quad \varphi_1(u) := \frac{s(u) - m_{s}}{\sigma_{s}},\ \quad \varphi_2(u):= \frac{v(u) - m_{v}}{\sigma_v},
\end{align*} where $s(u)$ and $v(u)$ compute a proxy for a given input's spatial support and spread, respectively $$s(u) = \frac{\|u\|_1^2}{d}\ \text{and}\ v(u) = \sum_{ij} p_{ij}(u)\left\|z_{ij} - \sum_{ij}p_{ij}(u)z_{ij}\right\|_2^2.$$ Here, $z_{ij}$ denotes the spatial coordinate of the pixel $(i,j)$ and $p_{ij}(u) = \frac{u_{ij}}{\sum_{k\ell}u_{k\ell}}$. In our case, the pixel values are non-negative and $u \neq 0$. The means $m_s,m_v$ and standard deviations $\sigma_s,\sigma_v$ are estimated from the training data. In particular, for each image in the training distribution $x_i$, we compute $u_i = x_i/\|x_i\|_2$ for $i \in [N]$ and set \begin{align*}
    m_f := \frac{1}{N}\sum_{i=1}^N f(u_i)\ \quad \text{and} \quad \sigma_f := \sqrt{\frac{1}{N}\sum_{i=1}^N(f(u_i) - m_f)^2},\ \quad\ f \in \{s,v\}.
\end{align*}

\subsection{Additional Proofs}
\begin{lemma} \label{lem:technical-bound-convexity}
    Suppose $\Phi(z):= \psi(z/\|z\|_2) z$ where $\psi \in C^2(\mathbb{S}^{d-1})$, $\psi > 0$, and $\Phi(0)=0$. Then there exists a constant $C_{\psi} < \infty$ such that $$\|\Phi(z_1) + \Phi(z_2) - \Phi(z_1+z_2)\|_2 \leq C_{\psi}(\|z_1\|_2 + \|z_2\|_2 - \|z_1+z_2\|_2),\ \qquad \forall z_1,z_2\in\R^d.$$
\end{lemma}
\begin{proof}[Proof of Lemma~\ref{lem:technical-bound-convexity}]
We first collect several useful facts about the map $\Phi$. First, note that by a straightforward extension of \cite[Lemma D.1]{ji2020directional}, we have that since $\Phi$ is $1$-homogeneous and $C^2$ on $\R^d \setminus\{0\},$ its Jacobian and second-order derivative are $0$- and $-1$-homogeneous respectively, i.e., $D \Phi(tz) = D \Phi(z)$ and $D^2\Phi(tz) = t^{-1}D^2\Phi(z)$ for all $t > 0$ and $z \neq 0$. By continuity, we have that $$L_{\psi,1} := \sup_{\|u\|_2 = 1} \|D \Phi(u)\|_{\mathrm{op}} < \infty\ \text{and}\ L_{\psi,2} := \sup_{\|u\|_2 = 1} \|D^2\Phi(u)\|_{\mathrm{op}}:=\sup_{\|u\|_2=\|v\|_2=\|z\|_2=1}\|D^2\Phi(u)[v,z]\|_{2} < \infty.$$ Since $D\Phi(z) = D\Phi(z/\|z\|_2)$, we have that for all $z \neq 0$, $\|D\Phi(z)\|_{\mathrm{op}} \leq L_{\psi,1}$. Also, $\|D^2\Phi(z)\|_{\mathrm{op}} =\|D^2\Phi(z/\|z\|_2)\|_{\mathrm{op}} / \|z\|_2 \leq L_{\psi,2}/\|z\|_2.$ The former point also implies that $\Phi$ is Lipschitz continuous. Indeed, suppose $x,y \neq 0$ and consider the line segment $\gamma(t):= y + t(x-y)$ for $t \in [0,1]$. Suppose the line segment does not pass through the origin. Then by the Fundamental Theorem of Calculus, we have \begin{align*}
    \|\Phi(x) - \Phi(y)\|_2 \leq \int_0^1 \|D\Phi(\gamma(t))\|_{\mathrm{op}}\|x-y\|_2 \mathrm{d}t \leq L_{\psi,1}\|x-y\|_2.
\end{align*} When the line segment does pass through the origin, we have that there exists a $\lambda > 0$ such that $y = -\lambda x$ and $\|y-x\|_2 = (\lambda+1)\|x\|_2.$ In this case, set $M_{\psi}:= \sup_{\|u\|_2=1} \psi(u) < \infty$. Then note that $\|\Phi(x)\|_2 = \|x\|_2\psi(x/\|x\|_2)\leq M_{\psi}\|x\|_2$ while $\|\Phi(y)\|_2=\|\lambda\Phi(-x)\|_2  \leq \lambda M_{\psi}\|x\|_2$. This gives \begin{align*}
    \|\Phi(x)-\Phi(y)\|_2 \leq \|\Phi(x)\|_2 + \|\Phi(y)\|_2 \leq M_{\psi}(1 + \lambda)\|x\|_2 = M_{\psi}\|x-y\|_2
\end{align*} as desired. The case when either $x=0$ or $y=0$ follows similarly. Hence $\Phi$ is Lipschitz continuous with Lipschitz constant $L:= \max\left\{L_{\psi,1}, M_{\psi}\right\}$.

We now move on to the proof of the desired inequality. Suppose $z_1,z_2 \neq 0$ since otherwise the result is trivial. By homogeneity, it suffices to consider the case  when $\|z_1\|_2 + \|z_2\|_2 = 1$. Letting $\theta = \|z_1\|_2$, write $z_1 = \theta u$ and $z_2 = (1-\theta)v$ for $u,v \in \mathbb{S}^{d-1}$ and $m := z_1 + z_2 = \theta u + (1-\theta) v$.  We consider two cases:
\begin{itemize}
    \item Suppose $\delta := \|u-v\|_2 \leq 1/2.$ Then note that the line segment $\gamma(t) = (1-t)u + tv$ satisfies \begin{align*}
        \|\gamma(t)\|^2_2 = 1 - t(1-t)\delta^2 \geq 1 - \frac{\delta^2}{4} \geq \frac{15}{16}.
    \end{align*} 
 Hence all points along $\gamma(t)$ are uniformly bounded away from the origin for all $t \in [0,1]$. This along with our bound $\|D^2\Phi(z)\|_{\mathrm{op}} \leq L_{\psi,2}/\|z\|_2$ implies \begin{align*}
    \sup_{t \in [0,1]}\|D^2 \Phi(\gamma(t))\|_{\mathrm{op}} \leq \frac{16}{15}L_{\psi,2} =: C_{\psi,2}.
\end{align*}  Via a second-order Taylor expansion, we have \begin{align*}
    \|\theta \Phi(u) + (1-\theta)\Phi(v) - \Phi(m)\|_2 \leq C_{\psi,2} \theta(1-\theta)\|u-v\|_2^2 =C_{\psi,2} \theta(1-\theta)\delta^2.
\end{align*} But note that $\|m\|_2^2 = 1 - \theta(1-\theta)\delta^2$ which implies \begin{align}
    \theta(1-\theta)\delta^2 = 1 - \|m\|_2^2 = (1-\|m\|_2)(1+\|m\|_2) \leq 2(1-\|m\|_2). \label{eq:theta-delta-bound-to-m}
\end{align} Hence we attain \begin{align*}
    \|\theta \Phi(u) + (1-\theta)\Phi(v) - \Phi(m)\|_2 \leq C_{\psi,2} \theta(1-\theta)\delta^2 \leq 2C_{\psi,2}(1-\|m\|_2)
\end{align*} as desired. 
    \item Now assume $\delta > 1/2.$ Since $\Phi$ is $L$-Lipschitz, we have \begin{align*}
        \|\theta \Phi(u) + (1-\theta)\Phi(v) - \Phi(m)\|_2 & = \|\theta \Phi(u) - \theta \Phi(m)  + (1-\theta)\Phi(v) -(1-\theta) \Phi(m)\|_2 \\
        & \leq L(\theta\|u-m\|_2 + (1-\theta)\|v-m\|_2) \\
        & = 2L\theta(1-\theta)\delta
    \end{align*} where the last equality follows from the fact that $u - m = (1-\theta)(u-v)$ and $v - m = -\theta(u-v)$. Now, since $u,v \in \mathbb{S}^{d-1}$, we have $1/2 < \delta \leq 2$ so that $\delta \leq 4\delta^2$ which implies $$\|\theta \Phi(u) + (1-\theta)\Phi(v) - \Phi(m)\|_2  \leq 8L\theta(1-\theta)\delta^2 \leq 16L(1-\|m\|_2)$$ where we used \eqref{eq:theta-delta-bound-to-m} in the last inequality. 
\end{itemize}
Combining our two cases, we see that the desired inequality holds for all $z_1,z_2 \in \R^d$ if we choose $$C_{\psi} := \max\left\{16L, 2C_{\psi,2}\right\} = \max\left\{16\max\left\{\sup_{\|u\|_2 = 1} \|D \Phi(u)\|_{\mathrm{op}}, \sup_{\|u\|_2=1} \psi(u)\right\},\frac{32}{15}\sup_{\|u\|_2 = 1} \|D^2\Phi(u)\|_{\mathrm{op}}\right\}.$$
\end{proof}

\begin{proof}[Proof of Theorem~\ref{thm:gaussianmixture}]\label{proof:gaussianmixture}
Our goal is to prove that the diagnostic ratio in \eqref{eq: GMM-diagnostic-ratio} is constant if and only if the covariance matrices of the Gaussian mixture model are all proportional. We establish this equivalence by first showing that proportional covariance matrices yield a constant ratio, and then proving the converse. 
\paragraph{Proportional covariance matrices $\implies $ $\alpha_P(u)$ is constant}
Assuming $\Sigma_j=c_j\Sigma_0$, we have $a_j(u) = c_j^{-1} a_0(u)$ where $a_0(u):=u^\top \Sigma_0^{-1}u$. Furthermore, the determinants scale as $|\Sigma_j|^{-1/2} = |c_j \Sigma_0|^{-1/2} = c_j^{-d/2} |\Sigma_0|^{-1/2}$. Therefore, the denominator of the diagnostic ratio simplifies to
\begin{align*}
\sum_{j=1}^k w_j |\Sigma_j|^{-1/2} a_j(u)^{-d/2}
&= \sum_{j=1}^k w_j\,c_j^{-d/2}|\Sigma_0|^{-1/2}\,
(c_j^{-1}a_0(u))^{-d/2} \\
&= \sum_{j=1}^k w_j\,c_j^{-d/2}|\Sigma_0|^{-1/2}\, c_j^{d/2}a_0(u)^{-d/2} \\
&= |\Sigma_0|^{-1/2}a_0(u)^{-d/2}\sum_{j=1}^k w_j \\
&= |\Sigma_0|^{-1/2}a_0(u)^{-d/2}.
\end{align*}
Similarly, the term inside the numerator can be evaluated as
\begin{align*}
\sum_{j=1}^k w_j |\Sigma_j|^{-1/2} a_j(u)^{-(d+1)/2}
&= \sum_{j=1}^k w_j\,c_j^{-d/2}|\Sigma_0|^{-1/2}\, (c_j^{-1}a_0(u))^{-(d+1)/2} \\
&=
\sum_{j=1}^k w_j\,c_j^{-d/2}|\Sigma_0|^{-1/2}\, c_j^{(d+1)/2}a_0(u)^{-(d+1)/2} \\
&= |\Sigma_0|^{-1/2}a_0(u)^{-(d+1)/2} \sum_{j=1}^k w_jc_j^{1/2}.
\end{align*}
Plugging these back into the diagnostic ratio, we obtain
$$
\alpha_P(u)\propto \frac{ \left( |\Sigma_0|^{-1/2}a_0(u)^{-(d+1)/2} \sum_{j=1}^k w_jc_j^{1/2}\right)^{d/(d+1)}}{|\Sigma_0|^{-1/2}a_0(u)^{-d/2}}.
$$
We see that the $u$-dependency in the numerator, given by $\left(a_0(u)^{-(d+1)/2}\right)^{d/(d+1)}$, is the same as the $u$-dependency in the denominator $a_0(u)^{-d/2}$. Therefore, they cancel exactly in the ratio, meaning $\alpha_P(u)$ is constant on $\mathbb S^{d-1}$.
\paragraph{$\alpha_P(u)$ is constant $\implies$ Proportional covariance matrices}
When $d=1$, every covariance matrix is a positive scalar and is thus trivially proportional to any other. Hence, we may assume $d \ge 2$. Our general strategy is to define an algebraic function representing the constancy of $\alpha_P(u)$, analytically extend it to the complex space, and show that the non-proportionality of any two covariance matrices yields a contradiction. 

The proof of this converse involves five main steps: (1) assuming the covariances group into $m \ge 2$ distinct proportionality classes, we formulate the constancy of $\alpha_P(u)$ as an algebraic function on $\mathbb{R}^d \setminus \{0\}$ involving fractional powers of real quadratic forms; (2) we extend these quadratic forms to the complex domain $\mathbb{C}^d$ and isolate a singularity unique to one form; (3) we construct a local biholomorphism in a neighborhood of this singularity and define a complex covering space; (4) we analytically continue our constancy function along a path—lifted into the covering space—from $\mathbb{R}^d \setminus \{0\}$ to a neighborhood of this singularity; and (5) we expand the analytically continued function into a power series to reveal a contradictory  mismatch in the exponents, thereby forcing $m=1$.

\paragraph{Step 1: Constancy function and proportionality classes.}

We assume $d \ge 2$. Let $x = ru \in \mathbb{R}^d \setminus \{0\}$, and define the unnormalized numerator and denominator functions on the Euclidean space:
\begin{align}
    N(x) = \sum_{j=1}^k w_j | \Sigma_j |^{-1/2} a_j(x)^{-(d+1)/2}, \quad
    D(x) = \sum_{j=1}^k w_j | \Sigma_j |^{-1/2} a_j(x)^{-d/2}. \label{eq: num-denom}
\end{align}
The assumption that the diagnostic ratio $\alpha_P(u)$ is constant implies there exists a $K > 0$ such that $N(u)^d = KD(u)^{d+1}$ for all $u \in \mathbb{S}^{d-1}$. By the homogeneity of the quadratic forms, $a_j(x) = r^2 a_j(u)$, and thus this constancy relation holds for all non-zero real vectors
\begin{align}
    N(x)^d = KD(x)^{d+1}, \qquad \text{for all } x \in \mathbb{R}^{d} \setminus \{0\}.\label{eq: constancy-extended}
\end{align}
To frame this algebraically, we define the constancy function
\begin{align}
F(x) := N(x)^d - KD(x)^{d+1}, \qquad x \in \mathbb{R}^{d} \setminus \{0\}, \label{eq: constancy-function}
\end{align}
so that $F \equiv 0$ on $\mathbb{R}^{d} \setminus \{0\}$. For the sake of contradiction, we assume that the covariance matrices belong to $m \ge 2$ distinct proportionality classes represented by $\{ \bar \Sigma_1, \ldots, \bar \Sigma_m \}$. Let $I_\ell \subset \{1, \ldots, k\}$ denote the indices of the components in the $\ell$-th class, thus there exists some $c_j$ such that $\Sigma_j = c_j \bar{\Sigma}_\ell$ for all $j \in I_\ell$. We define the corresponding quadratic forms $q_\ell(x) = x^\top \bar{\Sigma}_\ell^{-1} x$. By aggregating the terms in \eqref{eq: num-denom} within each proportionality class, we obtain
\begin{align}
N(x) = \sum_{\ell = 1}^m n_\ell q_\ell(x)^{-(d+1)/2} \quad \text{and} \quad D(x) = \sum_{\ell = 1}^m \delta_\ell q_\ell(x)^{-d/2}, \label{eq: aggregated-ND}
\end{align}
where the constants are given by $n_\ell = |\bar \Sigma_\ell |^{-1/2} \sum_{j \in I_\ell} w_j c_j^{1/2} > 0$ and $\delta_\ell = |\bar \Sigma_\ell|^{-1/2} \sum_{j \in I_\ell} w_j > 0$. Our strategy is to analytically continue identity \eqref{eq: constancy-function} to a complex covering space $\Omega$ and analyze its behavior near an isolated singularity, showing that the power mismatch of $N$ and $D$ yields a fatal misalignment in the power series terms.
\paragraph{Step 2: Isolating a singularity of complexified quadratic forms.}
We first extend each quadratic form $q_\ell$ to a complex homogeneous polynomial $q_\ell(z) = z^\top \bar \Sigma_\ell^{-1} z \in \mathbb{C}[z_1, \dots, z_d]$. In this domain, the singularities of \eqref{eq: constancy-extended} correspond to the zero loci of these polynomials, which carve out complex quadratic cones
$$
C_\ell = \{z \in \mathbb{C}^d : q_\ell (z) = 0 \} \qquad \text{for } \ell = 1, \ldots, m. 
$$
Since each $\bar{\Sigma}_\ell$ is positive definite, $q_\ell(x) > 0$ for all real $x \neq 0$, meaning the non-trivial points of $C_\ell$ lie strictly in the complex domain $\mathbb{C}^d \setminus \mathbb{R}^d$. Two cones $C_1$ and $C_\ell$ coincide if and only if $q_1$ and $q_\ell$ are scalar multiples. We claim that $C_1 \not\subset \bigcup_{\ell=2}^m C_\ell$, which allows us to isolate a singularity in $C_1$.

\begin{enumerate}
    \item Case $d=2$: Every positive definite binary quadratic form $q_\ell(z)$ factors over $\mathbb{C}$ into two complex-conjugate linear forms, meaning each cone $C_\ell \subset \mathbb{C}^2$ consists of a pair of complex-conjugate lines intersecting at the origin. If two such cones $C_1$ and $C_\ell$ (for some $\ell \ge 2$) share a non-zero point $z$, they must share the complex line spanned by $z$. By conjugation symmetry, they must also share the conjugate line spanned by $\bar{z}$. Thus, the zero sets $C_1$ and $C_\ell$ coincide. Since both quadratic polynomials are square-free and homogeneous of degree two, the equality of their zero sets forces $q_1$ and $q_\ell$ to be proportional, contradicting the assumption of distinct proportionality classes. Therefore $C_1 \cap C_\ell = \{0\}$ for all $\ell \ge 2$.
     
     \item Case $d \ge 3$: Since each $\bar{\Sigma}_\ell^{-1}$ is non-degenerate, $q_1$ is an irreducible polynomial of degree two, making $C_1$ an irreducible affine variety. If $C_1 \subset \bigcup_{\ell=2}^m C_\ell$, then irreducibility forces $C_1 \subset C_\ell$ for some $\ell \ge 2$. By Study's Lemma \cite{study1923einleitung}, $q_1$ must divide $q_\ell$. Since both are homogeneous of degree two, they must be proportional, which is a contradiction. Thus, $C_1 \not\subset \bigcup_{\ell=2}^m C_\ell$.
\end{enumerate}
 Consequently, the intersection $C_1 \cap \left(\bigcup_{\ell=2}^m C_\ell\right)$ is a proper subvariety of $C_1$ of strictly lower dimension. It is therefore nowhere dense in $C_1$ under the standard Euclidean topology. Removing this intersection leaves an open and dense subset of $C_1$, from which we select a non-zero singularity $z_0 \in C_1 \setminus \bigcup_{\ell=2}^m C_\ell$ satisfying $q_1(z_0) = 0$ and $q_\ell(z_0) \neq 0$ for all $\ell = 2, \ldots, m$.

 \paragraph{Step 3: Local biholomorphism and covering space.}
To analyze the singularity locally, we construct a biholomorphism near $z_0$. Since $\bar{\Sigma}_1^{-1}$ is non-degenerate and $z_0 \neq 0$, the gradient $\nabla q_1(z_0) = 2 \bar{\Sigma}_1^{-1} z_0$ is non-vanishing. Permuting coordinates if necessary, we assume $\frac{\partial q_1}{\partial z_1}(z_0) \neq 0$. Define the local coordinate change
$$
\Upsilon: \mathbb{C}^d \to \mathbb{C}^d, \qquad \Upsilon(z_1, \dots, z_d) = (q_1(z), z_2, \ldots, z_d).
$$
One can verify that the Jacobian determinant $\det J_\Upsilon(z_0) = \frac{\partial q_1}{\partial z_1}(z_0)$ is non-zero. By the Holomorphic Inverse Function Theorem \cite[Ch 1, Sec B]{gunning1965analytic}, there exist an open neighborhood $U \subset \mathbb{C}^d$ containing $z_0$ and an open neighborhood $W \subset \mathbb{C}^d$ containing $(0, \xi_0)$—where $\xi_0 := (z_{0,2}, \dots, z_{0,d}) \in \mathbb{C}^{d-1}$—such that $\Upsilon\vert_U: U \to W$ is a biholomorphism with holomorphic inverse $\Phi := \Upsilon^{-1}: W \to U$.

Writing coordinates in $W$ as $(t, \xi) \in \mathbb{C} \times \mathbb{C}^{d-1}$, the construction yields $q_1(\Phi(t, \xi)) = t$. Shrinking \(W\) and \(U\) if necessary, we may write \(W=\mathbb D_r\times V\), where \(\mathbb D_r\subset\mathbb C\) is an open disc of radius \(r>0\) centered at \(0\), and \(V\) is a sufficiently small connected polydisc around \(\xi_0\). Since $W$ is a contractible polydisc, every nonvanishing holomorphic function on $W$ admits a holomorphic logarithm, and hence a holomorphic square root. We take $W$ small enough such that $q_\ell(\Phi(t, \xi)) \neq 0$ for all $(t, \xi) \in W$ and all $\ell = 2, \ldots, m$.

On the punctured neighborhood $W^\times = (\mathbb{D}_r \setminus \{0\}) \times V$, the local pullbacks of $N$ and $D$ take the form
\begin{align*}
N(\Phi(t, \xi)) &= n_1 t^{-(d+1)/2} + \sum_{\ell = 2}^m n_\ell q_\ell(\Phi(t, \xi))^{-(d+1)/2}, \\
D(\Phi(t, \xi)) &= \delta_1 t^{-d/2} + \sum_{\ell = 2}^m \delta_\ell q_\ell(\Phi(t, \xi))^{-d/2}.
\end{align*}
Because the fractional powers $t^{-(d+1)/2}$ and $t^{-d/2}$ are multi-valued on $W^\times$, defining them directly on $U$ would require branch cuts that disrupt the global analytic structure. We circumvent this by lifting the relation to a covering space.

Let the base space be the complex domain with the singular cones removed 
$\Omega_0 = \mathbb C^d \setminus \bigcup_{\ell = 1}^m C_\ell.$ A point $z$ lies in $\Omega_0$ if and only if $q_\ell(z) \neq 0$ for all $\ell = 1, \dots, m$. We define the covering space $\Omega$ by
\begin{align}
\Omega = \left\{ (z, \eta) \in \Omega_0 \times \mathbb C^m : \eta_\ell^2 = q_\ell(z) \text{ for } \ell = 1, \ldots, m \right\}.
\label{eq: covering-space}
\end{align}
The projection $\pi: \Omega \to \Omega_0$ defined by $\pi(z, \eta) = z$ is a $2^m$-sheeted covering map. Locally, the holomorphic implicit function theorem guarantees that each coordinate root $\eta_\ell$ is a holomorphic function of $z$ because the derivative of $\eta_\ell^2 - q_\ell(z)$ with respect to $\eta_\ell$ is $2 \eta_\ell \neq 0$ on $\Omega$. We can directly define the globally single-valued, holomorphic functions on $\Omega$:
\begin{align}
N(z, \eta) := \sum_{\ell = 1}^m n_\ell \eta_\ell^{-(d+1)}, \qquad D(z, \eta) := \sum_{\ell = 1}^m \delta_\ell \eta_\ell^{-d}. \label{eq: N-D-lifted}
\end{align}
These functions are holomorphic because each coordinate projection $\eta_\ell \mapsto \eta_\ell^{-k}$ is holomorphic on $\mathbb{C} \setminus \{0\}$. We lift the algebraic relation \eqref{eq: constancy-function} to the function
\begin{align}
F( z, \eta) = N(z, \eta)^d - KD(z, \eta)^{d+1}, \qquad (z, \eta) \in \Omega. \label{eq: constancy-lifted}
\end{align}
By construction, $F$ is a single-valued, holomorphic function on $\Omega$.
\paragraph{Step 4: Analytic continuation of $F(z, \eta)$ along a lifted path.}
For any real point $x \in \mathbb R^d \setminus \{0\}$, we lift it to $\Omega$ using positive real roots:
$$
x_\Omega = \left( x, \sqrt{q_1(x)}, \ldots, \sqrt{q_m(x)} \right) \in \Omega.
$$ 
Since $\pi$ is a local biholomorphism, there exists a small complex neighborhood $U_x \subset \Omega_0$ of $x$ and a unique local holomorphic section $\sigma_+: U_x \to \Omega$ of $\pi$ satisfying $\sigma_+(x) = x_\Omega$. Explicitly,
$$
\sigma_+(z) = \left(z, r_1(z), \ldots, r_m(z)\right),
$$
where each $r_\ell$ is the local holomorphic square root satisfying $r_\ell(x) = \sqrt{q_\ell(x)} > 0$. The composition $F\circ \sigma_+$ is holomorphic on $U_x$. Since $F \circ \sigma_+$ vanishes on the real-open set $U_x \cap \mathbb{R}^d$, and the restriction of a holomorphic function to a real slice is real-analytic, the vanishing of $F \circ \sigma_+$ on this real-open set implies it vanishes identically on the complex-open neighborhood $U_x$. Consequently, $F$ vanishes on the open set $\sigma_+(U_x) \subset \Omega$. By the Identity Theorem for connected complex manifolds \cite[Ch 1, Sec A]{gunning1965analytic}, $F$ vanishes identically on the entire connected component $\Omega_x \subset \Omega$ containing $x_\Omega$.

We now connect $x$ to a point near our chosen singularity $z_0$. For a small positive parameter $\varepsilon < r$, we set $y := \Phi( \varepsilon, \xi_0) \in \Omega_0$. To connect $x$ to $y$ within $\Omega_0$, we define the complex affine line $L(\lambda) := x + \lambda(y - x)$ for $\lambda \in \mathbb{C}$. For each $\ell$, the polynomial $q_\ell(L(\lambda))$ has degree at most two and is not identically zero since $q_\ell(L(0)) = q_\ell(x) > 0$. The union of the zeros of these $m$ polynomials is a finite set of points in $\mathbb{C}$. Choosing a continuous path from $0$ to $1$ in the $\lambda$-plane that avoids these finite zeros yields a continuous path from $x$ to $y$ contained entirely in $\Omega_0$. By the path-lifting property \cite[Proposition 1.30]{hatcher2002algebraic}, there exists a unique lifted path starting at $x_\Omega$ and ending at a point $y_\Omega = (y, \eta_1^\varepsilon, \ldots, \eta_m^\varepsilon) \in \Omega$. Since path-connected points lie in the same connected component, it follows that $y_\Omega \in \Omega_x$.

For each $\ell \ge 2$, the map $(t, \xi) \mapsto q_\ell(\Phi(t, \xi))$ is holomorphic and non-vanishing on the simply connected domain $W$, and thus admits local holomorphic square roots $\eta_\ell(t, \xi)$ on $W$ \cite[Ch 1]{gunning1965analytic}. The uniquely lifted path selects the branch of each square root function satisfying $\eta_\ell(\varepsilon, \xi_0) = \eta_\ell^\varepsilon$ for all $\ell \ge 2$.

To construct a local lifted parametrization over $W$ that respects these branch choices, let $s_0 := \eta_1^\varepsilon$ (so that $s_0^2 = q_1(y) = \varepsilon$). Because $\eta_1(t, \xi)^2 = q_1(\Phi(t, \xi)) = t$, we parameterize $t = s^2$ to render $\eta_1$ globally single-valued. This change of variables defines the punctured product neighborhood
$$
\mathcal{D}^* \times V = \{s \in \mathbb{C} : 0 < |s| < \sqrt{r}\} \times V,
$$
which forms a two-sheeted covering of the punctured domain $W^\times = (\mathbb{D}_r \setminus \{0\}) \times V$. This punctured neighborhood lifts to $\Omega$ via the holomorphic map:
$$
\Psi(s, \xi) = \left(\Phi(s^2, \xi), s, \eta_2(s^2, \xi), \ldots, \eta_m(s^2, \xi)\right).
$$
By construction, $\Psi$ is a holomorphic lift of the map $(s, \xi) \mapsto \Phi(s^2, \xi)$ under $\pi$ (satisfying $\pi \circ \Psi(s, \xi) = \Phi(s^2, \xi)$). The image $\mathcal{U} := \Psi(\mathcal{D}^* \times V)$ is a connected subset of $\Omega$ containing $y_\Omega = \Psi(s_0, \xi_0)$. Furthermore, because $s \neq 0$, the map $s \mapsto s^2$ is locally biholomorphic, the coordinate change $\Phi$ is a biholomorphism, and the chosen square-root branches define a local section of the covering. Consequently, $\Psi$ is locally biholomorphic, which directly implies that its image $\mathcal{U}$ is an open subset of $\Omega$. Since $\mathcal{U}$ contains $y_\Omega \in \Omega_x$, we have $\mathcal{U} \subset \Omega_x$, which immediately implies that the lifted algebraic relation vanishes identically on this local parametrization, i.e.,
$$
F \circ \Psi \equiv 0 \qquad \text{on } \mathcal{D}^* \times V.
$$
\paragraph{Step 5: Power series expansion and parity contradiction.}
To derive the final contradiction, we decompose $N \circ \Psi$ and $D \circ \Psi$ along this local lift:
\begin{align}
N \circ \Psi(s, \xi) &= n_1 s^{-(d+1)} + A(s^2, \xi), \label{eq: N-decomp} \\
D \circ \Psi(s, \xi) &= \delta_1 s^{-d} + B(s^2, \xi), \label{eq: D-decomp}
\end{align}
where the remainder terms are defined by
$$
A(s^2, \xi) = \sum_{\ell = 2}^m n_\ell \eta_\ell(s^2, \xi)^{-(d+1)} \quad \text{and} \quad B(s^2, \xi) = \sum_{\ell = 2}^m \delta_\ell \eta_\ell(s^2, \xi)^{-d}.
$$
Because each $\eta_\ell(t, \xi) \neq 0$ on the unpunctured product neighborhood $W = \mathbb{D}_r \times V$, the functions $A(s^2, \xi)$ and $B(s^2, \xi)$ are holomorphic on the unpunctured domain $\mathcal{D} \times V = \{s \in \mathbb{C} : |s| < \sqrt{r}\} \times V$, which includes the fiber $s = 0$.

Furthermore, these local remainder functions do not vanish identically on $\mathcal{D} \times V$. To see this, we define the global holomorphic remainder functions on $\Omega_x$:
$$
R_N(z, \eta) = \sum_{\ell = 2}^m n_\ell \eta_\ell^{-(d+1)} \quad \text{and} \quad R_D(z, \eta) = \sum_{\ell = 2}^m \delta_\ell \eta_\ell^{-d}.
$$
The functions $A$ and $B$ are exactly the restrictions of $R_N$ and $R_D$ to the local parameterization image $\mathcal{U} = \Psi(\mathcal{D}^* \times V)$. Evaluating the global remainders at our real lifted point $x_\Omega \in \Omega_x$ yields strictly positive real numbers, as all coefficients $n_\ell, \delta_\ell$ and roots $\sqrt{q_\ell(x)}$ are strictly positive. Since $R_N(x_\Omega) > 0$ and $R_D(x_\Omega) > 0$, neither global remainder vanishes identically on the connected component $\Omega_x$. By the Identity Theorem, their restrictions $A$ and $B$ cannot vanish identically on the open subset $\mathcal{U} \subset \Omega_x$.

Substituting the decompositions \eqref{eq: N-decomp} and \eqref{eq: D-decomp} into the identity $F \circ \Psi \equiv 0$ yields
\begin{align}
\left(n_1 s^{-(d+1)} + A(s^2, \xi)\right)^d = K \left(\delta_1 s^{-d} + B(s^2, \xi)\right)^{d+1} \qquad \text{on } \mathcal{D}^* \times V. \label{eq: functional-equality}
\end{align}
Factoring out the singular power $s^{-d(d+1)}$ from both sides, we obtain
\begin{align}
\left(n_1 + s^{d+1}A(s^2, \xi)\right)^d = K\left(\delta_1 + s^dB(s^2, \xi)\right)^{d+1}. \label{eq: constancy-factored}
\end{align}
Since $A$ and $B$ are holomorphic on $\mathcal{D} \times V$, they are continuous at $s=0$. Taking the limit as $s \to 0$ in \eqref{eq: constancy-factored} causes the terms $s^{d+1}A$ and $s^dB$ to vanish (since $d \ge 1$), which establishes the leading coefficient identity
\begin{align}
n_1^d = K\delta_1^{d+1}. \label{eq: leading-identity}
\end{align}
To isolate the contradiction, we choose a generic slice. Because $A(t, \xi)$ and $B(t, \xi)$ are holomorphic on $W$, we can expand them as power series in $t$ with coefficients that are holomorphic on $V$
$$
A(t, \xi) = \sum_{j=0}^\infty a_j(\xi) t^j \quad \text{and} \quad B(t, \xi) = \sum_{j=0}^\infty b_j(\xi) t^j.
$$
Let $p$ and $q$ be the smallest non-negative integers such that the coefficient functions $a_p(\xi)$ and $b_q(\xi)$ are not identically zero on $V$. Since $V \subset \mathbb{C}^{d-1}$ is a connected open domain, the zero sets of the non-trivial holomorphic functions $a_p$ and $b_q$ are proper analytic subvarieties, and are thus closed and nowhere dense in $V$. Consequently, we can select a point $\xi \in V$ such that both $a := a_p(\xi) \neq 0$ and $b := b_q(\xi) \neq 0$. Fixing this choice of $\xi$ and substituting $t = s^2$ yields the single-variable Taylor expansions
$$
A(s^2, \xi) = a s^{2p} + \mathcal{O}(s^{2p+2}) \quad \text{and} \quad B(s^2, \xi) = b s^{2q} + \mathcal{O}(s^{2q+2}).
$$
Inserting these expansions into \eqref{eq: constancy-factored} and applying the binomial theorem to both sides, we find
\begin{align}
\text{LHS} &= \left(n_1 + a s^{d+1+2p} + \mathcal{O}(s^{d+2p+3})\right)^d = n_1^d + d n_1^{d-1} a s^{d+1+2p} + \mathcal{O}(s^{d+2p+3}), \\
\text{RHS} &= K\left(\delta_1 + b s^{d+2q} + \mathcal{O}(s^{d+2q+2})\right)^{d+1} = K\delta_1^{d+1} + K(d+1)\delta_1^d b s^{d+2q} + \mathcal{O}(s^{d+2q+2}).
\end{align}
Subtracting the leading coefficient identity \eqref{eq: leading-identity} from both sides, we obtain the equality of holomorphic series in $s$:
\begin{align}
d n_1^{d-1} a s^{d+1+2p} + \mathcal{O}(s^{d+2p+3}) = K(d+1)\delta_1^d b s^{d+2q} + \mathcal{O}(s^{d+2q+2}). \label{eq: final-comparison}
\end{align}
We now analyze the orders of vanishing of both sides of \eqref{eq: final-comparison}. The leading exponents on the left- and right-hand sides are $d+1+2p$ and $d+2q$, respectively. Because $p$ and $q$ are integers, the difference between these exponents is $(d+1+2p) - (d+2q) = 1 + 2(p-q),$ which is strictly odd. Consequently, the two exponents cannot be equal, and one must be strictly smaller than the other.

Suppose $d+1+2p < d+2q$. Dividing both sides of \eqref{eq: final-comparison} by $s^{d+1+2p}$ and taking the limit as $s \to 0$, the left-hand side converges to the non-zero constant $d n_1^{d-1} a \neq 0$, whereas the right-hand side converges to $0$, a contradiction. By a similar argument, the symmetric inequality $d+2q < d+1+2p$ evaluates the right-hand side to the non-zero constant $K(d+1)\delta_1^d b \neq 0$ and the left-hand side to $0$, yielding an identical contradiction. 

Thus, the assumption that $m \ge 2$ must be false. It follows that $m=1$, meaning all covariance matrices lie in the same proportionality class. This proves the converse and completes the proof.
\end{proof}
\end{document}